\documentclass[11pt]{amsart}
\usepackage[margin=2.5cm]{geometry}

\usepackage[english]{babel}
\usepackage{amsmath,amsfonts,amssymb,commath}
\usepackage{mathrsfs}

\newdimen\parindentt
\parindentt=\parindent
\advance\parindentt by 5mm

\usepackage{geometry}
\usepackage{indentfirst}   

\usepackage[usenames,svgnames,dvipsnames]{xcolor}
\definecolor{verde}{rgb}{0,0.5,0}

\definecolor{laranja}{rgb}{0.95,0.45,0}

\definecolor{vermelho}{rgb}{0.666,0,0}

\usepackage[all]{xy}
\usepackage[utf8]{inputenc}
\usepackage{mathptmx}

\usepackage{lipsum} 

\usepackage{hyperref,cleveref}
\usepackage{enumerate,multicol}
\usepackage{tikz-cd}

\newtheorem{theorem}{Theorem}[section]
\newtheorem{proposition}[theorem]{Proposition}

\newtheorem{lemma}[theorem]{Lemma}
\newtheorem{claim}[theorem]{Claim}

\newtheorem{question}[theorem]{Question}

\theoremstyle{definition}
\newtheorem{definition}[theorem]{Definition}
\newtheorem{example}[theorem]{Example}

\theoremstyle{remark}
\newtheorem{remark}[theorem]{Remark}

\numberwithin{equation}{section}

\def\cal#1{\mathcal{#1}}
\def\bb#1{\mathbb{#1}}

\usepackage[all]{xy}

\renewcommand{\d}[1]{\ensuremath{\operatorname{d}\!{#1}}}

\tikzset{
  symbol/.style={
    draw=none,
    every to/.append style={
      edge node={node [sloped, allow upside down, auto=false]{$#1$}}}
  }
}

\usepackage[mathlines]{lineno}

\newcommand{\expo}{\mathrm{exp}}

\newcommand{\M}{\mathbb{M}}
\newcommand{\R}{\mathbb{R}}

\newcommand{\JH}{J}
\newcommand{\FH}{\widehat{\mathcal{F}}}
\newcommand{\LH}{\widehat{L}}

\newcommand{\metric}{\mathbf{g} }

\newcommand{\linearizedbasic}{C^{\infty}(\{U_i\})^{\ell}}

\newcommand{\closedlinearizedbasic}{W^{1,p}(\{U_i\})^{\ell}}

\newcommand{\closedlinearizedbasictwo}{W^{1,2}(\{U_i\})^{\ell}}

\newcommand{\paralleltransport}{\mathcal{P}}

\newcommand{\weight}{\mathcal{M}}

\newcommand{\vizinhancaB}{P}

\newcommand{\formacurvaturamedia}{\mathsf{k}}

\newcommand{\vol}{\mathrm{\, \, vol}}

\newcommand{\Mprincipal}{M^{0}}

\newcommand{\density}{\, \upsilon}

\newcommand{\efunction}{\mathsf{e}}

\newcommand{\ufunction}{u}

\newcommand{\wfunction}{w}

\newcommand{\lfunctional}{\mathtt{l}}

\newcommand{\polarsection}{\Sigma}

\newcommand{\basicspace}{\mathcal{B}}

\newcommand{\weylgroup}{\mathrm{W}}

\title[SRF, variational problems and principle of symmetric criticalities]{Singular Riemannian Foliations, variational problems and  Principles of Symmetric Criticalities}

\author[Alexandrino]{Marcos M. Alexandrino}
\address{Instituto de Matemática e Estatística da Universidade de São Paulo (IME-USP), \\ R. do Mat\~{a}o, 1010 - Butant\~{a}, S\~{a}o Paulo - SP, 05508-090.}
\email{malex@ime.usp.br}

\author[Cavenaghi]{Leonardo F. Cavenaghi}
\address{Instituto de Matemática, Estatística e Computação Científica (IMECC-Unicamp) \\
da Universidade Estadual de Campinas , Cidade Universitária, Campinas - SP, 13083-856}
\email{leonardofcavenaghi@gmail.com}

\author[Corro]{Diego Corro}
\address{School of Mathematics, Cardiff University,\\  Abacws Building Senghennydd Road, Cathays, Cardiff, CF24 4AG.}
\email{diego.corro.math@gmail.com}

\author[Inagaki]{Marcelo K.  Inagaki}
\address{Instituto de Matemática e Estatística da Universidade de São Paulo (IME-USP), \\ R. do Mat\~{a}o, 1010 - Butant\~{a}, S\~{a}o Paulo - SP, 05508-090.}
\email{kodi.inagaki@gmail.com}


\begin{document}

\begin{abstract}
A singular foliation $\mathcal{F}$ on a complete Riemannian manifold $M$ is  called Singular Riemannian foliation (SRF for short) if its leaves are locally equidistant, e.g., the partition of $M$  into the orbits of a Lie group action by isometries. 
In this paper, we investigate variational problems in compact Riemannian manifolds equipped with SRFs with special properties,
 which we name as AVP. Examples of such SRFs being considered include isoparametric foliations,  
SRFs on Euclidean fiber bundles, and 
the partition of $M$  into the orbits of a Lie group acting by isometries. More precisely, we prove an analog  to  Palais' Principle of Symmetric Criticality
for  $\mathcal{F}$-symmetric integral  operators on the  Banach spaces $W^{1,p}(M)$. This  result 
together with a  version of the Rellich–-Kondrachov--Hebey--Vaugon Embedding Theorem for $\mathcal{F}$-basic Sobolev functions allows us 
to circumvent difficulties with Sobolev's critical exponents when  considering applications of techniques from Calculus of Variations to find solutions to PDEs.  
To exemplify this, we prove the existence of countably infinite many weak solutions to a class of variational problems, which includes $p$-Kirchhoff
 problems for manifolds equipped with AVP.
\end{abstract}

\maketitle
\tableofcontents


\section{Introduction}
\label{Section-Introduction}

 It is known   
nowadays (see, e.g., \cite{hebey1,hebey2, Hebey_2000}, \textbf{\cite{kazdaninventiones,kazadanannals,kazdan1975}, \cite{cavenaghi2021}}) 
that geometric analytic problems modeled on a  Riemannian manifold equipped with 
an isometric action of a Lie group are easier to deal with due to 
the Palais'  Principle of Symmetric Criticality  \cite{palais1979}, among other things. For the sake of motivation,  
let us  briefly  review a version of this principle.

Let $G$ be a (connected) closed subgroup of the group of   isometries of a compact Riemannian manifold $(M,\metric)$,
$\mu\colon G\times M\to M$ the induced isometric action (i.e., $\mu(g,x)=g(x)$) 
 and $\mathcal{C}=\{\vec{X}\}$ the set of Killing vector fields induced by the $\mu$ action, i.e., 
 for $X\in\mathfrak{g}$, where $\mathfrak{g}$ is the Lie algebra of $G$, 
we consider the induced vector field $\vec{X}(\cdot)=\frac{d}{d t} \big( \mu(\efunction^{t X},\cdot) \big)\Big|_{t=0}$.
Following the conventions in Control Theory \cite{Agrachev}, let  $\efunction^{tX}=\varphi_t$ denote the associated flow induced by  vector field $\vec{X}.$
Let $\basicspace:=W^{1,p}(M)^{G}$ be the closure (with respect to the Sobolev norm) of the  vector space of smooth basic functions $b$.
Recall that $b$ is a \emph{basic function} if 
$b\circ \mu(g,\cdot)=b(\cdot)$, for all $g\in G$.  
In other words, if $b$ is constant along each orbit $L_x=G(x)=\{\mu(g,x)\in M, \forall g\in G \}$. 
We say that a $C^{1}$   funcional $J\colon W^{1,p}(M)\to \mathbb{R}$ is \emph{critical symmetric} at $b_0\in \basicspace$ if:
\begin{equation}
\label{equation2-motivation}
\d J(b_0)f\circ \efunction^{t X} =\d J(b_0) f \mathrm{, \, \,  \, \, \, }    \forall f\in C^{\infty}(M).
\end{equation}
For example, consider a  $C^{1}$ functional $J\colon W^{1,p}(M)\to \mathbb{R}$ whose restriction to the subspace  $C^{\infty}(M)$ 
(here denoted by $J|_{C^{\infty}(M)}$) is an integral operator as follows: 
\begin{equation}
\label{equation1-motivation}
J|_{C^{\infty}(M)}(f)=\int_{M} \mathcal{L}(|\nabla f|^{2},f,x) \density_{\metric}, 
\end{equation}
where $ \density_{\metric}$ is the  \emph{Riemannian density} (i.e., the measure induced by the Riemannian metric $\metric$) and 
$\mathcal{L}\colon \mathbb{R}\times\mathbb{R}\times M\to \mathbb{R}$ is a smooth \emph{basic Lagrangian}, i.e.,
$\mathcal{L}\big(s_1,s_2,x\big)= \mathcal{L}\big(s_1,s_2, y\big)$, for $s_1,s_2\in \mathbb{R}$ and $y\in L_{x}, \forall x\in M$. 
The \emph{Principle of Symmetric Criticality of Palais} assures that for $b_{0}\in \basicspace$  
\begin{equation}
\label{equation3-motivation}
 \d J(b_0)|_{\basicspace}=0 \implies \d J(b_0)=0.
\end{equation}
In other words, the principle allows us to look for critical points of $J$ on $\basicspace$, where there exist 
 better compactness embeddings of Sobolev spaces into Lebesgue spaces \cite{hebey3}.

Our main goal in this paper is to generalize the  Principle of Symmetric Criticality for the classes of 
Riemannian manifolds $(M,\metric)$ with singular Riemannian foliations, where several  analytical  geometric problems 
have been recently approached, see e.g., \cite{corro2022}, \cite{Liu-Radeschi,ACG20,Alex-Radeschi-MCF}.

We recall that  a singular foliation $\mathcal{F}=\{L_{x}\}_{x\in M}$ (with embedded leaves) 
on a complete Riemannian manifold $(M,\metric)$ is called a \emph{singular Riemannian foliation} (SRF for short) if its leaves are locally equidistant, 
or equivalently 
if each geodesic that is perpendicular to a leaf 
remains perpendicular to every leaf it meets, see Definition \ref{MALEX-definition-SRF}. 
In addition to the examples of partitions of $M$ into the  orbits of  Lie group actions by isometries (known as \emph{homogeneous SRFs}), there are infinitely many other (non-homogenous) examples of SRFs,
such as   polar foliations  (and, in particular, isoparametric foliations),
and  partitions  of  a compact Riemannian manifold $M$ into  orbits of proper groupoids on compact manifolds; see
\cite{Terng-Thorbergsson-95,Thorbergsson-survey-10,Berndt-Console-Olmos-2016,alexandrino-Bettiol-2015,Radeschi-14, Gromoll-Walschap09,Hoyo-Fernandes-18}.

Although a SRF is not necessarily homogeneous,  
it always carries a set of vector fields tangent to the leaves, which share some common properties with Killing vector fields.
Therefore, we will frequently be considering, throughout this article, 
a \emph{geometric control system} on some (tubular) 
neighborhood  $U$ , i.e.,  a set of vector fields whose domain cover $U$, see \cite{Agrachev}.  
More precisely, we are going to consider a set of  linearized vector fields  tangents to the leaves of the SRF.

\begin{definition}[Linearized vector fields]
\label{definition-linearized-vector-field}
Given a non-trivial (possibly singular) leaf $L_{q}$ of a SRF 
 with embedded leaves 
 we can always find a tubular neighborhood $U=\mathrm{Tub}_{\delta}(L_q)$ of 
the leaf $L_{q}$ and a geometric control system
 $\mathcal{C}(U):=\{ \vec{X}_{\alpha} \}\subset \mathfrak{X}(U)$  
of vector fields tangent to the leaves (the so-called \emph{linearized vector fields with respect to $L_q$} ) such that:
\begin{enumerate}
\item[(a)] $\vec{X}_{\alpha}$ are   
 invariant by the homothetic transformation $h_{\lambda}(\exp^{\nu} w)=\exp^{\nu}(\lambda w)$, where $\exp^{\nu}$ is the normal exponential map; 
\item[(b)] the orbits of  $\{ \vec{X}_{\alpha} \}$  are leaves of a  singular (sub)foliation (called the
\emph{linearized foliation})
$ \mathcal{F}^{\ell}=\{L^{\ell}_{x} \}_{x\in U}\subset \mathcal{F}_{U}$, cf., \cite[Theorem 5.1]{Agrachev}. 
The leaves of $ \mathcal{F}^{\ell}$ are also 
  orbits  from a Lie  groupoid $\mathcal{G}^{\ell}$ and such that for the central leaf we have $L_{q}= L_{q}^{\ell}$, 
	 see \cite{alexandrino2021lie}.
\end{enumerate}
\end{definition}

A SRF is called \emph{orbit-like foliation} (see Definition \ref{definition-orbit-like})
if for each leaf $L_q$ there exists  a small tubular neighborhood 
$U=\mathrm{Tub}_{\delta}(L_q)$ such  that the restricted foliation $\mathcal{F}_{U}$ coincides with the linearized foliation $\mathcal{F}^{\ell}$. 
 In other words, if  $\mathcal{F}_{U}=\mathcal{F}^{\ell}$.  

Now that we have  established the definition of SRF (which generalizes the concept of partition into orbits of  isometric action),  
we  discuss what it means for a variational problem to have symmetry in the context of SRFs.

   The principle of symmetric criticality requires that we consider an isometric action $\mu\colon G\times M\to M$ and not just the partition of orbits 
	$\mathcal{F}=\{G(x)\}_{x\in M}$.  
 It could happen that for another non-isometric action 
$\tilde{\mu}\colon G \times M \to M$,
 but with the same orbits,
 the functional would not be symmetric with respect to $\tilde{\mu}$. 	
This could be  an issue of little relevance,
after all, geometric objects in Riemannian Geometry (such as volume, gradient, etc.) 
are invariant under isometries, and so  often  are the functionals defined in terms of these objects.
 So why bother considering a non-isometric action but orbit equivalent to an isometric one? 
However, when we consider SRF, we have a more complicated issue. What would an \emph{isometric representation of a foliation} be?  
In particular, even if we restrict our attention to an orbit-like foliation, 
the technical problem  is that neither the so-called linearized foliations nor 
their geometric control system is defined over the entire manifold $M$, 
only on tubular neighborhoods of leaves. This forces us to have a \emph{semi-local definition}
(i.e.,  definition of mathematical objects defined  on tubular neighborhoods of leaves) as we discuss below.

\begin{definition}
\label{definition-F-critical-symmetric}
Given a SRF   $\mathcal{F}=\{L_{x}\}_{x\in M}$ on a compact Riemannian manifold $(M,\metric)$, 
a $C^{1}$ operator $J\colon W^{1,p}(M)\to\mathbb{R}$ is called $\mathcal{F}$-\emph{critical symmetric} if:
\begin{enumerate}
\item[(a)] for each leaf $L_q$, we can find a small tubular neighborhood
$U=\mathrm{Tub}_{\delta}(L_q)$; 
\item[(b)] there exists some  geometric control system $\mathcal{C}(U)=\{\vec{X}_{\alpha}\}\subset \mathfrak{X}(U)$ 
of linearized vector fields (with respect to $L_q$) that generates $\mathcal{F}^{\ell}\subset \mathcal{F}_{U} $ so that,
\[
\mathrm{d} J(b) (\ufunction \circ \efunction^{t X_{\alpha}}) = \mathrm{d} J(b)\ufunction,\]

  $ \forall \ufunction\in C_{c}^{\infty}(U)$, $\forall \vec{X}_{\alpha}\in \mathcal{C}(U)$, and  $\forall  b\in W^{1,p}(M)^{\mathcal{F}}$.
\end{enumerate}
Here $W^{1,p}(M)^{\mathcal{F}}$ denotes the closure of the vector space of smooth basic functions 
(i.e., smooth functions that are constant along the leaves of 
$\mathcal{F}$) with respect to the Sobolev's norm of $W^{1,p}(M)$.   
\end{definition}

Definition \ref{definition-F-critical-symmetric} 
 motivate us to consider the following question.

\begin{question}
\label{question1}
 Can we assure that a given $C^{1}$ operator $J\colon W^{1, p}(M)\to \mathbb{R}$ 
(that includes the natural  operator defined in Equation \eqref{equation1-motivation}) is $\mathcal{F}$-critical symmetric, at least
 for some class of SRF (that should include relevant examples, such as isoparametric foliations)? After all, condition (b) in 
Definition \ref{definition-F-critical-symmetric}  
may sound a bit vague and difficult to verify in concrete examples.
\end{question}

To start to answer Question \ref{question1}, let us look for the class of  SRF  suitable for problems with symmetries.  
We propose as a strategy to examine the operator  $J$ defined 
in Equation \eqref{equation1-motivation} and ask ourselves what properties the  flows $\efunction^{t X}$
must meet for Equation \eqref{equation2-motivation}   to be valid.
It's not hard to see that it is natural to expect the flows to preserve the volume of $M$ (or, more generally, density, if $M$ is not orientable). 
When we explicitly calculate the derivative of $J$ defined in Equation \eqref{equation1-motivation} 
we see how useful it is to assume that the flows $\efunction^{t X}$ take normal spaces of the regular leaves to normal spaces of the regular leaves.
Fortunately, these requirements are easily met in the first two natural examples of foliations  we consider:
\begin{enumerate}
\item partition of $M$ into orbits of an isometric action;
\item SRFs in Euclidian bundles $\mathbb{R}^{k}\to E\to B$ with  Sasaki metrics $\metric^{\tau}$ (see Proposition
\ref{MALEX-flows-preserve-volume}). In particular, these examples  model SRFs near  leaves  
and approximate  the geometry of foliation, cf., \cite{alexandrino2021lie,ACG20}.
\end{enumerate}

We can therefore establish the following new definition. 

\begin{definition}[AVP]
\label{definition-a-v-p}
We say that a SRF $\mathcal{F}=\{ L_{x}\}_{x\in M}$ 
with closed leaves on a  complete Riemannian  manifold $(M,\metric)$ 
is \emph{adapted to  variational problems} (AVP for short) if 
for each leaf $L_{q_0}$ there exists a tubular neighborhood $U=\mathrm{Tub}_{\delta}(L_{q_0})$ and a 
(geometric control) system 
$\mathcal{C}(U)=\{\vec{X}_{\alpha}\}\subset\mathfrak{X}(U)$ (namely  \emph{an associated AVP system})
of complete linearized vector fields (concerning $L_{q_0}$) 
whose orbits coincide with $\mathcal{F}^{\ell}$  and their flows  $\efunction^{t X_{\alpha}}$ preserve:
 \begin{enumerate}
\item[(a)] the normal distribution $\nu(\mathcal{F})$ (with respect to original metric $\metric$) of the regular leaves;
\item[(b)] the Riemannian density $\density_{\metric}$,  
 i.e, the measure induced by the Riemannian metric $\metric$.  
\end{enumerate}
\end{definition}

The next result ensures  that isoparametric foliations (a relevant class of SRF, see Section \ref{Section-Isoparametric-are-AVP}) 
are also AVP foliations.

Recall that  a SRF $\mathcal{F}$ is called \emph{isoparametric} if
\begin{enumerate}
\item the mean curvature vector field $\vec{H}$ is basic, when restricted to the regular stratum  $M^{0}$ (i.e,
the open subset that is the union of leaves with maximal dimension);
\item the normal distribution $\nu(\mathcal{F})$ of the foliation $\mathcal{F}$ restricted to $M^{0}$ is integrable, 
see Definition~\ref{definition-polar-isoparametric}. 
\end{enumerate}

\begin{theorem}
\label{theorem-isoparemetric-is-AVP}
Let $\mathcal{F}$ be an isoparametric foliation on a compact Riemannian manifold $(M,\metric)$. 
 Then $\mathcal{F}$ is an
 AVP foliation, i.e.,  adapted to variational problems.  
 \end{theorem}

This result shows already that the class of AVP foliations is quite rich. The next result provides an answer to Question~\ref{question1}.

\begin{theorem}
\label{proposition-operator-J-kirschoff-generalized}
Let $\mathcal{F}$ be an AVP on a compact Riemannian manifold $(M,\metric)$. 
Consider a   $C^{1}$  operator  $J_{\lambda}\colon W^{1,p}(M)\to\mathbb{R}$ that when  restricted 
to the  subspace $C^{\infty}(M)$ is 
\begin{equation}
\label{equation-proposition-J-lambda}
J_{\lambda}(u):= \weight\Big(\int_{M} | \nabla u |^{p}\density_{\metric}\Big)\int_M\cal L(|\nabla u|^2,u,x)\density_{\metric} 
- \frac{\lambda}{c}\left(\int_MF(u,x)\density_{\metric}\right)^{r+1},
\end{equation}
where 
\begin{itemize}
\item $r, c, \lambda>0$ are positive constants   and $p\in [2,n[$; 
\item  $F\colon \mathbb{R}\times M\to\mathbb{R}$  and  $\mathcal{L}\colon \mathbb{R}\times\mathbb{R}\times M\to \mathbb{R}$ are   $C^{1}$ and  
$\mathcal{F}$-basic, i.e., $F(s,y)=F(s,x)$ and  $\mathcal{L}(s_1,s_2,y)=\mathcal{L}(s_1,s_2,x)$ $\forall s_1,s_2,s\in\mathbb{R}$,
  and $\forall y\in L_{x}$,  for each fixed $x\in M$;
\item $\weight\colon \mathbb{R}\to\mathbb{R}$ is $C^{1}$. 
\end{itemize}
 Then $J_{\lambda}$ is $\mathcal{F}$-critical symmetric. 
\end{theorem}
We remark that the operator in Theorem~\ref{proposition-operator-J-kirschoff-generalized} above    
describes the energy operator of the general $p$-Kirchhoff's equation,
see Section~\ref{section-An abstract setup for variational problems with (and via) symmetries} for more details.

Differentiating  the operator $J$  presented in Equation \eqref{equation-proposition-J-lambda} (restricted to $C^{\infty}(M)$), we have  for all $b\in C^{\infty}(M)^{\mathcal{F}}$, 
  an integral operator on  $C^{\infty}(M)$ of a very particular form:
\[
\mathrm{d} J_{\lambda}(b)\wfunction=\int_{M}\big( \mathcal{L}_{1}^{\ell}(b)\, \metric (\nabla b, \nabla \wfunction) +  
\mathcal{L}_{2}^{\ell}(b) \wfunction \big) \density_{\metric}, \, \, \forall \wfunction\in C^{\infty}(M).
\]
The next result assures that the principle of  symmetric criticality   holds for  all operator $J$
whose derivatives have the same as $\mathrm{d} J_{\lambda}$ above.

\begin{theorem}
\label{theorem-simple-version-Palais-principle-variational-formulation}
Let $\mathcal{F}$ be an orbit-like,   AVP foliation on a compact (connected) Riemannian manifold $(M,\metric)$, 
and $J\colon W^{1, p}(M) \to \mathbb{R}$ a $C^{1}$ operator such that for each smooth basic function 
$b_{0}\in C^{\infty}(M)^{\mathcal{F}}$ we have
\[
\mathrm{d} J(b_{0})\wfunction=\int_{M}\big( \mathcal{L}_{1}^{\ell}(b_{0})\, \metric (\nabla b_{0}, \nabla \wfunction) +  
\mathcal{L}_{2}^{\ell}(b_{0}) \wfunction \big) \density_{\metric}, \, \, \forall \wfunction\in C^{\infty}(M),
\]
where $ \mathcal{L}_{i}^{\ell}(b_{0})$ is a continuous $\mathcal{F}$-basic function, 
 i.e.,  $ \mathcal{L}_{i}^{\ell}\colon C^{\infty}(M)^{\mathcal{F}}\to C^{0}(M)^{\mathcal{F}}$.  
Assume that  for a fixed $b_{0}\in  W^{1,p}(M)^{\mathcal{F}}$ 
we have $ \mathrm{d}J(b_0)(b)=0$ for all $b\in   W^{1,p}(M)^{\mathcal{F}}$.
Then  $\mathrm{d}J(b_0)=0$. 
\end{theorem}
\begin{remark} 
We  stress that to increase the number of examples where this principle can be applied, we 
show that by dealing with the slightly more technical concept of linearized basic functions (see Definition \ref{definition-linearized-basic}) 
we may drop the orbit-like condition on the foliation.
\end{remark}

Roughly speaking, the above theorem guarantees that we can reduce a class of variational problems on  $W^{1,p}(M)$ 
to variational problems on the \emph{basic Sobolev space}  $W^{1,p}(M)^{\mathcal{F}}$.  
A  direct adaptation of Hebey and Vaugon's Theorem \cite{Hebey_2000} for isometric actions (see Theorem \ref{thm:reillich})
 guarantees that we have better Sobolev embeddings of  $W^{1,p}(M)^{\mathcal{F}}$ into $L^q$-spaces, allowing us to avoid critical exponents.
These improved Sobolev embeddings, together with Theorem \ref{theorem-Palais-principle-variational-formulation}
 (which in  turn is a generalization of Theorem \ref{theorem-simple-version-Palais-principle-variational-formulation}), 
and a few  classic arguments of the Calculus of Variations
 imply the existence of infinitely many weak foliated 
solutions of the following PDE, which is related to  $p$-Kirchhoff problems (see \cite{martinez,ALVES200585}):
\[
    -\mathrm{m}\left(1/p\int_M|\nabla u|^p\density_{\metric}\right)\Delta_pu 
		= \lambda |u|^{p^*-2}u\left[1/p^*\int_M|u|^{p^*}\density_{\metric}\right]^r.
    \]
More precisely, by setting $\weight (t):= \int_0^t \mathrm{m}(s)\mathrm{d}s$, where  
$\mathrm{m}\colon\mathbb{R}\to\mathbb{R}$ is a non-negative continuous function, 
 we have the following result, which in this format is new (even for isometric actions), as far as we know.
\begin{proposition}\label{corollary:kirshofinho}
Let $(M^n,\metric)$ be a $n$-dimensional ($n\geq 3)$  compact Riemannian manifold with an  AVP foliation $\cal F$.
Assume that the leaves of  $\mathcal{F}$ are closed, and each leaf has positive dimension.
If $n> p\geq 2$ and  $\weight \colon \bb R_+ \rightarrow \bb R_+$ is a non-negative convex $C^1$ function such that
$\lim_{t\rightarrow \infty}\weight(t)= +\infty$, then 
 there is  a sequence $\{\lambda_{i}\}_{i=1}^{\infty} \subset \bb R$ and 
a sequence of distinct non-zero functions
  $\{u_{i}\}_{i=1}^{\infty} \subset W^{1,p}(M)$ 
such that $\mathrm{d}J_{\lambda_{i}}(u_{i}) = 0 $, where
\[
J_{\lambda}(u) = \weight\left(\frac{1}{p}\int_M|\nabla u|^p\density_{\metric}\right)
 - \frac{\lambda}{r+1}\left(\int_M \frac{1}{p^*}|u|^{p^*}\density_{\metric}\right)^{r+1}, r >1.
 \]
 In particular, if $\mathcal{F}$ is also orbit-like then
the infinitely many weak solutions may be chosen to be basic, i.e., 
 $ \{u_{i}\}_{i=1}^{\infty} \subset W^{1,p}(M)^{\mathcal{F}}$. 
\end{proposition}

\begin{remark}
The ideas behind the proof of Proposition \ref{corollary:kirshofinho} 
can be generalized,  
allowing  us to prove a sequence of  critical points of the operator $J_{\lambda}$
defined in Theorem \ref{proposition-operator-J-kirschoff-generalized}, 
once the operator satisfies some technical conditions, see Theorem \ref{thm:prettygeneralzinho}.
We emphasize that we only wish to illustrate that our symmetry principle can, in fact,  be applied to the theory of PDEs. 
One of the difficulties  is checking whether technical hypotheses can be met.
 For example, to demonstrate that an operator $J$ associated with a given PDE is of class $C^{1}$, 
experts often create estimates about the growth of  derivatives of  Lagrangian functions. 
Such procedures require very specific characteristics of Lagrangians. 
Our goal is  to stress
how our principle of  symmetric criticality allows one to reduce the search for critical points of certain classes of variational problems 
to the space of basic functions (where there exists better compactness embeddings of Sobolev spaces in Lebesgue spaces), 
thus allowing experts to apply more classical arguments in these spaces.
 \end{remark}

\subsection*{This paper is organized as follows:}

In Section \ref{MALEX-subsection-Linear-afewfacts} we review several facts about singular
Riemannian foliations, as well as their relations with 
Sasaki metrics on fiber bundles. 
In Section \ref{Section-The principle of symmetric criticallity}, we prove 
the  principle of symmetric criticality, i.e.,  
Theorem \ref{theorem-Palais-principle-variational-formulation} (assuming Proposition  \ref{proposition-new-MALEX--Av-l-orbitlike})
that directly implies  
Theorem \ref{theorem-simple-version-Palais-principle-variational-formulation}. 
We also present Theorem \ref{theorem-Diego-general-version-Palais-Hilbert}
whose proof is  alternative  to that of Theorem 
 \ref{theorem-Palais-principle-variational-formulation} for the particular case of  $W^{1,2}(M)$.   
In Section \ref{Section- Symmetric Criticality of Palais on manifold}, we discuss  different aspects of an AVP foliation. 
We start by stressing  why this class of foliation is natural  to consider for variational problems by proving  Theorem \ref{proposition-operator-J-kirschoff-generalized}.
We  briefly discuss  symmetric  linear operators and Sasaki metrics (see Subsections \ref{subsection-Remark-definition-AVP} and 
\ref{subsection-few properties-Section- Symmetric Criticality of Palais on manifold}).  
We also present the average operator $\mathrm{Av}\colon C_{c}^{\infty}(U)\to C_{c}^{\infty}(U)^{\ell}$  
(associated to   the \emph{Linear holonomy groupoid}   $\mathcal{G}^{\ell}$).   
   Moreover, we stress its relations with symmetric linear functionals 
	(see Proposition  \ref{proposition-new-MALEX--Av-l-orbitlike} and 
	Remark \ref{remark-why-integration-groupoid}). In other words, we discuss the    
results used to prove the generalization of Theorem \ref{theorem-simple-version-Palais-principle-variational-formulation}.
In Section \ref{section-exa-AVP} 
 we present more examples of AVP foliations and, in particular, prove that isoparametric foliations are AVP (see Theorem \ref{theorem-isoparemetric-is-AVP}).
In Section \ref{Section-Rellich-Kondrachov}, we review a version of 
a basic Rellich-Kondrachov-Hebey-Vaugon Theorem (see Theorem \ref{thm:reillich}). 
In  Section \ref{section-An abstract setup for variational problems with (and via) symmetries} we apply
Theorem \ref{theorem-Palais-principle-variational-formulation} (which is a generalization of 
 Theorem \ref{theorem-simple-version-Palais-principle-variational-formulation}) 
and Theorem \ref{thm:reillich}) to prove Propositions \ref{corollary:kirshofinho} and Theorem \ref{thm:prettygeneralzinho}.
In Section \ref{section-Metric-foliation-Hilbert-manifolds} we briefly discuss 
metric partitions on Hilbert manifolds and prove Theorem \ref{theorem-CorrosTheorem-A}.
Finally in Section  
\ref{section-appendix} we present  an appendix
with a few facts on the linear holonomy Lie groupoid $\mathcal{G}^{\ell}$.

Due to  the background of the authors, 
this manuscript is written in the context of Differential Geometry, but has strong connections to other areas, such as Geometric Analysis and the Calculus of Variations. We suggest readers primarily
interested in PDE, first read 
Sections \ref{MALEX-subsection-Linear-afewfacts}
 \ref{Section-The principle of symmetric criticallity}, \ref{Section-Rellich-Kondrachov} and
\ref{section-An abstract setup for variational problems with (and via) symmetries}. 
 Readers primarily interested in Riemannian Geometry may be interested in reading first
Section \ref{MALEX-subsection-Linear-afewfacts} 
\ref{Section-The principle of symmetric criticallity},
\ref{Section- Symmetric Criticality of Palais on manifold}, Subsection  \ref{Section-Isoparametric-are-AVP}, 
  Subsection  \ref{Section-Applications-PP} and  Section \ref{section-Metric-foliation-Hilbert-manifolds}.
 We also note that 
 Section \ref{Section- Symmetric Criticality of Palais on manifold}, Subsection  \ref{Section-Isoparametric-are-AVP},
 Section \ref{Section-Rellich-Kondrachov}, 
Section \ref{section-An abstract setup for variational problems with (and via) symmetries} and 
Section \ref{section-Metric-foliation-Hilbert-manifolds} 
 are almost independent from one another, 
which in principle allows the paper to be read in different orders.

\subsection*{Acknowledgments}
The authors thank Jo\~ao Marcos do \'{O},  Llohann Speran\c{c}a,  Gustavo P. Ramos and Mateus M. Melo for fruitful conversations and Alexander Lytchak for some helpful preliminary questions. D. Corro also thanks Juan Carlos Fernández for productive discussions.
We also thank the anonymous reviewer for his (her) various suggestions, 
which significantly contributed to  improving 
article's presentation.

\begin{itemize}
 \item M. M. Alexandrino was supported by grants  \#22/16097-2  and  \#2016/23746-6, S\~{a}o Paulo Research Foundation (FAPESP). 
\item L. F. Cavenaghi was partly supported by the SNSF-Project 200020E\_193062 and the DFG-Priority Program SPP 2026 and by the S\~{a}o Paulo Research Foundation grant \#2022/09603-9.
\item D. Corro was supported by  a uKRI Future Leaders Fellowship [grant number MR/W01176X/1; PI: J Harvey], by the DFG (281869850, RTG 2229 ``Asymptotic Invariants and Limits of Groups and Spaces''), a DGAPA postdoctoral Scholarship of the Institute of Mathematics of UNAM, and DFG-Eigene\-stelle Fellowship CO 2359/1-1.
\end{itemize}

\subsection*{Data Availability Statement} Data sharing not applicable to this article as no datasets were generated or analyzed during the current study.

\section{Review: a few facts about singular Riemannian foliations}
\label{MALEX-subsection-Linear-afewfacts}

In this  section, we fix a few notations, definitions, and     
review several facts on singular Riemannian foliations $\mathcal{F}$, 
most of which can be found in \cite{alexandrino_radeschi_2017,alexandrino2021lie,Molino,Radeschi-notes,alexandrino-Bettiol-2015}. 

We also stress that, throughout this paper, the leaves of $\mathcal{F}$ are closed.

\vspace{0.5\baselineskip}

\subsection{Singular Riemannian foliations}
\label{MALEX-Section-SRF}

\begin{definition}[SRF]
\label{MALEX-definition-SRF}
A \emph{singular Riemannian foliation} on a complete Riemannian manifold $(M,\metric)$ is a partition $\mathcal{F}=\{ L_{x}\}_{x\in M}$ of $M$ into immersed submanifolds without self-intersections (the \emph{leaves}) that satisfies the following properties:
\begin{enumerate}
\item[(a)] $\mathcal{F}$ is a \emph{singular foliation}, i.e. for each $X_p$ tangent to $L_p$ (i.e. the leaf through $p\in M$)  there exists a local vector field $\vec{X}$ so that $\vec{X}(p)=X_p$ and $\vec{X}$ is tangent to the leaves; 

\item[(b)] $\mathcal{F}$ is \emph{Riemannian}, i.e, each geodesic $\gamma$ that starts orthogonal to a leaf $L_{\gamma(0)}$ remains orthogonal to all leaves that it meets.
\end{enumerate}
\end{definition}

\begin{remark}
\label{remark-equivalent-definition-SRF}
Item (a) is equivalent to saying that given a point $q_{0}\in M$, there exists a neighborhood $U$ of $q_{0}$ in $M$, 
a simple foliation $\mathcal{P} = \{P\}$ on $U$ (i.e. given by fibers of submersion on $U$) so that 
 $\mathcal{P}$ is a sub-foliation of $\mathcal{F}_U:=\mathcal{F}|_{U}$ (i.e. for each $x \in U$ we have $P_{x}\subset L_{x}$) and
 the leaf  $P_{q_0} \in \mathcal{P}$ (the \emph{plaque} through $q_0$) is a relatively compact open set of the leaf $L_{q_0}$.  
In particular item (a) implies that $\mathcal{F}\cap S_{q_0}$ is a singular foliation for each transverse submanifold $S_{q_0}$ 
(i.e. $T_{q_0} M = T_{q_0} S_{q_0} \oplus T_{q_0} L_{q_0}$). 
Roughly speaking, item (b) says that the leaves are \emph{locally equidistant}. 
In other words, item (b) is equivalent to saying that there exists $\epsilon>0$ so that if $x\in \partial \mathrm{Tub}_{\epsilon} (P_{q_0})$ 
(the cylinder of radius $\epsilon$ of the plaque $P_{q_0}$)  then the connected component of 
$L_x \cap U$ containing $x$ is contained in $\partial \mathrm{Tub}_{\epsilon} (P_{q_0})$.
\end{remark}

Several properties of SRF are natural generalizations of classical properties
of homogeneous foliations, see e.g., \cite[Chapter 3]{alexandrino-Bettiol-2015}.  Let us review a few of them.

The first one is the generalization of the so-called \emph{slice representation}. 

Let $U=\mathrm{Tub}_{\delta}(L_{q})$ be a (normal) tubular neighborhood around a leaf $L_q$ (where $\delta$ is small enough, i.e., smaller than the focal radius), $\pi\colon U\to L_{q}$ be the metric projection, and  $S_{q}=\pi^{-1}(q)$ be the (normal) \emph{slice}, i.e., 
$S_{q}=\exp_{q}\big(\nu_{q}(L_q)\cap B_{\delta}(0) \big)$
where  $\nu_{q}(L_q)$ is the normal space. Then the \emph{infinitesimal foliation} $\mathcal{F}_{q}=\exp_{q}^{-1}\big(S_{q}\cap \mathcal{F}\big)$ 
turns out  to be a SRF  on a neighborhood of the origin of  the  Euclidean space 
$(\nu_{q}(L_{q}),\metric_q)$. The infinitesimal foliation  $\mathcal{F}_{q}$  
can be extended  via the homothetic transformation $h^{0}_{\lambda}(v)=\lambda v$ to a SRF on $(\nu_{q}(L_q), \metric_q)$. The foliation $\mathcal{F}_{q}$ plays a role in the theory of SRF, similar to the role played by the slice representation in the theory of isometric actions. 

Another general property of SRF that is analogous to the theory of isometric action is that \emph{the partition of $M$ into the leaves of $\mathcal{F}$ with the same dimension is a stratification.} Recall that a \emph{stratification} of $M$ is a partition of $M$ into embedded submanifolds $\{M_{i} \}_{i\in I}$ (called strata) such that: 
\begin{enumerate}
\item[(i)] the partition is locally finite, i.e., each compact subset of $M$ only intersects a finite number of strata;

\item[(ii)] for each $i\in I$, there exists a subset $I_{i}\subset I/\{i\}$ such that the closure of $M_i$ is $\overline{M}_{i}= M_{i}\cup \, \bigcup_{j\in I_{i}} M_{j}$;

\item[(iii)] $\dim M_{j}< \dim M_{i}$ for all $j\in I_{i}$
\end{enumerate}
The stratum with the leaves with larger dimensions (the \emph{regular leaves}) \emph{is a dense open set, and its space of leaves is connected}.
In addition, in the regular stratum, we can consider the \emph{principal leaves}, i.e., the leaves with trivial holonomy. Recall here that, given a regular leaf $L$ and a curve $\beta\colon [0,1]\to L$, a \emph{holonomy map} 
(associated to the  homotopy class $[\beta]$) is a map 
$\varphi_{[\beta]}\colon S_{\beta(0)}\to S_{\beta(1)}$ defined as $\varphi_{\beta}(x)= \exp_{\beta(1)}\circ\mathcal{P}_{\beta}\circ\exp_{\beta(0)}^{-1}(x)$, where 
$\mathcal{P}_{\beta}$ is the parallel transport along $\beta$ with respect to the (flat) Bott connection $\nabla_{X}\xi=[X,\xi]^{\nu}$ (where $X$ is tangent to the leaves and $\xi$ is normal). In other words,  a holonomy map along a curve $\beta$ 
is a map defined from $S_{\beta(0)}$ to $S_{\beta(1)}$ by \emph{``sliding along the plaques"}, 
see also \cite[Definition  5.11, Remark 5.12]{alexandrino-Bettiol-2015} or \cite[Section 2.1]{MM03} or \cite[Section 1.7]{Molino}. 
We denote the space of principal leaves as $M^{0}$.
It is  a dense and open subset of $M$, and the  principal leaves can be described (when  the foliations have closed leaves)
as fibers of a submersion $\pi_{\mathcal{F}}\colon \Mprincipal\to \Mprincipal/\mathcal{F}$, where the leaf space $\Mprincipal/\mathcal{F}$ is in fact a manifold.
In addition, there exists a metric $\metric_{\mathcal{F}}$ on $\Mprincipal/\mathcal{F}$ so that the submersion 
$\pi_{\mathcal{F}}\colon (\Mprincipal,\metric) \to (\Mprincipal/\mathcal{F}, \metric_{\mathcal{F}})$ turns out to be a Riemannian submersion.

In the next subsection, we will consider a particular type of SRF (the so-called orbit-like foliation) 
that is fundamental to understanding the semi-local model of SRF; see  
\cite{alexandrino2021lie}.

We finish this subsection with a useful tool.

\begin{proposition}[\cite{Mendes-Radeschi-19}]
Let $\mathcal{F}=\{L\}$ be a SRF with closed leaves on a compact Riemannian manifold $(M,\metric)$. Consider a finite cover of $M$ by
of geometric tubular neighborhoods of leaves $U_{i}=\mathrm{Tub}_{\delta_i}(L_{q_i})$. 
Then there exists a $\mathcal{F}$-partition of unity $\{\rho_{i}\}$ subordinate
to $\{U_i\}$, i.e. where the functions $\rho_{i}$ is $\mathcal{F}$-basic.   
\end{proposition}

\vspace{0.5\baselineskip}

\subsection{Linearization of SRF and orbit-like foliations}

Given a closed leaf $L_q$ we can always find a $\mathcal{F}$-saturated tubular neighborhood $U=\mathrm{Tub}_{\delta}(L_{q})$ of $L_q$.
The foliation restricted to $U$, i.e. $\mathcal{F}_{U}$ (and in particular the partition by plaques) 
are invariant by the \emph{homothetic transformation} 
$h_{\lambda}\colon U\to U$ defined as $h_{\lambda}(\exp(v))=\exp(\lambda v)$ for each $v\in\nu^{\delta}(L_q)$
where $\lambda\in(0,1]$.

For each smooth vector field $\vec{X}$ in $U$ tangent to $\mathcal{F}$, we associate
a smooth vector field $\vec{X}^{\ell}$, called the \emph{linearization of
$\vec{X}$} with respect to $B$ as:
\[
\vec{X}^{\ell}(q)=\lim_{\lambda\to 0} (h_{\lambda}^{-1})_{*} (\vec{X})
=\lim_{\lambda\to 0} \mathrm{d}(h_{\lambda}^{-1}) \vec{X}\circ h_{\lambda}(q).
\]
Since the restricted foliation $\mathcal{F}_{U}$ is homothety-invariant, $\vec{X}^{\ell}$ is still tangent to $\mathcal{F}$.

\begin{lemma}
\label{MALEX-lemma-fluxo-isometiras}

The flows of the linearized  vector fields (once identified with flows on the normal space $\nu(L_q)$ via the  normal  exponential map)
induce isometries on the fibers of the normal $\delta$-fiber bundle 
$E^{\delta}=\nu^{\delta}(L_q)=\{e_{p}\in\nu(L_q); \, \| e_{p} \|<\delta, \ p\in L_{q} \}$.
\end{lemma}

\begin{example}
\label{MALEX-example-linearized-vector-field}

Given a SRF $\mathcal{F}$ with compacts leaves on $\mathbb{R}^{n}$ and $L_{0}=\{0\}$
we have for $\vec{X}$ tangent to the leaves of $\mathcal{F}$ that $\vec{X}^{\ell}(v) = \lim_{\lambda\to 0}\frac{1}{\lambda}\vec{X}(\lambda v) =(\nabla_{v}\vec{X})_{0}$, i.e, $\vec{X}^{\ell}$ is in fact a linear vector field. In addition, one can check that it is also a Killing vector field. This can be proved using the fact that the leaves are tangent to the spheres and hence $0=\langle \vec{X}^{\ell}(v),v\rangle= \langle (\nabla_{v}\vec{X})_{0},v\rangle$.
Note that the Killing vector fields $\vec{X}^{\ell}$ induce a Lie algebra of a  connected Lie subgroup $G^{0}\subset \mathbb{O}(n)$. 
Since by hypothesis $\mathcal{F}$ is compact, it is possible to check that $G^{0}$ is also compact. We have then in this example an homogenous subfoliation $\mathcal{F}^{\ell}=\{G^{0}(v)\}_{v\in\mathbb{R}^{n}}\subset \mathcal{F}$. This is the maximal  homogeneous sub-foliation of $\mathcal{F}$.
\end{example}

The above example illustrates a  more general phenomenon. As explained in Section \ref{Section-Introduction}
given a SRF $\mathcal{F}_{U}$ with closed leaves, the composition of linearized flows tangent to $\mathcal{F}_{U}$ induces a singular sub-foliation $\mathcal{F}^{\ell}\subset \mathcal{F}_{U}$ on $U$, the so-called \emph{linearized foliation}.
The foliation $\mathcal{F}^{\ell}$ can also be seen as the maximal infinitesimally homogeneous sub-foliation of $\mathcal{F}_{U}$. In other words,  
define $\mathcal{F}_{q}^{\ell}$ as the extension of $\exp^{-1}_{q}\big(S_{q}\cap \mathcal{F}^{\ell}\big)$ via the homothetic transformation $h_{\lambda}^{0}(v)=\lambda v$. The foliation $\mathcal{F}_{q}^{\ell}$ is the maximal transverse homogenous sub-foliation of the infinitesimal foliation $\mathcal{F}_{q}$.

\begin{remark}[Connected component of the isotropy]
\label{remark-Connected-component-isotropy}
The leaves of $\mathcal{F}_{q}^{\ell}$ are orbits of a connected Lie group $G^{0}_{q}$. 
We point out that $G_q^0$ is a subgroup of $\mathbb{O}(k)$, where $k=codim(L_q)$.
As in the Example \ref{MALEX-example-linearized-vector-field} 
the Lie algebra of this Lie group is generated by 
the linearization of vector fields tangents to $\mathcal{F}_{q}$. 
Note that given a  leaf $L_x^{\ell}\in \mathcal{F}^{\ell}$ (where $x\in S_q$),  the intersection $L_{x}^{\ell}\cap S_{q}$ 
can be a (finite) union of (connected)  leaves of $\mathcal{F}_q^{\ell}$. 
So it will be convenient to define the isotropy group as a possibly non-connected Lie group $G_q$, 
where the  orbit $G_{q}(x)$   corresponds to the leaves $L_{x}^{\ell}\cap  S_q$, see 
Definition \ref{definition-isotroy-group}. 
\end{remark}

\begin{definition}[Orbit-like foliation]
\label{definition-orbit-like}
A SRF $\mathcal{F}$  with closed leaves is called \emph{orbit-like foliation} if for each leaf $L_q$ and a small tubular neighborhood 
$U=\mathrm{Tub}_{\delta}(L_q)$ the restricted foliation $\mathcal{F}_{U}$ coincides with the linearized foliation $\mathcal{F}^{\ell}$.
\end{definition}

\begin{remark}
\label{MALEX-remark-orbit-like-topological}

To be orbit-like could be considered a topological property in the following sense: let $(M_i,\mathcal{F}_{i})$ be two SRF and $\psi\colon (M_{1},\mathcal{F}_{1})\to (M_{2},\mathcal{F}_{2})$ be a foliated diffeomorphism. Then $(M_{2},\mathcal{F}_{2})$ is orbit-like if and only if $(M_{1},\mathcal{F}_{1})$ is orbit-like, see \cite{alexandrino_radeschi_2017}.
\end{remark}

\subsection{Sasaki metric and SRF}
\label{subsection-sasaki-metric-SRF}

Consider an \emph{Euclidean fiber bundle} $\mathbb{R}^{k}\to E\to B$, 
(i.e. with an inner product $\langle\,,\,\rangle_p$ on each fiber $E_p$ that depends smoothly on $p\in B$)
and an affine connection $\nabla^\tau\colon\mathfrak{X}(B)\times\Gamma(E)\to\Gamma(E)$ \emph{compatible with the metric},   
i.e. $\mathbb{R}$-bilinear map such that:
\begin{itemize}
\item $\nabla_{f X }^{\tau} \zeta= f\nabla_{X}^{\tau}\zeta$,
\item $\nabla_{X}^{\tau} f\zeta=X \cdot f\zeta+ f\nabla_{X}\zeta$,
\item $X\langle \zeta, \eta \rangle=\langle\nabla^{\tau}_X\zeta, \eta\rangle+\langle \zeta, \nabla^{\tau}_X\eta\rangle$. 
\end{itemize}
Here $\Gamma(E)$ denotes the sections of $E$.   
Recall that for  each curve $\alpha\colon(-\epsilon,\epsilon)\to B$  starting at $p=\alpha(0)$ 
and for each vector  $e_p\in E_p$ there  exists a unique vector field $t\to \tilde{\alpha}(t)\in E_{\alpha(t)}$
along the cuve $\alpha$ so that $\tilde{\alpha}(0)=e_p$ and $\frac{\nabla^{\tau}}{d t}\tilde{\alpha}=0$, i.e., it is  \emph{parallel}.
We  can then define \emph{parallel transport} $\paralleltransport_{\alpha}^{t}\colon E_{p}\to E_{\alpha(t)}$ as 
$\paralleltransport_{\alpha}^{t}e_p=\widetilde{\alpha}(t)$. This map is an isometry between the fibers. 
The space spanned 
by all velocities   $\tilde{\alpha}'(0)$ (for all piecewise curves $t\to\alpha(t) \subset B$) is denote by $\mathcal{T}_{e_p}\subset T_{e_p}E$. 
We have then constructed  the  \emph{linear Ehresmann connection} $\mathcal{T}$  
associated to the connection  $\nabla^{\tau}$. This distribution is homothety-invariant, i.e. $\mathrm{d} (h_{\lambda})\mathcal{T} =\mathcal{T}\circ h_{\lambda}$. 

\begin{remark}
Conversely, given a  \emph{homothety-invariant} distribution $\mathcal{T}$ on $E$, transverse to the fibers of $E$, 
we can define a linear map
$\paralleltransport_{\alpha}^{t}\colon E_{\alpha(0)}\to E_{\alpha(t)}$ as $\paralleltransport_{\alpha}^{t}e_p=\widetilde{\alpha}(t)$
where $t\to\tilde{\alpha}(t)$ is  the lift of $\alpha$ starting at $e_p$, i.e. the only curve 
with $\widetilde{\alpha}(0)=e_p$ and $\widetilde{\alpha}'(t)\in \mathcal{T}_{\tilde{\alpha}(t)}$.
We can recover the linear connection
 $\nabla^{\tau}\colon \mathfrak{X}(B)\times\Gamma(E)\to \Gamma(E)$
defining 
$
\nabla_{X}^{\tau}\zeta(p)=\frac{d}{d t} \paralleltransport_{\gamma}^{-t}(\zeta\circ\alpha(t))|_{t=0}
$
where $\zeta\in \Gamma(E)$ and $X\in T_{p}B$ and $\alpha'(0)=X$. In particular if the linear map $\paralleltransport_{\alpha}^{t}$ is  an isometry,
then $\nabla_{X}^{\tau}$ is compatible with the metric.  
 \end{remark}

Now assume the base  manifold  is a Riemannian manifold $(B,\metric)$. Then, we can define   the so-called
\emph{Sasaki metric} (denoted here as $\metric^{\tau}$) on the total space $E$ as the metric such that:  
\begin{itemize}
\item $\mathcal{T}$ is orthogonal to the fibers of $E$,
\item the foot point projection, $\pi\colon (E,\metric^{\tau})\to (B,\metric)$ is a Riemannian submersion,
\item the fibers $E_{p}$ have the flat metric $\langle\cdot,\cdot\rangle_{p}$.
\end{itemize}

\begin{definition}[Holonomy foliation]
\label{definition-holonomy-foliation}
The connection $\nabla^{\tau}$ induces a foliation 
$\mathcal{F}^{\tau}=\{ L^{\tau}_{e_x}\}_{e_x\in E}$ (the so-called
\emph{holonomy foliation}) as follows: we say that two vectors $e_x$ and $e_y$ of $E$ belong to the same leaf $L_{e_x}^{\tau}$
if we can parallel  transport one to the other, i.e. if $\paralleltransport_{\alpha}^{1} e_{x}=e_y$ for some piecewise smooth curve
$\alpha\colon [0,1]\to B$ with $\alpha(0)=x$ and $\alpha(1)=y$. We denote by $\mathrm{Hol}_p$ the holonomy group of the connection $\nabla^\tau$ at $p$. 
Note that $\mathcal{T}\subset T \mathcal{F}^{\tau}$ and  $\mathcal{F}^{\tau}\cap E_{p}= \{\mathrm{Hol}_{p}e_{p} \}_{e_{p}\in E_{p}}$.  
Also note that $\nabla^{\tau}$ is flat if and only if   $\mathcal{T}= T \mathcal{F}^{\tau}$ (in particular  $\mathcal{F}^{\tau}$ is
 regular). The singular foliation $\mathcal{F}^{\tau}$ on $E$ is a SRF with respect to $\metric^{\tau}$. 
\end{definition}

\begin{remark}
\label{remark-Section-background-parallel-translation}
Given  a vector field $\vec{X}$ on $B$, the parallel transport 
along  the integral curves of $\vec{X}$  induces a homothety-invariant  vector field $\vec{Y}\in \mathfrak{X}(E)$. More precisely, 
for $\vec{X}\in \mathfrak{X}(B)$ we can define a vector field
$\vec{Y}\in \mathfrak{X}(E)$ with flow 
$\efunction^{t Y}= \paralleltransport_{\efunction^{t X}}$. The leaves of  $\mathcal{F}^{\tau}$ can be also seen as orbits of these kinds of
vector fields, i.e. of compositions of \emph{parallel translations flow}. 

\end{remark}

 On the one hand, holonomy foliations are natural examples of SRF. 
This can be proved by recalling that the fibers of $E$ are totally geodesics with respect to the Sasaki metric $\metric^{\tau}$, see \cite[Proposition 2.7.1]{Gromoll-Walschap09}.
 
On the other hand, they are always sub-foliations 
of $\mathcal{F}_U$, as we now recall.

 Let $\mathcal{F}$ be a SRF (with  closed leaves) on $(M,\metric)$. 
For a fixed $p_0$ set $B=L_{p_0}$ a leaf  and let  $E = \nu(B)$ be the normal bundle of the leaf $B$. 
Consider the Euclidean vector bundle $\mathbb{R}^{k} \to E \to B$ where the metric on each fiber $E_p$ is the metric $\metric_p$ 
restricted to the normal space $\nu_p B$.
By pulling back  the foliations $\mathcal{F}^{\ell}$ and $\mathcal{F}_{U}$
(via the normal exponential map $\exp^{\nu}$),  we obtain singular foliations on the open set $(\exp^{\nu})^{-1}(U) \subset E$ which can be extended 
by homothety transformations $h_\lambda$ to foliations on the whole vector bundle $E$, since both foliations are invariant under homotheties. From now on, we are going to use the same notation for the foliations $\mathcal{F}^{\ell}$ and 
$\mathcal{F}_{U}$ on $U$ or $E$,
hoping that it will be clear from the context when we are working in the tubular neighborhood $U$ or in the normal bundle $E$.

It is always possible to  find a homothety-invariant distribution $\mathcal{T}$ (where $\dim \mathcal{T}=\dim B$)
tangent to the leaves
of $\mathcal{F}_U$ (and in fact, tangent to the leaves of $\mathcal{F}^{\ell})$ that induces a  connection $\nabla^{\tau}$ compatible with the metrics on the fiber
$(\nu(B),\metric)\to B$, see \cite[Section 5.1]{alexandrino_radeschi_2017} 
 Therefore, we have:  
\[
\mathcal{F}^{\tau}\subset\mathcal{F}^{\ell}\subset\mathcal{F}_{U}, 
\]
where the above three foliations are SRF for the Sasaki metric  $\metric^{\tau}$ induced by $\mathcal{T}$, 
see \cite[Proposition 7.1]{alexandrino_radeschi_2017}. 
In particular, the fact  that $\mathcal{F}^{\tau}$ is a SRF can be proved
using the fact that the fibers of the fiber bundle are totally geodesic with respect to the Sasaki metric. 

\begin{remark}
As we will see in Subsection \ref{subsection-Ex-SRF-Sasaki-AVP}, 
the leaves of the  linearized foliation $\mathcal{F}^{\ell}$ are orbits of two type of linearized vector fields: 
one is just like those vector fields presented in Remark \ref{remark-Section-background-parallel-translation},
i.e., vector fields 
whose flows are  \emph{parallel translation flows}; the other is linearized vector fields,
vector fields 
 whose  flows fix each fiber of the normal bundle $\nu(B)$   
and act isometrically on each fiber of $\nu(B)$, and we could call  them \emph{rotation flows}. We can choose these vector fields so that
their flows preserve the volume of the  Sasaki metric. 
\end{remark}

We should stress that there are several ways to build 
such distributions $\mathcal{T}$, 
and thus, in contrast to the concept of linearizable foliation 
$\mathcal{F}^{\ell}$, there is not a unique holonomy sub-foliation $\mathcal{F}^{\tau}$. 
Nevertheless, the existence of holonomy foliations is quite important to understand  the semi-local behavior of $\mathcal{F}_U$, see also \cite{alexandrino2021lie}. However, at this point, the reader may already be able to intuit that although the  
identity component of the holonomy group $\mathrm{Hol}_{p}^{0}(E)$ 
is a normal subgroup of the
 connected group $G_{p}^{0}$ (recall Remark \ref{remark-Connected-component-isotropy}), there is some information about $\mathcal{F}$ that $G_{p}^{0}$ does not contain. 
For example, it does not describe how leaves return to a slice $S_p$. 
Such information is contained in the other connected components  of the Holonomy group $\mathrm{Hol}_{p}(E)$.   
Moreover, this motivates us to consider the following definition.

\begin{definition}[isotropy group]
\label{definition-isotroy-group}
We define the (possibly) non-connected group 
\emph{isotropy group} $G_{p}$ as the compact group
generated by $G_{p}^{0}$ and $\mathrm{Hol}_{p}$,  recall Remark \ref{remark-Connected-component-isotropy}.
\end{definition}

We should remark that the Lie group $G_p$ does not depend on the choice of the distribution $\mathcal{T}$, and that 
each  orbit $G_{p}(x)$ in the slice $S_p$ coincides with the (possibly non-connected) submanifold $L_{x}^{\ell}\cap S_{p}$, for $x\in S_p$.

\section{The principle of symmetric criticality}
\label{Section-The principle of symmetric criticallity}

The goal of this section is to prove Theorem \ref{theorem-simple-version-Palais-principle-variational-formulation}. For the proof of  it
we are going to use  a few results that are proved in  Section \ref{Section- Symmetric Criticality of Palais on manifold}.

Along this section, we always assume that $\mathcal{F}$ is an AVP foliation and the geometric controls $\mathcal{C}(U)$'s on tubular neighborhoods 
are always AVP.

Although in this article we are only interested in AVP geometric controls, we highlight, for future reference, 
 that  the  results and definitions presented in Subsection \ref{subsection-Semi-local definitions} 
make sense for general geometric control systems.

\subsection{Semi-local definitions}
\label{subsection-Semi-local definitions}

To increase the number of examples where the principle of  symmetric criticality
 can be applied, we wish to discard the orbit-like hypothesis. 
This forces us to increase the set of ``basic'' functions we must deal with. 
Although this new space of ``basic'' functions is larger than the space of basic functions itself, 
it will still be smaller than the original space of functions, 
and this will allow us to conclude that there are weak solutions 
bypassing the problem of the critical index of the Rellich–-Kondrachov embedding theorem.

\begin{definition}
\label{definition-linearized-basic}
Let $U_{i}=\{\mathrm{Tub}_{\delta_{i}}(L_{p_i})\}$ be a finite cover by (geometric) tubular neighborhoods of the compact Riemannian manifold $(M,\metric)$, and on each $U_i$ consider the linearized foliations $\mathcal{F}^{\ell}_{i}$ for $L_{p_i}$.
Also  consider on each $U_{i}$ a geometric control  system of linearized vector fields $\mathcal{C}_{i}:=\{\vec{X}_{\alpha,i}\}$ associated to
$\mathcal{F}^{\ell}_{i}$,  i.e, whose orbits are the leaves of $\mathcal{F}^{\ell}_{i}$. 
A function $b\in C^{\infty}(M)$  is called 
\emph{linearized basic}  
if the restriction of $b$ to each  $U_{i}$ is $\mathcal{F}^{\ell}_{i}$-basic (i.e. is constant along the leaves of   $\mathcal{F}^{\ell}_{i}$).
Equivalently, we say that a smooth function $b$  is linearized basic if its directional derivatives for all these geometric controls are zero, i.e.
$\vec{X}_{\alpha,i}\cdot b=d b (\vec{X}_{\alpha,i})=0$ for each $i$ and for each $\vec{X}_{\alpha,i}\in \mathcal{C}_{i}$. The vector space of these functions is denoted
by $\linearizedbasic$ and its closure for the Sobolev norm of $W^{1, p}(M)$ by $\closedlinearizedbasic$.
\end{definition}

Note that $\linearizedbasic\supset C^{\infty}(M)^{\mathcal{F}}$, i.e., $\linearizedbasic$
 contains the space of smooth basic functions and hence is not an empty vector space. 
In addition if $\mathcal{F}$ is orbit-like foliation, then $\linearizedbasic= C^{\infty}(M)^{\mathcal{F}}$.

\begin{definition}
\label{definition-C-symmetric-operator}
A continuous linear functional 
$\lfunctional\colon W^{1,p}(M)\to\mathbb{R}$ is called $\{\mathcal{C}_{i} \}$-\emph{locally symmetric} if  
\[
\lfunctional(h\circ \efunction^{t X_{\alpha,i}})=\lfunctional(h)
\]
 for each
$h\in C_{c}^{\infty}(U_i)$ and  $\vec{X}_{\alpha,i}\in\mathcal{C}_{i}$.
 A $C^{1}$ operator $J\colon  W^{1,p}(M)\to\mathbb{R}$ is called  critical 
$\{\mathcal{C}_{i} \}$-\emph{symmetric} if for each $b\in \closedlinearizedbasic$ the linear operator 
$ \mathrm{d} J(b)$ is locally symmetric with respect to $\{\mathcal{C}_{i}\}$.
\end{definition}

Consider a $\{\mathcal{C}_{i} \}$-\emph{locally symmetric}  continuous linear functional  $\lfunctional\colon W^{1,p}(M)\to\mathbb{R}$.  
We will say that it fulfills the \emph{local  principle of symmetric criticality}  if for each fixed $i$
the restricted  functional $\lfunctional_{i}:=\lfunctional|_{C_{c}^{\infty}(U_i)}\colon {C_{c}^{\infty}(U_i)}\to\mathbb{R}$  satisfies:
\begin{equation}
\label{equation-def-local-principle-symmetric-criticallity}
\lfunctional_{i}(b)=0,\, \forall b\in C_{c}^{\infty}(U_i)^\ell\, \Rightarrow  \, \lfunctional_{i}(f)=0, \, \forall f\in C_{c}^{\infty}(U_i).
\end{equation}

\begin{lemma}
\label{lemma-local-to-global-PSC}
If $\lfunctional\colon W^{1,p}(M)\to\mathbb{R}$ fulfills a local  principle of symmetric criticality, then
\[
\lfunctional(b)=0,\,  \forall b\in  \closedlinearizedbasic \, \Rightarrow \, \lfunctional(f)=0, \, \forall f\in W^{1,p}(M).
\]
\end{lemma}
\begin{proof}
Assume that $\lfunctional(b)=0,\,  \forall b\in  \closedlinearizedbasic$.
This implies in particular that  $\lfunctional(b_i)=\lfunctional_{i}(b_i)=0$ for $b_{i}\in C_{c}^{\infty}(U_i)$
and hence, we infer from \eqref{equation-def-local-principle-symmetric-criticallity}
that $\lfunctional_{i}=0$.
 Let  
$\{\rho_{i}\}$  be a $\mathcal{F}$-partition of unity subordinate to  $\{U_{i}\}$. Given a $f\in  C^{\infty}(M)$ we set
$f_{i}:=\rho_{i}f$. We conclude that  
$\lfunctional(f)=\sum_{i}\lfunctional(f_i)=\sum_{i}\lfunctional_{i}(f_i)=0$. 
The arbitrariness in the choice of $f$ implies that $\lfunctional=0$ on the vector space $C^{\infty}(M)$ and hence, since $\lfunctional$ is continuous, $\lfunctional=0$ on its closure, i.e. on $W^{1,p}(M)$.
\end{proof}

\subsection{Symmetric integral operators on the Banach spaces $W^{1,p}(M)$ }
Theorem \ref{theorem-simple-version-Palais-principle-variational-formulation} is a direct consequence of the next result.

\begin{theorem}
\label{theorem-Palais-principle-variational-formulation}
Let $\mathcal{F}$ be an  AVP foliation on a  compact (connected) Riemannian manifold $(M,\mathsf{g})$, 
and $J\colon W^{1, p}(M) \to \mathbb{R}$ be a $C^{1}$ operator such that for each  smooth basic function $b\in \linearizedbasic$ 
\[
\mathrm{d} J(b)f=\int_{M}\big( \mathcal{L}_{1}^{\ell}(b)\, \metric (\nabla b, \nabla f) +  \mathcal{L}_{2}^{\ell}(b) f\big) \density_{\metric}, \, \,
\forall f\in C^{\infty}(M),
\]
where (for $k=1,2$)  
$ \mathcal{L}_{k}^{\ell}\colon \linearizedbasic\to C^{0}(\{U_{i}\})^{\ell}$,
i.e., $\mathcal{L}_{k}^{\ell}(b)$  is a continuous linearized basic function. 
Assume that  for a fixed $b_{0}\in  \closedlinearizedbasic$ 
we have $ \mathrm{d}J(b_0)(b)=0$ for all $b\in   \closedlinearizedbasic$.
Then  $\mathrm{d}J(b_0)=0$.
\end{theorem}

\begin{proof}
 
For each $i$, set $\lfunctional_{i}= \mathrm{d} J(b_0)\colon C_{c}^{\infty}(U_{i})\to\mathbb{R}$.  
As will be proved in  Lemma \ref{lemma-J-implies-l-linear}, 
this continuous linear functional   is symmetric.
 
We first claim that  \emph{ $\lfunctional= \mathrm{d}J(b_0)$  
fulfills the local  principle of symmetric criticality, i.e., $\lfunctional_{i}$  fulfills Equation 
\eqref{equation-def-local-principle-symmetric-criticallity}.}
 By hypothesis, $ \mathrm{d}J(b_0)(b)=0$ for all $b\in   \closedlinearizedbasic$ and hence 
$\lfunctional_{i}(b)=0$, $\forall b\in C_{c}^{\infty}(U_i)^\ell$.  As we will explain in Section \ref{subsection-avarage-first-intuition}, there exists an operator
$ \mathrm{Av}\colon C_{c}^{\infty}(U_{i})\to C_{c}^{\infty}(U_{i})^{\ell}$, i.e., an operator that project smooths functions on $U_i$
onto smooth linearized basic functions on $U_i$. 
Given  $f \in C^{\infty}_{c}(U_{i})$, we have: 
\begin{equation*}
0 \stackrel{(*)}{=} \lfunctional_{i} \big(\mathrm{Av}(f )\big) \stackrel{(**)}{=}  \lfunctional_{i} \big(f \big)
\end{equation*}
where (*) follows the fact that $\mathrm{Av}( f)$ is a linearized basic function  
 and (**) follows from Proposition \ref{proposition-new-MALEX--Av-l-orbitlike}. Then $\lfunctional_{i}(f)=0$. This conclude the proof of the claim, i.e., that $\lfunctional=\mathrm{d}J(b_0)$  fulfills the local  principle of symmetric criticality.

Now  since $\lfunctional$ fulfills the local principle of symmetric criticality 
and $ \mathrm{d}J(b_0)(b)=0$ for all $b\in   \closedlinearizedbasic$,  Lemma \ref{lemma-local-to-global-PSC} implies that   
 $\lfunctional=\mathrm{d}J(b_0)=0$, and this conclude  the proof. 
\end{proof}

\subsection{Symmetric integral operators on the Hilbert space $W^{1,2}(M)$}

In what follows we give an  alternative proof  of the previous theorem in  the case of  $W^{1,2}(M)$.  
The main advantage of this proof is avoiding the use of the operator $\mathrm{Av}$.

\begin{theorem}
\label{theorem-Diego-general-version-Palais-Hilbert}
Let $\mathcal{F}$ be an  AVP foliation on a  compact (connected) Riemannian manifold $(M,\mathsf{g})$, 
and $J\colon W^{1, 2}(M) \to \mathbb{R}$ be a $C^{1}$ operator such that for each  smooth basic function $b\in \linearizedbasic$ 
\[
\mathrm{d} J(b)f=\int_{M}\big( \mathcal{L}_{1}^{\ell}(b)\, \metric (\nabla b, \nabla f) +  \mathcal{L}_{2}^{\ell}(b) f\big) \density_{\metric}, \, \,
\forall f\in C^{\infty}(M),
\]
where (for $k=1,2$)  
$ \mathcal{L}_{k}^{\ell}\colon \linearizedbasic\to C^{0}(\{U_{i}\})^{\ell}$,
i.e., $\mathcal{L}_{k}^{\ell}(b)$  is a continuous linearized basic function. 
Assume that  for a fixed $b_{0}\in  \closedlinearizedbasictwo$ 
we have $ \mathrm{d}J(b_0)(b)=0$ for all $b\in   \closedlinearizedbasictwo$.
Then  $\mathrm{d}J(b_0)=0$.
\end{theorem}

\begin{proof}
Let $\lfunctional_i= \mathrm{d} J(b_0)\colon C_{c}^{\infty}(U_i)\to\mathbb{R}$ be
the restricted operator. We know by hypothesis that $\lfunctional_{i}(b)=0$ for all $b\in C_{c}^{\infty}(U_i)^{\ell}$. 
We can now apply Proposition \ref{lemma-new-theorem-Diego-general-version-Palais-Hilbert} below
to infer that $0= \lfunctional_{i}(f)$ 
for all $f\in C_{c}^{\infty}(U_i)$, i.e., we have proved that 
$\lfunctional_{i}$  fulfills Equation \eqref{equation-def-local-principle-symmetric-criticallity}.
 Therefore, by Lemma \ref{lemma-local-to-global-PSC},  $\mathrm{d} J(b_0)=\lfunctional=0$
what finishes the proof. 
\end{proof}

To state and prove  Proposition \ref{lemma-new-theorem-Diego-general-version-Palais-Hilbert},   we proceed with a few preparations.

Let $U_r$ be a $\mathcal{F}$ saturated neighborhood contained in the principal stratum at a  distance $r$ to the singular stratification. 
We define also  a new inner product. 
\begin{equation}
\label{eq1-lemma-new-theorem-Diego-general-version-Palais-Hilbert}
\langle u,v\rangle_{\perp}= \int_{U_r} \metric^\perp(\nabla u,\nabla v)+uv \, \density_{\metric}   \quad \mbox{for }u,v\in C^{\infty}_{c}(U_r).
\end{equation}

Here $\metric^\perp$ at $x\in U_r$ is given by 
$\metric^\perp_{x}(\cdot,\cdot) = \metric|_{\nu_{x}(L_x)}(\pi^{\perp}_{x}(\cdot), \pi^{\perp}_{x} (\cdot) )$, 
where $\pi^{\perp}_{x}\colon T_{x} U_{r}\to \nu_{x}L_{x}$ is  the $\metric$-orthogonal
projection.  Let $W^{1,2}_{0,\perp}(U_r)$ be the closure of $C^{\infty}_{c}(U_r)$ (i.e, \emph{the completion})  
with respect to  the associated norm $\|\cdot\|_{\perp}$.

\begin{lemma}
\label{lemma1-new-theorem-Diego-general-version-Palais-Hilbert}
Fix $u\in C^{\infty}_{c}(U_r)\setminus\{\bar{0}\}$ and 
take $v = u\circ \efunction^{t X_{\alpha}}$, where  $\efunction^{t X_{\alpha}}$ is an AVP flow. 
Then 
$
\|u\|_{\perp}=\|v\|_{\perp}.
$
\end{lemma}
\begin{proof}
\[
\Vert v\Vert^2_{\perp} = \int_{U_r} | \nabla (u\circ \efunction^{t X_{\alpha}})|^2_{\metric^\perp}+u^2\circ \efunction^{t X_{\alpha}}\, \density_{\metric}.
\]
For $Y\in \nu_pL_p$ we have 
\begin{linenomath}
\begin{align*}
\metric^\perp_{x}(\nabla( u\circ \efunction^{t X_{\alpha}}),Y)& = \metric_{x}\big(\nabla( u\circ \efunction^{t X_{\alpha}}),Y\big)\\
&= \mathrm{d}(u\circ \efunction^{t X_{\alpha}} )_{x}(Y(x))\\
&=\mathrm{d} u_{\efunction^{t X_{\alpha}}(x)}(\mathrm{d} \efunction^{t X_{\alpha}}_{x}(Y(x)))\\
&= \metric_{\efunction^{t X_{\alpha}}(x)}\big(\nabla u (\efunction^{t X_{\alpha}}(x)), \mathrm{d} \efunction^{t X_{\alpha}}_{x}(Y(x))\big)\\
&= \metric^\perp_{\efunction^{t X_{\alpha}} (x)}\big(\nabla u^{\perp} (\efunction^{t X_{\alpha}} (x)), \mathrm{d}  \efunction^{t X_{\alpha}}_x   (Y(x))\big)\\
&= \metric^\perp_{x}\big(\mathrm{d}(\efunction^{t X_{\alpha}})^{-1}_{\efunction^{t X_{\alpha}} (x)}(\nabla u^{\perp}(\efunction^{t X_{\alpha}}(x))),Y(x)\big).
\end{align*}
\end{linenomath}
From this, we conclude that $\nabla( u\circ \efunction^{t X_{\alpha}} )^\perp(x)
 = \mathrm{d}(\efunction^{t X_{\alpha}}  )^{-1}_{\efunction^{t X_{\alpha}} (x)}(\nabla u^\perp(\efunction^{t X_{\alpha}} (x)))$.  This implies that 
\begin{equation}
\label{eq1-lemma1-new-theorem-Diego-general-version-Palais-Hilbert}
|\nabla ( u\circ \efunction^{t X_{\alpha}} )|^2_{\metric^{\perp}}(x)=|\nabla  u|^2_{\metric^\perp}(\efunction^{t X_{\alpha}} (x)).
\end{equation}
Thus, 
\begin{linenomath}
\begin{align*}
\|v \|^2_{\perp}&= \|u\circ  \efunction^{t X_{\alpha}}  \|^2_{\perp}\\
&=\int_{U_r} | \nabla (u\circ  \efunction^{t X_{\alpha}} )|^2_{\metric^\perp}+ (u\circ  \efunction^{t X_{\alpha}} )^2\, \density_{\metric} \\ 
& \stackrel{\eqref{eq1-lemma1-new-theorem-Diego-general-version-Palais-Hilbert}}{=}
 \int_{U_r} | \nabla u|^2_{\metric^\perp}\circ  \efunction^{t X_{\alpha}}  + u^2\circ \efunction^{t X_{\alpha}} \, \density_{\metric}\\
&= \int_{U_r} | \nabla u|^2_{\metric^\perp}+u^2\, \density_{\metric}\\
&= \|u\|^2_{\perp}.
\end{align*}
\end{linenomath}
\end{proof}

\begin{lemma}
\label{lemma2-new-theorem-Diego-general-version-Palais-Hilbert}
 $\lfunctional$ is a  continuous linear functional  on $C^{\infty}_{c}(U_r)$ with respect  to the norm $\|\cdot\|_{\perp}$. 
\end{lemma}
\begin{proof}
Since $\lfunctional$ is linear,
it  suffices to check that if a  sequence $\{u_{m}\} \subset C^{\infty}_{c}(U_r)$ 
converges to zero with  respect  to $\| \cdot\|_{\perp}$,  
 then $|\lfunctional(u_{m})|\to 0$.

From the  definition,  $\|u_{m}\|_{\perp}\to 0$ implies:
\begin{equation}
\label{eq1-lemma2-new-theorem-Diego-general-version-Palais-Hilbert}
\sqrt{\int_{U_r}|\nabla u_{m}|^{2}_{\metric^\perp} \vol_{\metric}}\to 0 
\end{equation}
and 
\begin{equation}
\label{eq2-lemma2-new-theorem-Diego-general-version-Palais-Hilbert}
\sqrt{\int_{U_r} |u_m|^{2} \vol_{\metric}}\to 0.
\end{equation}

Let  $\{e_{A}\}_{A=1}^{n}$ be an orthonormal local  frame  adapted to $\mathcal{F}$,  i.e., the $n-k$ first vectors fields 
$\{e_{i}\}_{i=1}^{n-k}$ are tangent to the leaves of $\mathcal{F}$ 
and last $k$ vector fields  $\{e_{\beta}\}_{\beta=1}^{k}$ are orthogonal to the leaves.  

In this adapted  orthonormal frame, we can write:
\begin{align*}
\nabla u_{m} &= \sum_{i} (e_{i}\cdot u_{m}) e_{i}+ \sum_{\beta} (e_{\beta}\cdot u_m) e_{\beta},\\
\nabla b_{0} &= \sum_{\beta}(e_{\beta}\cdot b_{0})e_{\beta}.
\end{align*}
From the above equations, we infer: 
\begin{equation}
\label{eq3-lemma2-new-theorem-Diego-general-version-Palais-Hilbert}
| \metric(\nabla b_{0}(x),\nabla u_{m}(x)) | \leq |\nabla b_{0}(x)|_{\metric} |\nabla u_{m}(x)|_{\metric^{\perp}}.
\end{equation}
Therefore we have: 
\begin{align*}
|\lfunctional(u_m)| &\leq \int_{U_r}\big( |\mathcal{L}_{1}^{\ell}(b_0)|\, |\metric (\nabla b_0, \nabla u_m)| 
+  |\mathcal{L}_{2}^{\ell}(b_0)| |u_m|\big) \density_{\metric} \\
&\leq \int_{U_r}\big( |\mathcal{L}_{1}^{\ell}(b_0)|\,|\nabla b_0|_{\metric} |\nabla u_{m}|_{\metric^{\perp}} 
+  |\mathcal{L}_{2}^{\ell}(b_0)| |u_{m}|\big) \density_{\metric} \\
&\leq c_{1} \int_{U_r} |\nabla u_{m}|_{\metric^{\perp}} \density_{\metric}
+ c_{2}\int_{U_r}   |u_{m}|  \density_{\metric}\\
&\leq \sqrt{\int_{U_r} 1 \density_{\metric}}
\Big(c_{1} \sqrt{\int_{U_r} |\nabla u_{m}|_{\metric^{\perp}}^{2} \density_{\metric}} + c_{2}\sqrt{\int_{U_r}   |u_{m}|^{2}  \density_{\metric}}    \Big),
\end{align*}
where the positive constants $c_1$ and $c_2$ do not depend on $\{u_{m}\}\subset C_{c}^{\infty}(U_r)$. 
Therefore,  
\begin{equation}
\label{eq4-lemma2-new-theorem-Diego-general-version-Palais-Hilbert}
|\lfunctional(u_{m})| 
\leq c_{3}\sqrt{\int_{U_r}|\nabla u_{m}|^{2}_{\metric^\perp} \density_{\metric} }+
 c_{4}\sqrt{\int_{U_r} |u_{m}|^{2} \density_{\metric}},
\end{equation}
where the positive constants $c_3$ and $c_4$ do not depend on $\{u_{m}\}\subset C_{c}^{\infty}(U_r)$. 

Equation \eqref{eq1-lemma2-new-theorem-Diego-general-version-Palais-Hilbert}, 
\eqref{eq2-lemma2-new-theorem-Diego-general-version-Palais-Hilbert} and 
\eqref{eq4-lemma2-new-theorem-Diego-general-version-Palais-Hilbert}
conclude the proof. 
\end{proof}

\begin{proposition}
\label{lemma-new-theorem-Diego-general-version-Palais-Hilbert}
Let $\lfunctional=\mathrm{d} J(b_0) \colon C_{c}^{\infty}(U)\to\mathbb{R}$.   
 Assume that  $\lfunctional(b)=0$ 
for all linearizable basic function $b\in C_{c}^{\infty}(U)^{\ell}$.
Then  $\lfunctional(f)=0$, for all $f \in C_{c}^{\infty}(U)$. 
\end{proposition}
\begin{proof}
Let $\hat{\lfunctional}\colon W_{0,\perp}^{1,2}(U_r)\to \mathbb{R}$ be the continuous  extension of the continuous functional  $\lfunctional$ 
(recall Lemma \ref{lemma2-new-theorem-Diego-general-version-Palais-Hilbert}). 
By Riez representation there exists a unique vector $\nabla \hat{\lfunctional}\in W_{0,\perp}^{1,2}(U_r)$ such that:  
\[
\hat{\lfunctional}(v)=\langle \nabla \hat{\lfunctional}, v \rangle_{\perp}.
\]
Let $\Sigma\subset W_{0,\perp}^{1,2}(U_r)$ be the closure of the vector space $C_{c}^{\infty}(U_r)^{\ell}$ with respect to $\|\cdot\|_{\perp}$. 
To prove the proposition it suffices  to check that $\nabla \hat{\lfunctional}\in \Sigma$. 
Let $v=\nabla\hat{\lfunctional}\circ \efunction^{t X_\alpha}$, where  $\efunction^{t X_\alpha}$ is an AVP flow.

\begin{eqnarray*}
\langle \nabla \hat{\lfunctional},\nabla \hat{\lfunctional} \rangle_{\perp} & = &\lfunctional(\nabla \hat{\lfunctional})\\
& \stackrel{(i)}{=} &  \lfunctional(v) \\
&= & \langle \nabla \hat{\lfunctional}, v \rangle_{\perp} \\
&\leq &\|\nabla \hat{\lfunctional}\|_{\perp} \|v\|_{\perp}\\                    
& \stackrel{(ii)}{=} & \|\nabla \hat{\lfunctional}\|_{\perp}^{2},\\
\end{eqnarray*}
where equality $(i)$ follows from the symmetry of $\lfunctional$ and equality (ii) follows from Lemma \ref{lemma1-new-theorem-Diego-general-version-Palais-Hilbert}.

From the above equation we conclude that or $\nabla \hat{\lfunctional}=0$ (what already implies that $\lfunctional=0$) or 
  $\nabla\hat{\lfunctional}=v=\nabla\hat{\lfunctional}\circ \efunction^{t X_\alpha}$, 
$\forall \vec{X}_{\alpha}\in\mathcal{C}(U)$.  Therefore 
$\nabla\hat{\lfunctional}\in \Sigma$.  Since by hypothesis $\lfunctional(b)=0$ for all $b\in C_{c}^{\infty}(U_r)^{\ell}$
we conclude that $\lfunctional(f)=0$ for all $f\in C_{c}^{\infty}(U_r)$.  We can now take $r\to 0$ and conclude that 
$\lfunctional(f)=0$, for all $f \in C_{c}^{\infty}(U)$. 
\end{proof}

\section{AVP foliations and symmetric integral operators  }
\label{Section- Symmetric Criticality of Palais on manifold}

Along  this section $\mathcal{F}=\{L_{x}\}_{x\in M}$ is going to be  an AVP foliation with closed leaves
 on a compact (connected) Riemannian manifold $(M,\mathsf{g})$, recall Definition \ref{definition-a-v-p}.  
We are going to  consider a finite open cover of tubular neighborhoods  $U_{i}=\mathrm{Tub}_{\delta_i}(L_{p_i})$, 
geometric AVP control systems   $\mathcal{C}_{i}:=\{\vec{X}_{\alpha,i}\}$ associated to 
$\mathcal{F}^{\ell}_{i}$, (recall  Definitions \ref{definition-a-v-p} and  \ref{definition-linearized-basic}).
We also consider a   $C^{1}$ operator  $J\colon W^{1, p}(M) \to \mathbb{R}$  such that for each  smooth basic function $b\in \linearizedbasic$ 
\begin{equation}
\label{eq: derivative of considered operators}
\mathrm{d} J(b)f=\int_{M}\big( \mathcal{L}_{1}^{\ell}(b)\, \metric (\nabla b, \nabla f) +  \mathcal{L}_{2}^{\ell}(b) f\big) \density_{\metric}, \, \,
\forall f\in C^{\infty}(M),
\end{equation}
where (for $k=1,2$)  
$ \mathcal{L}_{k}^{\ell}\colon \linearizedbasic\to C^{0}(\{U_{i}\})^{\ell}$,
i.e., $\mathcal{L}_{k}^{\ell}(b)$  is a continuous linearized-basic function. 
Here $\density_{\metric}$ denotes the \emph{Riemannian density}, i.e., the measure induced by the Riemannian metric $\metric$.
Recall that $\density_{\metric}$ can be 
 locally  written as a module of a Riemannian volume, i.e., 
$\density_{\metric}=|{\vol_{\metric}}|$ c.f., \cite[Proposition 16.45, Chapter 16]{Lee13}.

This section is organized as follows:
 In Subsection \ref{subsection-natural-setting-Section- Symmetric Criticality of Palais on manifold} 
 we prove  Theorem \ref{proposition-operator-J-kirschoff-generalized} 
and check  that $J$ is critical symmetric  with respect to  AVP control systems $\{\mathcal{C}_{i} \}$ 
 (recall Definitions \ref{definition-a-v-p} and  \ref{definition-C-symmetric-operator});
in Subsection \ref{subsection-Remark-definition-AVP} we stress the relation between AVP and (adapted) Sasaki metrics;
  in Subsection \ref{subsection-few properties-Section- Symmetric Criticality of Palais on manifold} 
	we discuss a few properties of the operator $\mathrm{d} J(b)$;
 in Subsection \ref{subsection-avarage-first-intuition} we define the operator 
$\mathrm{Av}\colon C_{c}^{\infty}(U_{i})\to C_{c}^{\infty}(U_{i})^{\ell}$ 
used in the proof of Theorem \ref{theorem-Palais-principle-variational-formulation};
in Subsection \ref{subsection-l(av(f)=l(f)} we  present and prove    Proposition  \ref{proposition-new-MALEX--Av-l-orbitlike} 
used in the proof of Theorem \ref{theorem-Palais-principle-variational-formulation};

\subsection{A natural setting}
\label{subsection-natural-setting-Section- Symmetric Criticality of Palais on manifold}

In this subsection we will prove  in Lemma \ref{lemma-J-implies-l-linear}  that 
a $C^{1}$-operator  $J$ that fulfills   Equation \eqref{eq: derivative of considered operators}
is in fact  critical $\{\mathcal{C}_{i} \}$-\emph{symmetric}. After that, we will check in Lemma \ref{lemma-derivada-J} 
 that the derivative of the $C^{1}$ operator  
$J_{\lambda}\colon W^{1,p}(M)\to\mathbb{R}$ presented in Theorem \ref{proposition-operator-J-kirschoff-generalized} 
is given by Equation \eqref{eq: derivative of considered operators}. Those facts together with 
Remark \ref{remark-L-ell-restriction}
 will imply the proof of Theorem \ref{proposition-operator-J-kirschoff-generalized}.

\begin{lemma}
\label{lemma-J-implies-l-linear}
Given a function $b\in   \linearizedbasic $
the linear function $\lfunctional= \mathrm{d}J(b)\colon  C_{c}^{\infty}(U_{i})\to \mathbb{R}$ 
defined in Equation \eqref{eq: derivative of considered operators} 
is  $\{\mathcal{C}_{i} \}$-locally symmetric, i.e.
$ \lfunctional(w \circ \efunction^{t X_{\alpha,i}})= \lfunctional(w)$, for each  $w\in C_{c}^{\infty}(U_{i})$ and $\vec{X}_{\alpha,i}\in\mathcal{C}_{i}$.
Therefore, since $J$ is $C^{1}$, we infer that  $\lfunctional= \mathrm{d}J(b)\colon  C_{c}^{\infty}(U_{i})\to \mathbb{R}$ 
is also $\{\mathcal{C}_{i} \}$-locally symmetric for all $b\in \closedlinearizedbasic$. 
\end{lemma}

\begin{proof} 
We start by fixing  $b\in   \linearizedbasic$.  
For a    fixed $i$, set $U=U_{i}$ and   $\varphi:= \efunction^{t X_{\alpha,i}}$.  
We want to check that $ \lfunctional(w\circ \varphi)= \lfunctional(w)$, 
where $w\in C_{c}^{\infty}(U)$.  

First, we claim that in the principal stratum $\Mprincipal$ we have:
\begin{equation}
\label{eq1-5-lemma-J-implies-l-linear}
\mathrm{d}\varphi_{x} \nabla b_{x}=(\nabla b)_{\varphi(x)}. 
\end{equation}

Using the fact that $\varphi$ is an isometry between normal spaces of principal leaves (because $\mathcal{F}$ is AVP) 
 and that  $\nabla b$ is normal to the principal leaves (because $b$ is linearized basic)  
we have, for each vector $Y\in\nu_{x}L_x$, that:   
\begin{align*}
\metric(\mathrm{d}\varphi_{x} \nabla b_{x}, \mathrm{d}\varphi_{x} Y) &= \metric(\nabla b_{x},Y)\\
&=\mathrm{d}(b\circ\varphi^{-1})_{\varphi(x)} \mathrm{d}\varphi_{x} Y\\
&=\metric(\nabla b_{\varphi(x)}, \mathrm{d}\varphi_{x} Y),
\end{align*}
where the last equation follows from the fact that $b$ is linearized basic, i.e., $b\circ\varphi^{-1}=b$.   
The arbitrary choice of $Y$ then implies 
Equation \eqref{eq1-5-lemma-J-implies-l-linear}.

For a $Z_{x}\in\nu_{x}L$ (and recalling  that $w\in C^{\infty}_{c}(U)$), we infer from the chain rule that: 
\begin{equation}
\label{eq1-lemma-J-implies-l-linear}
\metric(Z_{x},\nabla (w\circ \varphi)_{x})= \mathrm{d} w_{\varphi(x)} \mathrm{d}\varphi_{x} Z_x= \metric(\mathrm{d}\varphi_{x} Z_{x}, \nabla w_{\varphi(x)}). 
\end{equation}
Setting $Z_{x}=\nabla b_{x}$,  it follows from 
Equations  \eqref{eq1-lemma-J-implies-l-linear} and  \eqref{eq1-5-lemma-J-implies-l-linear} that 
\begin{equation}
\label{eq3-lemma-J-implies-l-linear}
\metric(\nabla b,\nabla (w\circ \varphi))_{x}= \metric(\nabla b, \nabla w)_{\varphi(x)}.
\end{equation} 
From the fact that $\mathcal{F}$ is AVP foliation we also have:
\begin{equation}
\label{eq4-lemma-J-implies-l-linear}
\varphi^{*}\density_{\metric}=\density_{\metric}.
\end{equation} 
Combining these with the change of variable (see (*)), we infer that:
\begin{align*}
\mathrm{d} J(b)(w\circ\varphi) & = \int_{M} \mathcal{L}_{1}^{\ell}(b)\, \metric \big(\nabla b, \nabla (w\circ\varphi) \big) 
+  \mathcal{L}_{2}^{\ell}(b) (w\circ\varphi)  \density_{\metric} \\
 & \stackrel{\eqref{eq3-lemma-J-implies-l-linear}}{=}  \int_{M} \mathcal{L}_{1}^{\ell}(b)\ g\big(\nabla b, \nabla w \big)_{\varphi}  \density_{\metric} \\
& + \int_{M} \mathcal{L}_{2}^{\ell}(b)  (w\circ\varphi)  \density_{\metric} \\
& \stackrel{\eqref{eq4-lemma-J-implies-l-linear} }{=} 
 \int_{M} \mathcal{L}_{1}^{\ell}(b)
\metric\big(\nabla b, \nabla w \big)_{\varphi}  \varphi^{*}\density_{\metric} \\
& + \int_{M}  \mathcal{L}_{2}^{\ell}(b)  (w\circ\varphi)  \varphi^{*}\density_{\metric} \\
& \stackrel{(*) }{=}   \int_{M}  \mathcal{L}_{1}^{\ell}(b) \metric(\nabla b, \nabla w)  \density_{\metric} \\
& + \int_{M} \mathcal{L}_{2}^{\ell}(b)  w  \density_{\metric} \\
&= \mathrm{d}J(b) w.
\end{align*}
\end{proof}

\begin{lemma}
\label{lemma-derivada-J}
Consider the  $C^{1}$ operator $J_{\lambda}\colon W^{1,p}(M)\to\mathbb{R}$ presented in Theorem \ref{proposition-operator-J-kirschoff-generalized} 
i.e., $J_{\lambda}|_{C^{\infty}(M)}$   satisfying 
\begin{equation*}
J_{\lambda}(u):= \weight\Big(\int_{M} | \nabla u |^{p}\density_{\metric}\Big)\int_M\cal L(|\nabla u|^2,u,x)\density_{\metric} 
- \frac{\lambda}{c}\left(\int_MF(u,x)\density_{\metric}\right)^{r+1},
\end{equation*}
(cf. Equation \eqref{equation-proposition-J-lambda}) where 
\begin{itemize}
\item $r, c, \lambda>0$ are constants  positive and $p\in [2,n[$; 
\item  $F\colon \mathbb{R}\times M\to\mathbb{R}$  and  $\mathcal{L}\colon \mathbb{R}\times\mathbb{R}\times M\to \mathbb{R}$ are   $C^{1}$ and  
$\mathcal{F}$-basic, i.e., $F(s,y)=F(s,x)$ and  $\mathcal{L}(s_1,s_2,y)=\mathcal{L}(s_1,s_2,x)$ $\forall s_1,s_2,s\in\mathbb{R}$,
  and $\forall y\in L_{x}$,  for each fixed $x\in M;$
\item $\weight\colon \mathbb{R}\to\mathbb{R}$ is $C^{1}$. 
\end{itemize}

Then for each $b\in \linearizedbasic$ the linear map  $\mathrm{d} J_{\lambda}(b)$ fulfills Equation \eqref{eq: derivative of considered operators}.

\end{lemma}
\begin{proof}

Recall that, in Theorem \ref{proposition-operator-J-kirschoff-generalized}, 
$\mathcal{L}\colon \mathbb{R}\times\mathbb{R}\times M\to \mathbb{R}$
and $F\colon \mathbb{R}\times M\to\mathbb{R}$ are $C^{1}$ funcions and $\mathcal{F}$-basic
i.e., $\mathcal{L}(s_1,s_2,y)=\mathcal{L}(s_1,s_2,x)$, 
$F(s,y)=F(s,x)$ for $s,  s_1,s_2 \in \mathbb{R}$ and for  $y\in L_{x}$ for all $x\in M$. 
Also $\weight\colon \mathbb{R}\to\mathbb{R}$ is just a  real $C^{1}$ function. 
Setting $\weight_{s}:=\frac{d}{d s}\weight$, $\mathcal{L}_{s_{i}}:=\frac{\partial}{\partial s_{i}} \mathcal{L}$, 
$F_{s}:=\frac{\partial }{\partial s} F$ we have:

\begin{align*}
  \mathrm{d}J_{\lambda}(b)w & =p\,  \weight_{s}(\|b\|^{p})
\big(\int_{M}|\nabla b|^{p-2}\metric(\nabla b,\nabla w) \density_{\metric}\Big)\Big(\int_{M}\mathcal{L}(\metric(\nabla b ,\nabla b),b,x)\density_{\metric}\Big) \\
	& + \weight(\|b\|^{p})  \int_{M} \cal L_{s_1}\big(|\nabla b |^{2},b,x\big) 2 \metric(\nabla b, \nabla w)  \density_{\metric} \\
& + \weight(\|b\|^{p})\int_{M} \cal L_{s_2}\big(|\nabla b|^{2},b,x \big)  w  \density_{\metric} \\
&- \lambda \frac{(r+1)}{c} \left(\int_{M}F(b,x)\density_{\metric} \right)^{r} \int_{M} F_{s}(b,x) w \density_{\metric}.
\end{align*}
We now set:
\begin{align*}
\mathcal{L}_{1}^{\ell}(b)(\cdot)& := p\, \weight_{s}(\|b\|^{p}) 
\Big(\int_{M}\mathcal{L}(\metric(\nabla b,\nabla b),b,x)\density_{\metric}\Big)
|\nabla b(\cdot)|^{p-2}\\
& + 2\weight(\|b\|^{p})\cal L_{s_1}\big( |\nabla b(\cdot) |^{2},b(\cdot),(\cdot)\big) \\
\mathcal{L}_{2}^{\ell}(b)(\cdot) &:= \weight(\|b \|^{p})\cal L_{s_2}\big(|\nabla b(\cdot)|^{2},b(\cdot), (\cdot)\big)\\
& - \lambda \frac{(r+1)}{c} \left(\int_{M}F(b,x)\density_{\metric}\right)^{r} F_{s}(b(\cdot),\cdot)
\end{align*}
Composing $\mathcal{L}$ with curve $t\to \alpha(t)=(s_{1}+t c_{1}, s_{2}+t c_{2},x)$, deriving at $t=0$ and replacing $c_{1}, c_{2}$ with $0,1$ and $1,0$
we conclude that $\frac{\partial }{\partial s_{i}} \mathcal{L}  = \mathcal{L}_{s_{i}}$ is linearized basic.
By the same argument, we can check that $F_{s}$ is also linearized basic. 
Since $|\nabla b|$ is linearized basic, we conclude that
 $\mathcal{L}_{i}^{\ell}(b) $ is a continuous linearized basic function. 
Thus, we have that the operator $J_\lambda$ has a derivative as in Equation \eqref{eq: derivative of considered operators}.
\end{proof}

\begin{remark}
\label{remark-L-ell-restriction}
As can be checked in the proof  of Lemma \ref{lemma-derivada-J}  above,
given a $C^{1}$ operator $J_{\lambda}$ satisfying    Equation \eqref{equation-proposition-J-lambda},
the operators $ \mathcal{L}_{k}^{\ell}\colon \linearizedbasic\to C^{0}(\{U_{i}\})^{\ell}$
(defined in  the proof of the above lemma) admits the following property: \emph{ if 
$b\in C^{\infty}(M)^{\mathcal{F}}$ then  $\mathcal{L}_{i}^{\ell}(b)\in C^{0}(M)^{\mathcal{F}}$, i.e.,
$\mathcal{L}_{i}^{\ell} \Big(C^{\infty}(M)^{\mathcal{F}}\Big)\subset C^{0}(M)^{\mathcal{F}}$. } 
\end{remark}

\subsection{ AVP and Sasaki metrics}
\label{subsection-Remark-definition-AVP}

We have seen in the previous subsection how natural  hypotheses (a) and (b)  are in the definition of AVP. 
In this section, we are going to stress at 
Lemma \ref{lemma-item-c-definitionAVP} 
 the relationship between AVP foliations and Sasaki metrics.
This lemma will   be  used in   the proof of 
Lemma \ref{lemma-relation-integral-linear-functional-subsets-basic} discussed in the next subsection.

First, we need to present an   appropriately adapted metric. 
Recall that, given the principal stratum $\Mprincipal$, there exists a metric $\metric_{\mathcal{F}}$ on $\Mprincipal/\mathcal{F}$ so that the submersion 
$\pi_{\mathcal{F}}\colon (\Mprincipal,\metric) \to (\Mprincipal/\mathcal{F}, \metric_{\mathcal{F}})$ turns out to be a Riemannian submersion.
We say  that a metric  $\hat{\metric}$ on $M$ is an \emph{adapted metric} if $\mathcal{F}$ is still a SRF on $(M,\hat{\metric})$ and 
$\pi_{\mathcal{F}}\colon (\Mprincipal,\hat{\metric}) \to (\Mprincipal/\mathcal{F}, \metric_{\mathcal{F}})$ is a Riemannian submersion. 
In general, this should  also hold in each stratum (and not only in the principal leaves), but since we are dealing with integration, we will not need
to discuss this technical aspect of the adapted metric, and we focus just on the stratum of principal leaves.

\begin{definition}[Adapted Sasaki metric on the principal stratum]
\label{definition-adapted-sasaki-metric} 
Consider a fixed leaf   $B=L_{p_0}$ and  a  tubular neighborhood   $U=\mathrm{Tub}_{\delta}(B)$ that can be  identify 
(via normal exponential map $\exp^{\nu}$) 
with \emph{the disc bundle} $D_{\delta}(0)\to E^{\delta}\to B$, i.e., 
\begin{equation}
\label{definition-E-delta}
 E^{\delta}:=\{e_{p}\in\nu_{p}B; |e_p|<\delta , p\in B\}\subset\nu(B). 
\end{equation}
Let $\mathcal{T}$  be an Ehresmann connection tangent to $\mathcal{F}^{\ell}$ and $\metric^{\tau}$ the Sasaki metric associated to $\mathcal{T}$.
Let  $\mathcal{H}=\nu\mathcal{F}$ denote the normal distribution of $\mathcal{F}$ on the principal stratum $U^{0}$ (with respect to $\metric^{\tau}$).   
Let $\hat{\metric}^{\tau}_{L}$ be the restriction of the Sasaki metric $\metric^{\tau}$ to the tangent spaces $T\mathcal{F}_{U^{0}}$ and
$\hat{\metric}_{\mathcal{H}}$  be the metric on the normal distribution $\mathcal{H}$ so that 
$\mathrm{d}\pi_{\mathcal{F}}\colon (\mathcal{H},\hat{\metric}_{\mathcal{H}})\to \big(T (U^{0}/\mathcal{F}),  \metric_{\mathcal{F}}\big)$
turns out to be an isometry. Set $\hat{\metric}^{\tau}=\hat{\metric}^{\tau}_{L}\oplus\hat{\metric}_{\mathcal{H}}$,
 meaning that $\mathcal{H}$ is orthogonal to the principal leaves with respect to the new metric $\hat{\metric}^{\tau}$. 
This metric $\hat{\metric}^{\tau}$ is called  \emph{adapted Sasaki metric} on the principal stratum $U^{0}$ of the (regular) foliation
$(U^{0},\mathcal{F}_{U^{0}})$. 
\end{definition}

By construction, $\mathcal{F}_{U^{0}}$ is a Riemannian foliation with respect to the metric  $\hat{\metric}^{\tau}$  on $U^{0}$. 
The next lemma stresses that the other two foliations also have this property.

\begin{lemma}
\label{lemma-adapted-sasaki-metric-foliations}
When restricted to their principal stratum, the following (regular) foliations are Riemannian with respect 
 to the  adapted Sasaki metric $\hat{\metric}^{\tau}$ on $U$:
\begin{enumerate}
\item[(a)]   $\mathcal{F}^{\ell}=\{L^{\ell}_{e}\}_{e\in E^{\delta}};$
\item[(b)]the partition into the orbits of isotropic groups, i.e.  $\{G_{p}(e_p)\}_{e_p\in E}$, recall Definition \ref{definition-isotroy-group}.

\end{enumerate}
\end{lemma}
\begin{proof}
It is well known   that  a (regular) foliation  on a manifold $U$ is Riemannian if and only  if for each $x\in U$  there exists
(in a neighborhood of $x$) an orthonormal  frame $\{\vec{\xi}_{i}\}$ of basic vector fields tangent to  the normal distribution of the foliation, 
see \cite{Molino}. The result now follows from the  two observations below: 
\begin{itemize}
\item   $\mathcal{F}$ is Riemannian with respect to the adapted metric $\hat{\metric}^{\tau}$;
\item fixed a principal leaf $L$,   the foliations   $\mathcal{F}^{\ell}=\{L^{\ell}\}$ and $\{G_{p}(e_p)\}_{e_p\in E}$ 
on $L$ are Riemannian foliations with respect to the Sasaki metric $\metric^{\tau}_{L}$.
\end{itemize} 
\end{proof}

\begin{lemma}
\label{lemma-item-c-definitionAVP}
Let $\mathcal{F}=\{L_{x}\}_{x\in M}$ be an AVP  foliation on a compact manifold $(M,\metric)$ and  $B=L_{p_0}$ a fixed leaf, set $U=\mathrm{Tub}_{\delta}(B)$. 
Then, restricting to the principal stratum $U^{0}$, we have:
\begin{equation}
\label{equation-0-lemma-item-c-definitionAVP}
\density_{\metric}=\kappa\, \density_{\hat{\metric}^{\tau}}, 
\end{equation}
where $\density$ is the Riemannian density (i.e., the measure induced by the Riemannian metric),  $\kappa\in C^{\infty}(U^{0})^{\ell}$ and $\hat{\metric}^{\tau}$ is an $\mathcal{T}$-adapted Sasaki metric, for
  an Ehresmann connection $\mathcal{T}$ tangent to $\mathcal{F}^{\ell}$. 
\end{lemma}
\begin{proof}
Consider an  AVP  geometric system  $\mathcal{C}(U)=\{\vec{X}_{\alpha}\}\subset \mathfrak{X}(U)$,
 where the flow $\efunction^{t X_{\alpha}}$ can be assumed to preserve densities $\density_{\metric^{\tau}}$ and 
$\density_{\metric_B}$, see Subsection \ref{subsection-Ex-SRF-Sasaki-AVP}.
 Given the function $\kappa$ so that 
$ \density_{\metric}=\kappa \, \density_{\hat{\metric}^{\tau}}$
our goal is to prove  that
$ \kappa=\kappa\circ \efunction^{t X_{\alpha}} $
for each $X_{\alpha}\in\mathcal{C}(U)$. 

By the definition of flow, it suffices  to check
for small $t$. Hence,  we can prove \eqref{equation-0-lemma-item-c-definitionAVP} locally, writing 
  the Riemannian density
	as a module of local volume Riemannian forms.
	More precisely, for a neighborhood of a point,  we can assume that 
$\density_{\metric}=| \vol_{\metric}|$ and $\density_{\hat{\metric}^{\tau}}=|\vol_{\hat{\metric}^{\tau}}|$,
where $\vol_{\metric}$ and $\vol_{\hat{\metric}^{\tau}}$  are 
 local Riemannian volume forms  with respect to $\metric$ and  $\hat{\metric}^{\tau}$  and $|\cdot|$ is the modulus function,  
 c.f., \cite[Proposition 16.45, Chapter 16]{Lee13}.

The proof of the lemma is divided into three steps.

\textbf{Step 1:} \emph{for each  $\vec{X}_{\alpha}\in\mathcal{C}(U)$ we check that the flow $\efunction^{t X_{\alpha}}$ preserves the Riemannian density 
 of each principal leaf  $L\subset U^{0}$ (with respect to the original metric $\metric$)  i.e}
\begin{equation}
\label{equation-1-lemma-item-c-definitionAVP}
(\efunction^{t X_{\alpha}})^{*}|\vol_{L}|=|\vol_{L}|.  
\end{equation}
In fact,  since $\mathcal{F}$ is a  Riemannian foliation on $U^{0}$, 
there  exists a unique Riemannian metric $\metric_{\mathcal{F}}$ on $U^{0}/\mathcal{F}$
so that 
 $\pi_{\mathcal{F}}\colon (U^{0},\metric ) \to (U^{0}/\mathcal{F},\metric_{\mathcal{F}})$
is a Riemannian submersion, where $\pi_{\mathcal{F}}$ is  the canonical projection. 
Let us denote $|\vol_{ U^{0}/\mathcal{F}}|$ the Riemannian density  associated to $\metric_{\mathcal{F}}$. Then:
\begin{align*}
\Big|\vol_{L}\wedge \pi_{\mathcal{F}}^{*}\big(\vol_{ U^{0}/\mathcal{F}}\big)\Big| &= \Big|\vol_{\metric}\Big|\\
&=\Big|(\efunction^{t X_{\alpha}})^{*}\vol_{\metric}\Big|\\
&=\Big|(\efunction^{t X_{\alpha}})^{*} \vol_{L}\wedge (\efunction^{t X_{\alpha}})^{*}  \pi_{\mathcal{F}}^{*}\big(\vol_{ U^{0}/\mathcal{F}}\big)\Big|\\
&=\Big|(\efunction^{t X_{\alpha}})^{*} \vol_{L}\wedge ( \pi_{\mathcal{F}}\circ \efunction^{t X_{\alpha}})^{*}\big(\vol_{ U^{0}/\mathcal{F}}\big)\Big|\\
&=\Big|(\efunction^{t X_{\alpha}})^{*} \vol_{L}\wedge  \pi_{\mathcal{F}}^{*}\big(\vol_{ U^{0}/\mathcal{F}}\big)\Big|,
\end{align*}
where we have used in the last equality the fact that the flow $\efunction^{t X_{\alpha}}$ is contained in the leaves and hence 
$\pi_{\mathcal{F}}\circ \efunction^{t X_{\alpha}}=  \pi_{\mathcal{F}}$. We have then proved that
\begin{equation}
\label{equation-2-lemma-item-c-definitionAVP}
\Big|\vol_{L}\wedge \pi_{\mathcal{F}}^{*}\big(\vol_{ U^{0}/\mathcal{F}}\big)\Big|=
\Big|(\efunction^{t X_{\alpha}})^{*} \vol_{L}\wedge  \pi_{\mathcal{F}}^{*}\big(\vol_{ U^{0}/\mathcal{F}}\big)\Big|.
\end{equation}
Equation \eqref{equation-2-lemma-item-c-definitionAVP} 
and the fact that $(\efunction^{t X_{\alpha}})^{*} \vol_{L}$ 
is a volume form on $L$ imply Equation \eqref{equation-1-lemma-item-c-definitionAVP}.

\textbf{Step 2:} \emph{For each principal leaf  $L\subset U^{0}$, we need to prove that:}
\begin{equation}
\label{equation-3-lemma-item-c-definitionAVP}
|\vol_{L}|= f |\pi_{L}^{*} \vol_{B}\wedge \vol_{\pi^{-1}_{L}}^{\tau}|\, \mathrm{ \, with \, } \, \, f=f\circ \efunction^{t X_{\alpha}},
\end{equation}
where $\pi_{L}\colon L\to B$ is the restriction to $L$  of the   projection  $\pi\colon E_{\delta}\to B$ and 
$\vol_{\pi^{-1}_{L}}^{\tau}$ is the volume form of the fibers $\{\pi^{-1}_{L}(p)\}_{p\in B}$ with respect to the Sasaki metric
$\metric^{\tau}$. 
In order to prove Equation \eqref{equation-3-lemma-item-c-definitionAVP} define 
$\omega^{\tau}:=\pi_{L}^{*}\vol_{B}\wedge \vol_{\pi^{-1}_{L}}^{\tau}$. 
 Let us check that 
\begin{equation}
\label{equation-4-lemma-item-c-definitionAVP}
(\efunction^{t X_{\alpha}})^{*}|\omega^{\tau}| =|\omega^{\tau}|.
\end{equation}
Note that $ (\efunction^{t X_{\alpha}})^{*}|\vol_{\pi^{-1}_{L}}^{\tau}|=|\vol_{\pi^{-1}_{L}}^{\tau}|$,  because 
with respect to the Sasaki metric $\metric^{\tau}$ linearized flows sends fibers of $E$ to fibers of $E$ isometrically. 
Therefore: 
\begin{align*}
(\efunction^{t X_{\alpha}})^{*}|\omega^{\tau}| &= |(\efunction^{t X_{\alpha}})^{*}\pi_{L}^{*}\vol_{B}\wedge (\efunction^{t X_{\alpha}})^{*}\vol_{\pi^{-1}_{L}}^{\tau}|\\
&= |(\efunction^{t X_{\alpha}})^{*}\pi_{L}^{*}\vol_{B}\wedge \vol_{\pi^{-1}_{L}}^{\tau}|\\
&= |(\pi_{L} \circ \efunction^{t X_{\alpha}})^{*}\vol_{B}\wedge \vol_{\pi^{-1}_{L}}^{\tau}|\\
&= |(\efunction^{t X_{\alpha}}\circ \pi_{L} )^{*}\vol_{B} \wedge \vol_{\pi^{-1}_{L}}^{\tau}|\\
&= |(\pi_{L} )^{*}(\efunction^{t X_{\alpha}})^{*}\vol_{B}\wedge \vol_{\pi^{-1}_{L}}^{\tau}|\\
&= |(\pi_{L} )^{*}\vol_{B}\wedge \vol_{\pi^{-1}_{L}}^{\tau}|\\
&=|\omega^{\tau}|.
\end{align*}
Define the function $f$ on $U^{0}$ so that 
\begin{equation}
\label{equation-5-lemma-item-c-definitionAVP}
|\vol_{L}|=f |\vol_{B}\wedge \vol_{\pi^{-1}_{L}}^{\tau}|=f|\omega^{\tau}|.
\end{equation}
Using local frames, one can check that $f\colon U^{0}\to\mathbb{R}$ is smooth. Moreover, we have
\begin{align*}
f|\omega^{\tau}| &= |\vol_{L}|\\
&\stackrel{\eqref{equation-1-lemma-item-c-definitionAVP}}{=}|(\efunction^{t X_{\alpha}})^{*}\vol_{L}|\\
&=f\circ \efunction^{t X_{\alpha}} |(\efunction^{t X_{\alpha}})^{*}\omega^{\tau}|\\
&\stackrel{\eqref{equation-4-lemma-item-c-definitionAVP}}{=}f\circ \efunction^{t X_{\alpha}} |\omega^{\tau}|,
\end{align*}
and hence  
$f=f\circ \efunction^{t X_{\alpha}}$.

\textbf{Step 3:} \emph{Using Equation \eqref{equation-3-lemma-item-c-definitionAVP} we conclude}
\begin{align*}
|\vol_{\metric}|&= \Big|\vol_{L}\wedge \pi_{\mathcal{F}}^{*}\big(\vol_{ U^{0}/\mathcal{F}}\Big)\Big|\\
& = f \Big| \omega^{\tau}\wedge \pi_{\mathcal{F}}^{*}\big(\vol_{ U^{0}/\mathcal{F}}\big) \Big|\\
&= f |\vol_{\hat{\metric}^{\tau}}|.
\end{align*}
This implies that $\kappa=f$ and hence $\kappa=\kappa\circ e^{tX_{\alpha}}$ as we wanted to prove. 
\end{proof}

\subsection{A few properties of symmetric integral linear functionals}
\label{subsection-few properties-Section- Symmetric Criticality of Palais on manifold}

As discussed in Lemma \ref{lemma-local-to-global-PSC}, we can reduce the study of $\mathrm{d} J$
to locally symmetric  operators, recall Definition \ref{definition-C-symmetric-operator}.
We also recall that we can identify (via normal exponential map) the tubular neighborhood $U=\mathrm{Tub}_{\delta}(L_q)$ of $B=L_q$
with a $\delta$-disc bundle denoted by $E^{\delta}$, recall Definition \ref{definition-E-delta}.

A linear functional $\lfunctional\colon C^{\infty}_{c} (E^\delta) \to \mathbb{R}$ is called \emph{linear integral functional} if
\begin{equation}
\lfunctional(w)=\int_{E^{\delta}} \mathcal{L}^{\ell}(\nabla w, w) \density_{\metric},
\end{equation}
where $\mathcal{L}^{\ell}\colon T E^{\delta}\times\mathbb{R}\to\mathbb{R}$ is a continuous Lagrangian.

\begin{remark}
Our goal  here is to highlight general properties  of symmetric  linear integral functional
that we will need. We hope to use some of these properties 
 in future works; see also Equations \eqref{eq-Fl-zero}, \eqref{equation-lr-decomposto-step3}, \eqref{equation-funcional-linear-gradiante-tangente}. 
We recall to the reader that, in  this article, 
 we are interested in the symmetric integral linear functional
 defined in
 Equation \eqref{eq: derivative of considered operators}.
\end{remark}


For  each relative open set $\vizinhancaB\subset B$, we can denote
 \[
 \lfunctional_{\vizinhancaB}(w)=\int_{\pi^{-1}(\vizinhancaB)} \mathcal{L}^{\ell}(\nabla w, w)\density_{\metric}.
 \]
 Recall that $\lfunctional\colon C^{\infty}_{c} (E^\delta) \to \mathbb{R}$
is  $\mathcal{C}$-\emph{locally symmetric}, if
\begin{equation*}
\lfunctional (w \circ \efunction^{t X_{\alpha}}) = \lfunctional(w),
\end{equation*}
for each $w \in C^{\infty}_{c} (E^{\delta})$ and for each $\efunction^{t X_{\alpha}}$,  where $\vec{X}_{\alpha}\in \mathcal{C}$ and
$\mathcal{C}$ an AVP
 geometric control system.
We will call $\lfunctional$ a  \emph{symmetric integral linear functional}.
If $\lfunctional$ is symmetric, the  
restricted functional $\lfunctional_{\vizinhancaB}$ is still
symmetric. More precisely, for linearized flows $\varphi=\efunction^{t X_{\alpha}}$, we have
\begin{equation}
\label{eq-mudanca-variavel-operadorlinear-restrito}
\lfunctional_{\vizinhancaB}(w\circ \varphi)=\lfunctional_{\varphi(\vizinhancaB)}(w).
\end{equation}
In fact, given a function  $w$ with domain  $\pi^{-1}(\varphi(P))$,
 the domain of  $w\circ\varphi$ is $\pi^{-1}(P)$ and hence we have:
\[
\lfunctional_{\varphi(P)}(w)=\lfunctional(w)=\lfunctional(w\circ\varphi)=\lfunctional_{P}(w\circ\varphi).
\]

Equation \eqref{eq-mudanca-variavel-operadorlinear-restrito} implies the following useful observation, 
which can be checked directly
 in the case of  integral linear functional defined in Equation \eqref{eq: derivative of considered operators}.

\begin{lemma}
\label{lemma-lagrangiano-eh-basico}
Consider $b\in C^{\infty}(E^{\delta})^{\ell}$ a fixed basic linearized function. 
Then $x\to \mathcal{L}^{\ell}(\nabla b(x), b(x))$ is a basic linearized function. 
\end{lemma}

\begin{proof}
Set  $\varphi=\efunction^{t X_{\alpha}}$ for $\vec{X}_{\alpha}\in\mathcal{C}$.  Our goal is to check that 
\begin{equation}
\label{eq-inifitesimal mudanca-variavel-operadorlinear-restrito}
\mathcal{L}^{\ell}(\nabla b, b)\circ\varphi=\mathcal{L}^{\ell}(\nabla b,b)
\end{equation}

\begin{eqnarray*}
\int_{\pi^{-1}(\vizinhancaB)} \mathcal{L}^{\ell}(\nabla b, b)\circ\varphi \, \density_{\metric}
&=&\int_{\pi^{-1}(\vizinhancaB)} \mathcal{L}^{\ell}(\nabla b, b)\circ\varphi \, \varphi^{*} \density_{\metric}\\
&=&\int_{\pi^{-1}(\varphi(\vizinhancaB))} \mathcal{L}^{\ell}(\nabla b, b) \density_{\metric}\\
&=& \lfunctional_{\varphi(\vizinhancaB)}(b)\\
&\stackrel{\eqref{eq-mudanca-variavel-operadorlinear-restrito}}{=} & \lfunctional_{\vizinhancaB}(b\circ \varphi)\\
&=& \lfunctional_{\vizinhancaB}(b)\\
&=& \int_{\pi^{-1}(\vizinhancaB)} \mathcal{L}^{\ell}(\nabla b, b) \, \density_{\metric}
\end{eqnarray*}
The above equation and the arbitrariness in the choice of region $\vizinhancaB$ imply Equation \eqref{eq-inifitesimal mudanca-variavel-operadorlinear-restrito}.
\end{proof}

A useful property of a symmetric integral linear functional is its operation on basic functions.

\begin{lemma} 
\label{lemma-relation-integral-linear-functional-subsets-basic}
Let $\lfunctional$ be a symmetric integral linear functional on $C^{\infty}_{c} (E^\delta)$.  
Given a relative compact neighborhood $\vizinhancaB\subset B$, we have for each linearized basic function $b$:
\begin{equation}
\label{eq-1-lemma-relation-integral-linear-functional-subsets-basic}
\lfunctional(b)=c_{0} \lfunctional_{\vizinhancaB}(b),
\end{equation}
where $c_0=\frac{|B|}{|\vizinhancaB|}$. 

Here $|B|$ and $|\vizinhancaB|$ denote the volume of $B$ and 
$\vizinhancaB$ with respect to the Riemannian density $\density_{\metric}$. 

\end{lemma} 
\begin{proof}
Given a linearized basic function $b$, 
we consider the function $x\to h(x)=\mathcal{L}^{\ell}(\nabla b(x),b(x))$, that is, 
as we have seen at  Lemma \ref{lemma-lagrangiano-eh-basico}, also linearized basic function.
Let $E_{r}$ be a  relatively compact   $\mathcal{F}^{\ell}$-saturated set 
 contained in the principal stratum of $E^{\delta}$ with distance  to  the singular stratum smaller than $r>0$. 
By Lemma \ref{lemma-item-c-definitionAVP},  $\density_{\metric}=\kappa\density_{\hat{\metric}^{\tau}}$ 
where $\kappa\in C^{\infty}(U^{0})^{\ell}$. 
In order to prove Equation \eqref{eq-1-lemma-relation-integral-linear-functional-subsets-basic} it suffices to check that:
\begin{equation}
\label{eq-2-1-lemma-relation-integral-linear-functional-subsets-basic}
\int_{E_{r}} h\, \kappa\density_{\hat{\metric}^{\tau}} = c_{0} \int_{E_{r}\cap \pi^{-1}(\vizinhancaB) } h \kappa\, \density_{\hat{\metric}^{\tau}}. 
\end{equation}
Let $\pi_{\mathcal{F}^{\ell}}$ be the canonical projection whose fibers are the leaves of $\mathcal{F}^{\ell}|_{E_{r}}$.
Item (a) of Lemma \ref{lemma-adapted-sasaki-metric-foliations}  implies  that 
$\pi_{\mathcal{F}^{\ell}}\colon (E_{r},\hat{\metric}^{\tau}) \to (E_{r}/\mathcal{F}^{\ell},\metric_{\mathcal{F}^{\ell}})$ 
is a Riemannian submersion.
On the one hand,
setting  $e^{*}=\pi_{\mathcal{F}^{\ell}}(e)$ for each $e\in E_{r}$, denoting 
$L_{e^{*}}^{\ell}:=(\pi_{\mathcal{F}^{\ell}})^{-1}(e^{*})=L_{e}^{\ell}$  
 and applying Fubini's theorem, we have
\begin{equation}
\label{eq-2-lemma-relation-integral-linear-functional-subsets-basic}
\int_{E_{r}} h(e) \kappa(e) \, \density_{\hat{\metric}^{\tau}} = 
\int_{ E_{r}/\mathcal{F}^{\ell}} h(e^{*})\kappa (e^{*})\, \big(\int_{L_{e^{*}}^{\ell}}  \density_{L_{e}^{\ell}}\big) \density_{\metric_{\mathcal{F}^{\ell}}},
\end{equation}
\begin{equation}
\label{eq-3-lemma-relation-integral-linear-functional-subsets-basic}
\int_{E_{r}\cap \pi^{-1}(\vizinhancaB) } h(e) \kappa(e) \density_{\hat{\metric}^{\tau}} 
 = \int_{ E_{r}/\mathcal{F}^{\ell}} h(e^{*}) \kappa(e^{*})\,  
 \big(\int_{L_{e^{*}}^{\ell}\cap \pi^{-1}(\vizinhancaB)}  \density_{L_{e}^{\ell} }\big) \density_{\metric_{\mathcal{F}^{\ell}}}.
\end{equation}
On the other hand, by item (b) of  Lemma \ref{lemma-adapted-sasaki-metric-foliations},    
each leaf $L_{e}^{\ell}$ admits a Riemannian submersion with  fibers isometric to  the (possible nonconnected) orbit of the isotropy group
  $G_{\pi(e)}(e)$ and basis $B$, i.e.
we have the bundle structure $G_{\pi(e)}(e)\to L_{e}^{\ell}\to B$. This implies that:
\begin{equation}
\label{eq-4-lemma-relation-integral-linear-functional-subsets-basic}
\int_{L_{e}^{\ell}}  \density_{L_{e}^{\ell}}=|G_{\pi(e)}(e)| |B|,
\end{equation}
\begin{equation}
\label{eq-5-lemma-relation-integral-linear-functional-subsets-basic}
 \int_{L_{e}^{\ell}\cap \pi^{-1}(\vizinhancaB)}  \density_{L_{e}^{\ell}}= |G_{\pi(e)}(e)|  |\vizinhancaB|.
\end{equation}
Equation \eqref{eq-1-lemma-relation-integral-linear-functional-subsets-basic} follows direct from Equations
\eqref{eq-2-lemma-relation-integral-linear-functional-subsets-basic}, \eqref{eq-3-lemma-relation-integral-linear-functional-subsets-basic},
\eqref{eq-4-lemma-relation-integral-linear-functional-subsets-basic} and \eqref{eq-5-lemma-relation-integral-linear-functional-subsets-basic}.
\end{proof}

\subsection{The average operator and Lie groupoids}
\label{subsection-avarage-first-intuition}
In the proof of the  
 criticality principle for a  symmetric integral linear functional 
(i.e. Theorem \ref{theorem-Palais-principle-variational-formulation}), 
we have used   an  \emph{average operator} $\mathrm{Av}\colon C^{\infty}_{c} (E^\delta) \to C^{\infty}_{c} (E^\delta)^{\ell}$ 
 that projects  smooth functions with compact support on $E^{\delta}$ (identified via exponential map with $\mathrm{Tub}_{\delta}(B)$) onto $\mathcal{F}^\ell$-basic functions with compact support on $E^{\delta}$.
The idea in the definition of the average operator is to observe that the foliation $\mathcal{F}^{\ell}$ 
comes from a Lie groupoid $\mathcal{G}^{\ell}$ and integrates along the fibers of the source map of this groupoid.  
As we are going to see below, this  generalizes the classical definition of  the average operator associated to an isometric action, see Definition \ref{remark-Av-acao-grupo} and Definition \ref{definition-AV-general}.
We also stress in Remark \ref{remark-why-integration-groupoid} that there are other ways of defining an average operator, depending on the type of problem one wants to address.

Let us start by recalling the definition of Lie groupoid and giving some examples,  increasing the complexity of the examples, 
allowing readers unfamiliar with these concepts to gradually develop an intuition of how they will be used here.

Roughly speaking a
\emph{Lie groupoid} $\mathcal{G} = \mathcal{G}_{1}\rightrightarrows \mathcal{G}_{0}$ consist 
of:
\begin{itemize}
 \item two  manifolds  $\mathcal{G}_{0}$ (the  \emph{set of objects}) and  $\mathcal{G}_{1}$ (\emph{the set of arrows} between objects);
\item two submersions maps  $\mathrm{s},\mathrm{t}\colon \mathcal{G}_{1} \to \mathcal{G}_{0}$ which associate to an arrow $g \in \mathcal{G}_{1}$ 
its \emph{source} (i.e. $\mathrm{s}(g)$) and its \emph{target} (i.e. $\mathrm{t}(g)$) respectively. 
\end{itemize}
We can multiply (or compose) arrows  as long as 
the source of one coincides with the 
the target of the other, and we can   invert arrows as well.
A groupoid also admits  a global section 
$\textbf{1}\colon \mathcal{G}_{0} \to \mathcal{G}_{1}$ called \emph{unit} that satisfies $\mathrm{t}(\textbf{1}(x)) = x = \mathrm{s}(\textbf{1}(x))$, $\mathbf{1}_{x}h = h$, 
$g \mathbf{1}_{x} = g$ (for all $h \in \mathrm{t}^{-1}(x)$ and $g \in \mathrm{s}^{-1}(x)$).

Let us review two simple but important examples to have in mind.

\begin{example}
\label{example-principal-Liegroup}
Let $G$ be a compact Lie group and  $\mu\colon G\times B\to B$ be an action on a  homogeneous
manifold $B=G(p_0)$. In this case 
the arrow set is   $\mathcal{G}_{1}=G\times B;$ the object set is  $\mathcal{G}_{0}=B;$ 
the source map is  $\mathrm{s}(g,p) =p;$  the target map is  $\mathrm{t}(g,p)=\mu(g,p)$
and  the identity map is  $\textbf{1}(p)=(id,p)$. 
\end{example}

\begin{example}
\label{example-transformation-group-action}
Let 
$\mu\colon G\times E\to E$  be a linear  action  on a fiber  bundle $\mathbb{R}^{k}\to E\to B=G(p_0)$, i.e. $\mu(g,\cdot)\colon E_{p}\to E_{\mu(g,p)}$ is a linear map. 
In this case: 
\begin{itemize}
\item The arrow set is   $\mathcal{G}^{\ell}_{1}=G\times E;$ 
\item the object set is  $\mathcal{G}^{\ell}_{0}=E;$ 
\item the source map is  $\mathrm{s}^{\ell} (g,e_p)  =e_p;$ 
\item the target map is  $\mathrm{t}^{\ell}(g,e_p)  =\mu(g,e_p);$ 
\item the identity map is  $\textbf{1}(e_p)=(id,e_p)$.  
\end{itemize}
\end{example}

Before discussing more complex examples and dealing with linearized foliations $\mathcal{F}^{\ell}$ and the 
Linear Holonomy Groupoid $\mathcal{G}^{\ell} = \mathcal{G} \ltimes E$, 
let us present what the Av operator is, at least for the groupoid we have just presented.

\begin{remark}
\label{remark-Av-acao-grupo}
In Example \ref{example-transformation-group-action}  the classical average operator 
  $\mathrm{Av}\colon C_{c}^{\infty}(E^{\delta})\to C_{c}^{\infty}(E^{\delta})^{\ell}$ is defined as:
$
\mathrm{Av}(f)(e_p): = \frac{1}{|G|} \int_{G} f \big( \mu(g, e_p)\big) \, \density
$
where $\density$  is a Riemannian density of a bi-invariant metric $Q$ on the  compact  Lie group $G$. 
This operator can also be written in the groupoid language as follows:
$
\mathrm{Av}(f)(e_p)
=\frac{1}{|G|} \int_{\mathrm{s}^{-1}(p)} f \circ \mu_{e_p} \, \density_p$. 
Here we note that $ \mathrm{s}^{-1}(p)= G\times \{p\}$ 
and therefore the fibers of sources are diffeomorphic to $G$ and hence we can induce a Riemannian metric on them,
such that they turn to be isometric to each other. 
\end{remark}

We can now move on to the formal definition of Lie groupoids. 

\begin{definition} \label{MALEXdefinition-Lie-groupoids}
A \emph{Lie groupoid} $\mathcal{G} = \mathcal{G}_{1} \rightrightarrows \mathcal{G}_{0}$ consist of:
\begin{enumerate}
\item a manifold $\mathcal{G}_{0}$ called the \emph{set of objects};

\item a (possibly non-Hausdorff) manifold $\mathcal{G}_{1}$ called \emph{the set of arrows} (between objects);

\item submersions $\mathrm{s},\mathrm{t}\colon \mathcal{G}_{1} \to \mathcal{G}_{0}$ 
which associate to an arrow $g \in \mathcal{G}_{1}$ its \emph{source} (i.e. $\mathrm{s}(g)$) and its \emph{target} (i.e. $\mathrm{t}(g)$) respectively;

\item a multiplication map $\mathrm{m}\colon \mathcal{G}_{2}\to\mathcal{G}_{1}$, $\mathrm{m} (g,h) = gh$ where 
$\mathcal{G}_{2} = \{(g,h)\in \mathcal{G}_{1} \times \mathcal{G}_{1} \;|\; \mathrm{s}(g) = \mathrm{t}(h)\}$, 
that satisfies $\mathrm{s}(g h) = \mathrm{s}(h)$ and $\mathrm{t}(g h) = \mathrm{t}(g)$;

\item a global section $\mathbf{1}\colon \mathcal{G}_{0} \to \mathcal{G}_{1}$ called \emph{unit} that satisfies $t (\mathbf{1}_x) = x = s(\mathbf{1}_x)$, 
$\mathbf{1}_x h = h$ and $g \mathbf{1}_x = g$ for all $h \in \mathrm{t}^{-1}(x)$ and $g \in \mathrm{s}^{-1}(x)$;

\item a diffeomorphism $i\colon \mathcal{G}_{1} \to \mathcal{G}_{1}$, $i(g)=g^{-1}$ 
called \emph{inverse map} that satisfies $\mathrm{s}(g^{-1}) = \mathrm{t}(g)$, $\mathrm{t}(g^{-1}) = \mathrm{s}(g)$, $g g^{-1} = \textbf{1}_{\mathrm{t}(g)}$,
$g^{-1}g = \textbf{1}_{\mathrm{s}(g)}$.
\end{enumerate}
\end{definition}

Every Lie groupoid $\mathcal{G}$ induces a singular foliation on $\mathcal{G}_{0}$ whose leaves are the connected components of the \emph{orbits} of $\mathcal{G}$, i.e. $\mathcal{G}(x) := \{\mathrm{t} (\mathrm{s}^{-1}(x)) \;|\; x \in \mathcal{G}_{0}\}$.

As proved at \cite{alexandrino2021lie}, the leaves of the linearizable foliation $\mathcal{F}^{\ell}$ 
(defined on a tubular neighborhood $U$ of a leaf $L$ of a SRF)
 are orbits of a Lie groupoid $\mathcal{G}^{\ell}$. 
Very roughly speaking, an arrow of the groupoid $\mathcal{G}^{\ell}$ is an equivalence  class that 
  admits a representative which is a flow $t\to \efunction^{t X}=\varphi_{t}$ of a  linear vector  field $\vec{X}$. 
	It  can  act on a vector $e_p$ of the normal bundle $\nu(L)$, as long as
	 $p\in L$ is contained in the domain of the flow   $\varphi_t$. 
To make this idea formal and compatible with the language of groupoid theory,  we are going to 
introduce the necessary ideas in the next three examples.
As will become clear very soon, Examples \ref{example-principal-groupoid}  and \ref{example-Linear-Holonomy-Groupoid}
will be generalizations of Examples \ref{example-principal-Liegroup}
 and \ref{example-transformation-group-action}.
We will also take this opportunity to introduce the Riemannian density along the fibers of the source map of the 
groupoid $\mathcal{G}^{\ell}$,   which we will need to define the average operator.

\begin{example}[The holonomy groupoid of the lifted foliation]
\label{example-holonomy-groupoid-lifted-foliation}

As in Section \ref{subsection-sasaki-metric-SRF}, let $\mathcal{F}$ be a SRF (with  closed leaves) on $(M,\metric)$. 
For a fixed $p_0$ set $B=L_{p_0}$  and let  $E = \nu(B)$ be the normal bundle of the leaf $B$. 
Consider the Euclidean vector bundle $\mathbb{R}^{k} \to E \to B$ where the metric on each fiber $E_p$ is the metric $\metric_p$ 
restricted to the normal space $\nu_p B$.  
The singular foliations  $\mathcal{F}^{\ell}$ and $\mathcal{F}_{U}$ are foliations on $E$,
via identification of the normal   exponential map.  
Consider   a homothety invariant distribution $\mathcal{T}$ tangents to the leaves
of $\mathcal{F}^{\ell}$  (and hence to $\mathcal{F}_U$ ) that induces a  connection $\nabla^{\tau}$ compatible with the metrics on the fibers.
Denote by $\mathbb{O}(k) \to \mathbb{O}(E) \overset{\tilde{\pi}}{\longrightarrow} B$ the orthonormal frame bundle of $E$ 
and $\widetilde{\mathcal{T}} $ the horizontal distribution on $\mathbb{O}(E)$  
induced by $\mathcal{T}$. We recall that each fiber $\mathbb{O}(E)_p$ of $\mathbb{O}(E)$ can be viewed 
as the set of isometries of $\mathbb{R}^k$ to $E_p$.   
It is also well known that $\mathbb{O}(E)$ is a principal $\mathbb{O}(k)$-bundle and we are going to denote 
the canonical right action of $T \in \mathbb{O}(k)$ on $\mathbb{O}(E)$ by $\mathrm{m}_T\colon \mathbb{O}(E) \to \mathbb{O}(E)$.

Each linearized flow $t\to \efunction^{t X_{\alpha}} $ on $E$ induces a flow $t \to \widetilde{\varphi}_{t}$ on $\mathbb{O} (E)$
by defining the integral curve starting at a frame $\xi_{p}=\{\xi_{i}\}\in \mathbb{O}(E)_p$ as
$\widetilde{\varphi}_{t}(\xi_{p}):=\{ \efunction^{t X_{\alpha}}(\xi_{i})\}$.  Let   $\widetilde{X}_{\alpha}$ denote 
the associated vector  field to $\widetilde{\varphi}_{t}$.  
 Let $\widetilde{\mathcal{F}} = \{\widetilde{L}_{\xi_p}\}_{\xi_p \in O (E)}$ 
be the singular foliation whose leaves are orbits of the geometric control $\{\widetilde{X}_{\alpha}\}$.  
Then
\begin{equation}
\label{eq-1-lemma-groupoid-linear-foliation}
T_{\xi_{p}}\widetilde{L}_{\xi_{p}}=\widetilde{\mathcal{T}}_{\xi_{p}}\oplus T_{\xi_{p}}\big( \mathbb{O}(E)_{p}\cap  \widetilde{L}_{\xi_{p}}  \big).
\end{equation}
The isotropic group $G_{p}^{0}$ induces a free left action on $\mathbb{O}(E)_{p}$ and the orbits of this action 
coincide with the intersection of the leaves of $\widetilde{\mathcal{F}}$ with $\mathbb{O}(E)_{p}$. In particular:
\begin{equation}
\label{eq-2-lemma-groupoid-linear-foliation}
\dim \big( \mathbb{O}(E)_{p}\cap  \widetilde{L}_{\xi_{p}}  \big)=\dim G_{p}^{0}.
\end{equation}
Once $ \dim G_{p}^{0}$ does not depend on $p\in B$,
 we infer from equations \eqref{eq-1-lemma-groupoid-linear-foliation} and \eqref{eq-2-lemma-groupoid-linear-foliation} that the foliation $\widetilde{\mathcal{F}}$ 
is a (regular) foliation.

Let $\mathrm{Hol}(\widetilde{\mathcal{F}}) \rightrightarrows \mathbb{O}(E)$ 
be the \emph{holonomy groupoid} of the foliation $\widetilde{\mathcal{F}}$,
i.e., the Lie groupoid over $ \mathbb{O}(E)$ whose arrows are holonomy class of paths contained in the leaves (cf \cite[Chapter 5]{MM03}). 
Since $\mathcal{F}$ has trivial holonomy (see Lemma \ref{Properties-Lifted-Foliation}),  
one can provide an alternative  definition of this Lie groupoid as  follows:
\begin{enumerate}
\item[(1)] \emph{the objects} are  frame $\xi=\{\xi_{i}\}\in\mathbb{O}(E);$
\item[(2)] \emph{arrows} are equivalence classes $\tilde{g}=[\widetilde{\varphi},\xi]$, where the equivalence relation is defined as:
$(\widetilde{\varphi},\xi)\sim (\widetilde{\varphi}^{'},\xi^{'})$ iff $\xi=\xi^{'}$ and $\widetilde{\varphi}(\xi)=\widetilde{\varphi}^{'}(\xi^{'});$
\item[(3)] \emph{source and target maps} are $\tilde{\mathrm{s}}([\widetilde{\varphi},\xi])=\xi$ 
and $\tilde{\mathrm{t}}([\widetilde{\varphi},\xi])=\widetilde{\varphi}(\xi);$
\item[(4)] the \emph{multiplication} is 
$\tilde{\mathrm{m}}([\widetilde{\varphi}_{\beta}, \widetilde{\varphi}_{\alpha}(\xi^{\alpha})],[\widetilde{\varphi}_{\alpha},\xi^{\alpha}])
=[\widetilde{\varphi}_{\beta}\circ\widetilde{\varphi}_{\alpha},\xi^{\alpha}];$
\item[(5)] the \emph{unit} is  $\mathbf{1}_{\xi}=[id,\xi];$
\item[(6)] the \emph{inverse} is $[\widetilde{\varphi},\xi]^{-1}=[\widetilde{\varphi}^{-1},\widetilde{\varphi}(\xi)]$. 
\end{enumerate}
\end{example}

The next lemma will be proved at Lemma \ref{Properties-Lifted-Foliation}.

\begin{lemma}
\label{lemma-metric-on-fibers-tildeF}
Consider the (regular) lifted foliation $\widetilde{\mathcal{F}}$ on $\mathbb{O}(E)$ defined in Example \ref{example-holonomy-groupoid-lifted-foliation}.
Then there exists a metric $\tilde{\metric}$ along the leaves of $\widetilde{\mathcal{F}}$
(i.e., a semi-Riemannian metric on $\mathbb{O}(E)$ associated to the distribution $T\widetilde{\mathcal{F}}$)
that satisfies the following properties:
\begin{itemize}
\item $\tilde{\metric}$  is  invariant by the canonical  $\mathbb{O}(k)$-right action on $\mathbb{O}(E);$ 
 \item the orbits of the left action of isotropy groups are isometric to each other.
\end{itemize}
 In particular, 
this semi-Riemannian metric $\tilde{\metric}$ induces a
Riemannian metric (also denoted by  $\tilde{\metric}$) 
on the fibers of the source map of the Lie groupoid $\mathrm{Hol}(\widetilde{\mathcal{F}})$. 
\end{lemma}

Now we can move to the next example. 

\begin{example}[The \emph{ ``principal"} Lie groupoid $\mathcal{G}$]
\label{example-principal-groupoid}
The foliation $\widetilde{\mathcal{F}}$ defined in Example \ref{example-holonomy-groupoid-lifted-foliation}
 is invariant by the  $\mathbb{O}(k)$-right action on $\mathbb{O}(E)$. 
This action can be lifted to a proper  free right  action on the holonomy groupoid 
$\mathrm{Hol}(\widetilde{\mathcal{F}})\times \mathbb{O}(k)\to \mathrm{Hol}(\widetilde{\mathcal{F}})$
and hence
\[
\xymatrix{
\mathrm{Hol}(\widetilde{\mathcal{F}}) \ar@<0.25pc>[r] \ar@<-0.25pc>[r] \ar[d] & \mathbb{O}(E) \ar[d] \\
\mathrm{Hol}(\widetilde{\mathcal{F}})/\mathbb{O}(k) \ar@<0.25pc>[r] \ar@<-0.25pc>[r] & \mathbb{O}(E)/\mathbb{O}(k) = B.
}
\]

The  Lie groupoid  $\mathcal{G} := \mathrm{Hol}(\widetilde{\mathcal{F}})/\mathbb{O}(k) \rightrightarrows B$ is characterized by:
\begin{enumerate}
\item[(1)] \emph{objects} being points of $B;$
\item[(2)] \emph{arrows} being the lateral classes $\bar{g}=\tilde{g}\, \mathbb{O}(k)=[\widetilde{\varphi},\xi]\,\mathbb{O}(k);$
\item[(3)]\emph{source/target maps} being  $\mathrm{s}\big([\widetilde{\varphi},\xi_{p}]\, \mathbb{O}(k)\big)=p$ and  
$\mathrm{t}\big([\widetilde{\varphi},\xi_{p}] \,\mathbb{O}(k)\big)=\varphi(p);$
\item[(4)]  \emph{multiplication map} being 
$\mathrm{m}\big(\tilde{g}_{\beta}\, \mathbb{O}(k),\tilde{g}_{\alpha}\, \mathbb{O}(k)\big)= 
\tilde{\mathrm{m}}\big(\tilde{g}_{\beta} \tilde{g}_{\alpha}\big)\, \mathbb{O}(k);$
\item[(5)] \emph{unit map} being  $\mathbf{i}_{p}=[id,\xi_{p}]\, \mathbb{O}(k);$
\item[(6)]\emph{inverse} being $\big(\tilde{g}\, \mathbb{O}(k)\big)^{-1}=\tilde{g}^{-1}\, \mathbb{O}(k)$. 
\end{enumerate}

We finish this example stressing that the fibers of the source map $\mathrm{s}:
 \mathrm{Hol}(\widetilde{\mathcal{F}})/\mathbb{O}(k) \rightrightarrows B$ admit also  Riemannian metrics $\bar{\metric}$.

In  fact, the $\mathbb{O}(k)$-right action on $\mathrm{Hol}(\widetilde{\mathcal{F}})$
 sends fibers to the fibers of the source map and does not fix  them. More precisely: 
\[
\big(\tilde{\mathrm{s}}^{-1}(\xi_{p})\big)\cdot g = \tilde{\mathrm{s}}^{-1}(\xi_{p}\cdot g). 
\]
Therefore the metric $\tilde{\metric}$ 
(defined at Lemma \ref{lemma-metric-on-fibers-tildeF}) on $\tilde{s}^{-1}(\xi_p)$ 
descends to a metric $\bar{\metric}$ on $s^{-1}(p)$ 
via pushforward by $\rho |_{\tilde{s}^{-1}(\xi_p)}$, which becomes an isometry. 
Here $\rho\colon\mathrm{Hol}(\widetilde{\mathcal{F}}) \to  \mathrm{Hol}(\widetilde{\mathcal{F}})/\mathbb{O}(k) $  is the canonical submersion.

We obtain an isometry 
$\psi_\xi\colon (s^{-1} (p),\bar{\metric})\to (\tilde{L}_{\xi_p},\tilde{\metric})$, 
which makes the following diagram commute: 

\begin{equation}
\label{diagram-metric-source-s}
\vspace{\baselineskip}
\hspace{0.5cm}
\begin{tikzcd}[column sep={1.5cm}, row sep={1.5cm}]
\tilde{s}^{-1}(\xi_{p}) \arrow{r}{\tilde{t}} \arrow[swap]{d}{\rho\hspace{0.25ex}} & \tilde{L}_{\xi_p} \\
s^{-1} (p) \arrow[swap]{ur}{\psi_{\xi_p}} &
\end{tikzcd}.
\vspace{\baselineskip}
\end{equation}
\end{example}


\begin{example}[Linear Holonomy Groupoid]
\label{example-Linear-Holonomy-Groupoid}
The Lie groupoid $\mathcal{G} := \mathrm{Hol}(\widetilde{\mathcal{F}})/\mathbb{O}(k) \rightrightarrows B$
defined in Example \ref{example-principal-groupoid}
 comes with a canonical representation on $E=\mathbb{O}(E)\times_{\mathbb{O}(k)} \mathbb{R}^k$.  In  fact, 
for $e_{p}=[\xi_{p},v]\in E_{p}$  we define: 
$\mu(\bar{g}_p, e_{p})=\mu(\bar{g}_p, [\xi_p, v]) := [\tilde{\mathrm{t}}(\tilde{g}),v]$, where $\tilde{g}$ 
is the unique representative of $\bar{g}$ in $\mathrm{Hol}(\widetilde{\mathcal{F}})$ such that $\tilde{s}(\tilde{g}) = \xi_p$.
The  \emph{Linear Holonomy Groupoid}  $\mathcal{G}^{\ell}:=\mathcal{G}\ltimes E \rightrightarrows  E$  
is defined by the  following data: 
\begin{enumerate}
\item[(1)] \emph{objects} being vectors of $E;$
\item[(2)]\emph{arrows}  being  pairs $(\bar{g}_p,e_{p})\in \mathcal{G}\times E$ where   
$\mathrm{s} ({\bar{g}}_{p})=p=\pi(e_p);$
\item[(3)] \emph{source and target maps} being $\mathrm{s}^{\ell}\big((\bar{g}_p,e_p)\big)=e_p$ 
 and $\mathrm{t}^{\ell}\big( (\bar{g}_p,e_p) \big)=\mu(\bar{g}, e_p);$
\item[(4)] \emph{multiplication mpa } being  
$\mathrm{m}^{\ell}\big((\bar{g}_{q},\bar{g}_{p}\cdot e_{p} ),(\bar{g}_{p},e_{p}) \big)=
(\mathrm{m}(\bar{g}_{q},\bar{g}_{p}),e_{p});$
\item[(5)] \emph{unit map} being $\mathbf{i}^{\ell}_{e_p}=(\mathbf{i}_{p},e_{p});$
\item[(6)] \emph{inverse}  being $(\bar{g}_{p},e_{p})^{-1}=(\bar{g}^{-1}, \mu(\bar{g}_{p}, e_{p}))$. 
\end{enumerate}

\end{example}

\begin{remark} \label{MALEX-remark-groupoid-volume-preserve}
The construction of the  above groupoids, in particular  the Lie linear holonomy groupoid,
can be done using only linearized flows that preserve the Sasakian volume of $E$.
\end{remark}

From the constructions presented in  Example  \ref{example-principal-groupoid} and Example \ref{example-Linear-Holonomy-Groupoid},
there is a natural relationship between the fibers of 
 the source map $\mathrm{s}:  \mathrm{Hol}(\widetilde{\mathcal{F}})/\mathbb{O}(k) \rightrightarrows B$ 
and the fibers of the  source map $\mathrm{s}^{\ell}\colon \mathcal{G}^{\ell}\to   E$  given by  the equation below.
\begin{equation}
\label{equation-remark-relation-sourcemaps-lineargroupoid-principalgroupoid}
(\mathrm{s}^{\ell})^{-1}(e_p) = \mathrm{s}^{-1}(p) \times \{e_p\}.
\end{equation}

\begin{definition}
We define a Riemannian metric along the fibers of the  source  map $\mathrm{s}^{\ell}: \mathcal{G}^{\ell}\to   E$ so that
the difeomorphism presented in Equation \eqref{equation-remark-relation-sourcemaps-lineargroupoid-principalgroupoid} is an isometry. 
We can denote this metric also as  $\tilde{\metric}$ since the  
$(\mathrm{s}^{\ell})^{-1}(e_p)$ can be identified with a leaf $\widetilde{L}$ of the lifted foliation on 
$\mathbb{O}(E)$. 
\end{definition}


We can now present the main definition of this section, 
generalizing the classical average map 
associated to group action on a vector fiber bundle, recall Remark \ref{remark-Av-acao-grupo}.

\begin{definition}
\label{definition-AV-general}
The \emph{average operator} $\mathrm{Av}\colon C^{\infty}_{c}(E)\to  C_{c}^{\infty}(E)^{\ell}$ is defined as:
\begin{equation*}
\mathrm{Av}(f)(e_p) =\frac{1}{\mathcal{V}(p)} \int_{(\mathrm{s}^{\ell})^{-1}(e_p)} f \circ \mathrm{t}^{\ell} \, \density^{\ell}_{e_p} 
=\frac{1}{\mathcal{V}(p)} \int_{\mathrm{s}^{-1}(p)} f \circ \mu_{e_p} \, \bar{\density}_p,
\end{equation*}
where:
\begin{itemize}
 \item $\mathrm{s}^{\ell}$,  $\mathrm{t}^{\ell}$  are the source and target map of 
the Linear Holonomy Groupoid $\mathcal{G}^{\ell} = \mathcal{G} \ltimes E$ and  $\mathrm{s}$ the source map of the principal groupoid $\mathcal{G}$;
\item $\density^{\ell}_{e_p}$  is a Riemannian density of a metric $\tilde{\metric}$ on $(\mathrm{s}^{\ell})^{-1} (e_p)$
and $\mathcal{V}(p)=|(s^{\ell})^{-1}(p)|$ is the volume of the fibers (that we will prove  is constant). 
\end{itemize} 
In the second equality,  we used Equation \eqref{equation-remark-relation-sourcemaps-lineargroupoid-principalgroupoid}
\end{definition}

We postpone to Lemma  \ref{Av-bem-definido} the fact that  $\mathrm{Av}(f)$ is  smooth and $\mathcal{V}(p)$ is constant.

\begin{remark}
\label{remark-why-integration-groupoid}
We should highlight that it was introduced in \cite{LytchakRadeschi}  another very important way of calculating an average operator with respect to SRF   
and this other average operator was used  in relevant problems, see  e.g., \cite{Mendes-Radeschi-quadradic}.
More precisely,  
Lytchak and Radeschi's average  operator integrates over  the leaves, i.e., 
$Av(f)(x)=\frac{1}{|L_x|}\int_{L_x} f\density_{\metric}$ (c.f., also Savo \cite[p.156]{Savo-18})
while our operator  integrates over a  groupoid. 
We believe it seems more convenient in our work to integrate over a groupoid because
  we are dealing with operators that are symmetric with respect to the geometric control (induced by the algebroid of the groupoid). 
In particular, two actions of the same group $G$  (one isometric action and the other not) 
could be orbit equivalent, but the operator could be symmetric with respect to one action but not the other.
We hope that this technical issue will become clearer in the proof of 
 Proposition   \ref{proposition-new-MALEX--Av-l-orbitlike} step 5. 
We also want to emphasize that our average operator is a generalization of the operator studied by E. Park and K. Richardson. In this work, the authors considered the Average for SRFs, which are the closures of regular Riemannian foliations, see \cite[p.1253]{Park-Richardson}. 
\end{remark}

\subsection{Relation between symmetric linear functional  and the average operator }
\label{subsection-l(av(f)=l(f)}

The next proposition is central to understanding the principle of symmetric criticality in vector bundles
i.e. Theorem \ref{theorem-Palais-principle-variational-formulation}.

\begin{proposition}
\label{proposition-new-MALEX--Av-l-orbitlike}
Let $\mathcal{F}$ be an AVP foliation on a  compact (connected) Riemannian manifold $(M,\metric)$ and
$\{U_{i}\}$ a finite cover of $M$ with tubular neighborhoods $U_{i}=\mathrm{Tub}_{\delta}(L_{i})$. 
Consider a   $C^{1}$ operator  $J\colon W^{1, p}(M) \to \mathbb{R}$  such that for each smooth basic function $b\in \linearizedbasic$ 
\begin{equation}
\label{eq2: derivative of considered operators}
\mathrm{d} J(b)f=\int_{M}\big( \mathcal{L}_{1}^{\ell}(b)\, \metric (\nabla b, \nabla f) +  \mathcal{L}_{2}^{\ell}(b) f\big) \density_{\metric}, \, \,
\forall f\in C^{\infty}(M),
\end{equation}
where (for $k=1,2$)  
$ \mathcal{L}_{k}^{\ell}\colon \linearizedbasic\to C^{0}(\{U_{i}\})^{\ell}$,
i.e., $\mathcal{L}_{k}^{\ell}(b)$  is a continuous linearized basic function. 
Let $B$ be a fixed leaf $L_{i}$ and  identify $U_{i}$ (via the normal exponential map) with 
the normal (disc) bundle $D_{\delta} (0)\to E^{\delta}\to B$, recall Equation \eqref{definition-E-delta}.
For $b_{0}\in \closedlinearizedbasic$ set 
$\lfunctional(f)=\mathrm{d}J(b_{0})(f)$ the restriction of $d J(b_0)$ on $C^{\infty}_{c} (E^\delta)$.
Then for each $f \in C^{\infty}_{c} (E^\delta)$ we have 
\begin{equation}
\label{equation-1-new-MALEX-lemma-Av-l-orbitlike} 
 \lfunctional \left( \mathrm{Av}(f) \right) =   \lfunctional \left( f \right).
\end{equation}
\end{proposition}

\begin{remark}
For the reader who decides to skip the proof of the result, we can at least comment on the idea of the proof.
The  main idea is that when we   are  dealing with the function  $\mathrm{Av}(f)$, 
instead of considering the integration along (different) fibers  of the source submersions,
 we consider the integration along a fixed manifold $\widetilde{L}$ (contained in the orthonormal frame bundle $\mathbb{O}(E)$) 
of a  function $(\tilde{q},e_{p})\to h(\tilde{q},e_p)$. The integral 
$\mathcal{V}\mathrm{Av}(f)(e_p) =\int_{\widetilde{L}} h(\cdot, e_p)\density_{\tilde{\metric}}$  is then approximated by a Riemann sum. 
Finally, we can apply the symmetric linear functional  $\lfunctional_{\metric}$ to this finite sum. 
\end{remark}

\begin{proof}
By partition of unity, it suffices to prove that 
 for each $p_0\in B$ there exists a neighborhood $\vizinhancaB\subset B$ so that for each $f\in C_{c}^{\infty}(E^{\delta})$ 
with $\mathrm{supp } f\subset \vizinhancaB$ we have Equation \eqref{equation-1-new-MALEX-lemma-Av-l-orbitlike}.
Also since $J$ is $C^{1}$ it suffices to check 
the equation for $\lfunctional=\mathrm{d} J(b)$ for smooth functions $b\in \linearizedbasic$.

We present a proof that holds for the general case of our linear holonomy groupoid $\mathcal{G}^{\ell}$, but the reader can think about the particular case of 
Example \ref{example-transformation-group-action} if it helps to get intuition.

\subsubsection{
\textbf{Step 1:} 
Recall from Example~\ref{example-principal-groupoid} that there exists a metric $\bar{\metric}$ 
along the fibers of the source map $\mathrm{s}\colon \mathcal{G}\to B$. 
This metric induces a Riemannian density $\bar{\density}$   on the fibers of $\mathrm{s}\colon \mathcal{G}\to B$, 
which appears in the definition of the average operator:
 }

\begin{equation}
\label{equation-1-5-new-MALEX-lemma-Av-l-orbitlike} 
\mathrm{Av}(f)(e_p)  =\frac{1}{\mathcal{V}(p)} \int_{\mathrm{s}^{-1}(p)} f \circ \mu_{e_p}  \bar{\density}_p, \,   \,  \,  \forall  e_p\in \pi^{-1}(P).
\end{equation}
Also  as explained in Example \ref{example-principal-groupoid}, 
  we have  an isometry 
$\psi_{\xi_p}\colon (s^{-1} (p),\bar{\metric})\to (\tilde{L}_{\xi_p},\tilde{\metric})$, 
which makes the diagram  \eqref{diagram-metric-source-s} commutative. 
In Lemma \ref{s-fiber-isometric-lifted-leaf} we prove that it is possible to locally collect all the isometries $\psi_{\xi_p}$. 
More precisely, fixed  a leaf $\widetilde{L} \in \widetilde{\mathcal{F}}$ and let  $\xi \in \Gamma(\mathbb{O}(E)_P)$ 
\textcolor{blue}{be a local} orthonormal frame into $\widetilde{L}$, there exists  $\Psi\colon s^{-1}(P) \to \widetilde{L}$ such that:
\begin{itemize}
\item  $\tilde{\pi} \circ \Psi = t$ and $\Psi \circ \textbf{1} = \xi$;
\item $\Psi$ is a submersion and $\psi_{\xi(p)}=\Psi|_{s^{-1}(p)} =: \psi_{p}$ is an isometry, for all $p \in P$;
\end{itemize}

This map allows  us to rewrite Equation  \eqref{equation-1-5-new-MALEX-lemma-Av-l-orbitlike} as an integral
 over a fixed leaf   $\widetilde{L}$ as we see in the  next step.

\subsubsection{
\textbf{Step 2:} For $p\in \vizinhancaB$, we write  the integrals along the pre-images of the source map of $\mathcal{G}^{\ell}$
 as an integral along an open set on the fixed leaf $\widetilde{L}\subset \mathbb{O}(E)$ }

Let $\widetilde{\pi}\colon \widetilde{L}\to B$ be the footpoint projection
and $\widetilde{\vizinhancaB}\subset\widetilde{L}$ 
the lift of $\vizinhancaB$ i.e. $\widetilde{\vizinhancaB}=\widetilde{\pi}^{-1}(\vizinhancaB)$. 
Note that  for $p\in \vizinhancaB$ we have:
\begin{equation}
\label{equation-step3-avarage-sketch}
\mathcal{V}\mathrm{Av}(f) (e_p)  = \int_{\widetilde{L}}  f\circ \mu_{e_p}\circ(\psi_{p})^{-1}  \density_{\tilde{\metric}}
 = \int_{\widetilde{\vizinhancaB}}    f\circ \mu_{e_p}\circ(\psi_{p})^{-1} \density_{\tilde{\metric}}, 
\end{equation}
where 
$\mu_{e_p}\colon s^{-1}(p)\to E^{\delta}$ is defined as $\mu_{e_p}(g)=\mu(g,e_{p})$

\subsubsection{\textbf{Step 3:} Adapted metric $\hat{\metric}$ on regular stratum $E_{r}$  }

Let $E_{r}$ be a saturated open set contained in the principal
stratum of $E^{\delta}$ with distance to the singular stratification
lower than $r>0$. We consider the adapted Sasaki metric  $\hat{\metric}^{\tau}$, see Definition \ref{definition-adapted-sasaki-metric}.
To avoid cumbersome notations, let us denote for the rest of this proof $\hat{\metric}^{\tau}$ by $\hat{\metric}$.  
The construction  implies that the AVP system preserves the normal space with respect to 
$\hat{\metric}$.

Given a basic function  
$b\in C_{c}^{\infty}(E^{\delta})^{\ell}$,
let  us denote by $\nabla b$ and $\widehat{\nabla}b$ the gradients with respect
to the original metric $\metric$ and the adapted metric $\widehat{\metric}$ respectively.  Let us define the difference vector field: 
\begin{equation}
\label{equation-diference-between-gradients}
\mathfrak{z}=\widehat{\nabla}b - \nabla b.
\end{equation}

Let us denote by $\lfunctional^{r}\colon C_{c}^{\infty}(E_{r})\to\mathbb{R}$ the operator $\lfunctional$ 
restricted to  $C_{c}^{\infty}(E_{r})$.  We have then: 
\begin{align*}
 \lfunctional^{r}(f) & = \mathrm{d}J(b)(f)\\
&= \int_{E_{r}} \mathcal{L}_{1}^{\ell}(b)\metric(\nabla b,\nabla f) \density_{\metric} \\
&+ \int_{E_{r}} \mathcal{L}_{2}^{\ell}(b) f   \density_{\metric}. 
\end{align*}
Note that
\begin{align*}
\hat{\metric}(\widehat{\nabla} b, \widehat{\nabla} f ) &= df (\widehat{\nabla} b)\\
&= df(\nabla b)+\mathrm{d}f(\mathfrak{z})\\
&=\metric(\nabla b, \nabla f) +\mathrm{d}f(\mathfrak{z}).
\end{align*}
This allow us to describe \emph{the linear operator $\lfunctional^{r}$ 
with respect to the adapted metric $\hat{\metric}$ on $E_{r}$}.
\begin{align*}
\lfunctional^{r}(f)
 & =\int_{E_{r}} \mathcal{L}_{1}^{\ell}(b)\hat{\metric}(\widehat{\nabla} b, \widehat{\nabla} f ) \density_{\metric}\\
&+ \int_{E_{r}}\mathcal{L}_{2}^{\ell}(b)  f  \density_{\metric}\\
&-\int_{E_{r}} \mathcal{L}_{1}^{\ell}(b) (\mathfrak{z}\cdot f) \density_{\metric}.
\end{align*}

It is also convenient (for the next step)  to fix some notation.

We start by defining  the Lagrangian (recall Equation \eqref{definition-E-delta})
$
\widehat{\mathcal{L}}^{\ell}\colon T E_{r}\times \mathbb{R} \to \mathbb{R}
$
\begin{equation}
\widehat{\mathcal{L}}^{\ell}(V_{x}, t) = 
\mathcal{L}_{1}^{\ell}(b)(x)\hat{\metric}(\widehat{\nabla} b(x), V_{x})  + \mathcal{L}_{2}^{\ell}(b)(x)\, t
\end{equation}

Note that 
\begin{equation}
\label{eq-Fl-zero}
\widehat{\mathcal{L}}^{\ell}(0_{x}, 0)=0, \, \forall x\in E_{r}.
\end{equation}

We also need to define the operators $\lfunctional_{i}\colon C_{c}^{\infty}(E_{r})\to\mathbb{R}$ as:
\begin{equation}
\hat{\lfunctional}_{1}(f) :=\int_{E_{r}}  \widehat{\mathcal{L}}^{\ell}(\widehat{\nabla}f, f) \density_{\metric},
\end{equation}
\begin{equation}
\hat{\lfunctional}_{2}(f) :=\int_{E_{r}}  - \mathcal{L}_{1}^{\ell}(b) \, \mathfrak{z}\cdot f \, \density_{\metric}
\end{equation}

Therefore 
\begin{equation}
\label{equation-lr-decomposto-step3}
\lfunctional^{r}=\hat{\lfunctional}_{1}+\hat{\lfunctional}_{2}.
\end{equation}
 Note that the fact that  $b$ is a basic function implies:
\begin{equation}
\label{equation-funcional-linear-gradiante-tangente}
\hat{\lfunctional}_{1}(f)=\int_{E_{r}} \widehat{\mathcal{L}}^{\ell}(\widehat{\nabla}^{E} f, f) \density_{\metric},
\end{equation}
where $\widehat{\nabla}^{E} f$ is the tangent part of $\widehat{\nabla} f$
to the fibers of $E$.


\subsubsection{
\textbf{Step 4:} We consider a partition $\{\vizinhancaB_{i}\}$ of $\vizinhancaB$ and for each fixed 
$i_{\scriptscriptstyle{0}}$ we rewrite,  for $p\in \vizinhancaB_{i_0}$ 
the integral in Equation \eqref{equation-step3-avarage-sketch} as a Riemann sum.
}

We start step 4 with a 
 coordinate system $(p_{1}(\cdot),\cdots p_{m}(\cdot))$ of $\vizinhancaB$ (recall  $\mathrm{supp } f\subset \vizinhancaB$)   
so that the pullback of  the Euclidian volume $\mathrm{d}p_{1}\wedge \cdots \wedge \mathrm{d}p_{m}$
is the volume in $B$ (recall \cite[p.6 Example 2.3]{Kob95}), we identify  $\vizinhancaB$  with a $m$-retangle of $\mathbb{R}^{m}$. 
Now we consider a partition of $\{\vizinhancaB_i\}_{i=1}^{m_1}$ so that $|\vizinhancaB_{i}|=|\vizinhancaB_{i_{\scriptscriptstyle{0}}}|$ 
 for all $i\in \{1\cdots m_{1}\}$.
For each fixed $i$ we consider  on $\widetilde{\vizinhancaB}_{i}:=\tilde{\pi}^{-1}(\vizinhancaB_{i})\subset \widetilde{L}$
a partition  into connected components $\{\widetilde{\vizinhancaB}_{i j}\}_{j=1}^{m_2}$.  
By Setting  $V_{i j}=|\widetilde{\vizinhancaB}_{i j}|$ we can check by Fubini's theorem
that 
\begin{equation}
\label{equation-step3-igual-volumeVij}
V:=V_{i j}=|G_{p_{0}}^{0}| |\vizinhancaB_{i_{0}}|, 
\end{equation}
where $G_{p_0}^{0}$ is the connected component of isotropy  group, recall Definition \ref{definition-isotroy-group} and  
Remark \ref{remark-Connected-component-isotropy}.

In each ``band'' $\widetilde{\vizinhancaB}_{i j}$ we   consider a small partition 
$\{ \widetilde{\vizinhancaB}_{i j k} \}_{k=1}^{m_3}$.
By setting $V_{i j k}=|\widetilde{\vizinhancaB}_{i j k}|$ we can observe that $\sum_{k} V_{i j k}= V_{i j}$.

We also need to say a few words about how small each $\widetilde{\vizinhancaB}_{i j k}$ should be.  
 Given an $\epsilon>0$, 
we can find $\delta_1$ so that for 
$|t_{1}|_{\hat{\metric}^{\tau}}<\delta_1$ (the norm with respect to $\hat{\metric}^{\tau})$, $|t_{2}|<\delta_{1}$, (the modulus  function) we have: 
\begin{equation}
\label{equation-step3-lagrangean-lower-E}
|\widehat{\mathcal{L}}^{\ell}(t_1,t_2)|< \frac{\epsilon}{2 |\pi^{-1}(\vizinhancaB)|}.
\end{equation}
Set  
\begin{equation}
\label{eq-step4-defh}
h(\tilde{q},e_p)=f\circ \mu(\psi_{p}^{-1}(\tilde{q}),e_p). 
\end{equation}

Using the uniform continuity of $h$ and its gradient $\widehat{\nabla} h$  restricted to the relatively  compact sets $\widetilde{\vizinhancaB}$ 
and $\pi^{-1}(\vizinhancaB)$,
we can find  $\delta_{2}>0$, so that if $\mathrm{diam}(\widetilde{\vizinhancaB}_{i j k})<\delta_{2}$ 
and $\tilde{q}, \tilde{y} \in \widetilde{\vizinhancaB}_{i j k}$ then:
\begin{equation}
\label{eq-ste3-uniforme-continua-h}
|h(\tilde{q},e_{p})-h(\tilde{y},e_{p})|<\frac{\delta_{1}}{|\widetilde{\vizinhancaB}|},
\end{equation}
\begin{equation}
\label{eq-ste3-uniforme-continua-d-h}
\Big| \widehat{\nabla} h(\tilde{q},e_{p})-\widehat{\nabla} h(\tilde{y},e_{p}) \Big|_{\hat{\metric}^{\tau}}<\frac{\delta_{1}}{|\widetilde{\vizinhancaB}|},
\end{equation}

\begin{equation}
\label{eq-ste3-uniforme-continua-derivada-direcional}
\Big|  \frac{d}{d t}h(\tilde{q},\efunction^{t\mathfrak{z}}(e_{p}))\big|_{t=0} 
- \frac{d}{d t} h(\tilde{y},\efunction^{t\mathfrak{z}}(e_{p}))\big|_{t=0}  \Big|<\frac{\epsilon}{2 K|\widetilde{\vizinhancaB}|},
\end{equation} 
for $t\in I$ so that $\efunction^{t\mathfrak{z}}(\vizinhancaB)$ is contained in a 
neighborhood of $\vizinhancaB$ and $K>0$ is a constant satisfying 
$| \mathcal{L}_{1}^{\ell}(b,\cdot)| <K $.

For $\tilde{q}\in \widetilde{\vizinhancaB}_{i j k}$ we can define the \emph{rest} as
\[
R_{i j k}(\tilde{q},e_p):=\int_{\widetilde{\vizinhancaB}_{i j k}} h(\cdot, e_{p})\density_{\tilde{\metric}} - h(\tilde{q}, e_{p})V_{i j k}.
\]
Equation \eqref{eq-ste3-uniforme-continua-h} implies 
\begin{equation}
\label{eq-ste3-Resto-ij}
|R_{i j k}|<\frac{\delta_{1}}{|\widetilde{\vizinhancaB}|}|\widetilde{\vizinhancaB}_{i j k}|. 
\end{equation}
We can also derive the rest $R_{ijk}$ in direction of $Y\in E_{p}$ ($|Y|_{\hat{\metric}^{\tau}}=1$) more precisely
\begin{align*}
\frac{\partial}{\partial Y} R_{i j k}(\tilde{q},e_p)& =\frac{d}{\mathrm{d}s}\big(R_{i j k}(\tilde{q}, e_{p}+s Y_p) \big)\big|_{s=0}\\
& = \int_{\widetilde{\vizinhancaB}_{i j k}} \big(\frac{\partial}{\partial Y} h(\cdot, e_{p}) - \frac{\partial}{\partial Y} h(\tilde{q}, e_{p})\big) 
\density_{\tilde{\metric}}.
\end{align*}
The above equation and Equation \eqref{eq-ste3-uniforme-continua-d-h} imply:
\begin{equation}
\label{eq-ste3-d-Resto-ij}
|\widehat{\nabla}^{E} R_{i j k}|_{\hat{\metric}^{\tau}}<\frac{\delta_{1}}{|\widetilde{\vizinhancaB}|}|\widetilde{\vizinhancaB}_{i j k}|
\end{equation}
where $\widehat{\nabla}^{E} R_{i j k}$ is the tangent part of $\widehat{\nabla} R_{i j k}$ with respect to the fibers of $E^{\delta}$.

Finally, we can derive in the direction of the vector field $\mathfrak{z}$:
\begin{align*}
\frac{\partial}{\partial \mathfrak{z}} R_{i j k}(\tilde{q},e_p)& =\frac{d}{\mathrm{d}t}\big(R_{i j k}(\tilde{q}, \efunction^{t \mathfrak{z}}(e_{p})) \big)\big|_{t=0}\\
& = \int_{\widetilde{\vizinhancaB}_{i j k}} \big(\frac{d}{d t}h(\cdot,\efunction^{t\mathfrak{z}}(e_{p}))|_{t=0}
- \frac{d}{d t}h(\tilde{q},\efunction^{t\mathfrak{z}}(e_{p}))|_{t=0} \big) \density_{\tilde{\metric}}.\\ 
\end{align*}
and hence
\begin{equation}
\label{eq-ste3-Z-Resto-ij}
| \frac{\partial}{\partial \mathfrak{z}} R_{i j k} |<\frac{\epsilon}{2 K |\widetilde{\vizinhancaB}|}|\widetilde{\vizinhancaB}_{i j k}|.
\end{equation}

We want to describe points $\tilde{q}_{i j k}\in \widetilde{\vizinhancaB}_{i j k}$
in terms of the $\mathcal{F}^{\ell}$ invariant local frame 
$\xi_{\vizinhancaB}\colon \vizinhancaB\to \mathbb{O}(E)$. 
 Let $\mathrm{Flow}^{\ell}(E^{\delta})\subset C^{\infty}(E^{\delta},E^{\delta})$ 
be the space of $\mathcal{F}$-linearized flows. 
Note that an element  $\efunction^{t X}\in  \mathrm{Flow}^{\ell}(E^{\delta})$ 
acts on $E^{\delta}$ as well as on $\mathbb{O}(E)$, because  it sends an orthonormal frame to another orthonormal frame. 
We define a map \emph{collection of rotations} 
 $\mathcal{R}\colon G_{p_0}\to \mathrm{Flow}^{\ell}(E^{\delta})$  so that  $\mathcal{R}(g)$ acts on $\pi^{-1}(\vizinhancaB)$ 
 fixing each fiber isometrically. Roughly speaking, the local frame $\xi_{\vizinhancaB}$ provides a trivialization
$\pi^{-1}(\vizinhancaB)= \vizinhancaB\times\mathbb{R}^{k}$, and under this identification we have $\mathcal{R}(g)(p,e)=(p, g e)$ where $g$ is identified with an element of
$\mathbb{O}(n)$.   The $\mathcal{F}^{\ell}$ invariant local frame $\xi_{\vizinhancaB_0}$ also allow us to construct the
map \emph{colllection of permutations} $\mathcal{P}\colon \vizinhancaB\times \vizinhancaB\to \mathrm{Flow}^{\ell}(E^{\delta})$. 
Roughly speaking, given two points $p, q$ in the $m$-rectangle $\vizinhancaB$, we consider a vector field   
   that is the constant vector field $\vec{p q}=q-p$ on $\vizinhancaB$  and is zero outside of an open set containing $\vizinhancaB$. The 
map $\mathcal{P}(p, q)$ is the parallel translation along the flow of this vector field $\vec{p q}$. 
In particular, using the local trivialization $\xi_{\vizinhancaB}$, we have that $\mathcal{P}(p, q) (p,e)=(q,e)$.
We can finally define $\mathcal{P}_{i i_{0}}=\mathcal{P}(p_{i_{0}},p_{i})$ for $p_{i_0}\in \vizinhancaB_{i_0}$ and $p_{i}\in \vizinhancaB_{i}$ and
$ \mathcal{R}_{ i j k}^{i_0}=\mathcal{R}(g_{i j k})$.
We must stress that  the maps $\mathcal{R}_{ i j k}^{i_0} $ 
and  $\mathcal{P}_{i i_{0}}$  have the  following properties:
\begin{itemize}  
\item the \emph{rotation map }  $ \mathcal{R}_{ i j k}^{i_0}\colon \pi^{-1}(\vizinhancaB_{i})\to \pi^{-1}(\vizinhancaB_{i})$ fixed each fiber; 
\item the \emph{permutation} map  $\mathcal{P}_{i i_{0}}\colon \pi^{-1}(\vizinhancaB_{i_0})\to \pi^{-1}(\vizinhancaB_{i})$ 
when restricted to $B$ permute $\vizinhancaB_{i_0}$ with $\vizinhancaB_{i}$ (in particular preserving
the volume);
\item $\mathcal{R}_{i j k}^{i_0}\circ\mathcal{P}_{i i_{0}}\circ \xi(p)\in \widetilde{\vizinhancaB}_{i j k}$ 
for each $p\in \vizinhancaB_{i_{0}}$.
\end{itemize}
Recalling the definition in Equation \eqref{eq-step4-defh}, 
it is possible to check  that 
\[
h(\mathcal{R}_{i j k}^{i_0}\circ\mathcal{P}_{i i_{0}}\circ \xi_{\vizinhancaB}(p),  e_p)
= f\big(\mathcal{R}_{i j k}^{i_0}\circ\mathcal{P}_{i i_{0}} e_{p} \big),
\]
see details in Lemma \ref{s-fiber-isometric-lifted-leaf}.
	
Finally we can write  \eqref{equation-step3-avarage-sketch} as a Riemann sum:
\begin{equation}
\label{equation-step4-1-avarage-sketch}
 \mathcal{V}\mathrm{Av}(f) (e_p) 
= \sum_{i j k} f\big(\mathcal{R}_{i j k}^{i_0}\circ\mathcal{P}_{i i_{0}} e_{p} \big) V_{i j k} + R_{i_0} (e_p)
\end{equation}
for $e_{p}\in \pi^{-1}(\vizinhancaB_{i_{\scriptscriptstyle{0}}})$ and  
$ R_{i_0} (e_p)=\sum_{i j k} R_{i j k}(  \mathcal{R}_{i j k}^{i_0}\circ\mathcal{P}_{i i_{0}} \circ\xi_{\vizinhancaB}(p), e_{p})$.

Since $\mathrm{diam}(\vizinhancaB_{i j k})<\delta_2$, Equation \eqref{eq-ste3-Resto-ij} implies that $|R_{i_0}|< \delta_1$ and 
Equation \eqref{eq-ste3-d-Resto-ij} implies   $|\widehat{\nabla}^{E} R_{i_0}|_{\hat{\metric}^{\tau}}<\delta_{1}$.   
It follows from  Equation \eqref{equation-step3-lagrangean-lower-E} and 
\eqref{equation-funcional-linear-gradiante-tangente}
 that: 
\begin{align*}
|(\hat{\lfunctional}_{1})_{\vizinhancaB_{i_0}}(R_{i_{0}})| &= \Big|\int_{\pi^{-1}(\vizinhancaB_{i_{0}})}
\widehat{\mathcal{L}}^{\ell}\big(\widehat{\nabla} R_{i_{0}},R_{i_{0}}\big)\density_{\metric}\Big|\\
&=   \Big|\int_{\pi^{-1}(\vizinhancaB_{i_{0}})}\widehat{\mathcal{L}}^{\ell}\big(\widehat{\nabla}^{E} R_{i_{0}},R_{i_{0}}\big)\density_{\metric}\Big|\\
 & <\frac{\epsilon}{2|\pi^{-1}(\vizinhancaB)|}|\pi^{-1}(\vizinhancaB_{i_{0}})|. 
\end{align*}

From Equation   \eqref{eq-ste3-Z-Resto-ij} we also infer that:
\begin{align*}
|(\hat{\lfunctional}_{2})_{\vizinhancaB_{i_0}}(R_{i_{0}})| &=  \Big| \int_{\pi^{-1}(\vizinhancaB_{i_{0}})}   
\mathcal{L}_{1}^{\ell}(b)
 \mathfrak{z}\cdot R_{i_{0}}\, \density_{\metric}\Big|\\
 & 
<\frac{\epsilon}{2|\pi^{-1}(\vizinhancaB)|}|\pi^{-1}(\vizinhancaB_{i_{0}})|. 
\end{align*}
Therefore 
\[
|(\lfunctional_{\vizinhancaB_{i_0}}^{r}(R_{i_{0}})|  <\frac{\epsilon}{|\pi^{-1}(\vizinhancaB)|}|\pi^{-1}(\vizinhancaB_{i_{0}})|.
\]

In addition letting the index $i$ vary from $1$ to $n$, we infer: 
\begin{equation}
\label{equation-step3-controle-resto}
\Big|\sum_{i=1}^{m_{1}} \lfunctional_{\vizinhancaB_i}^{r}(R_{i}) \Big|< \epsilon.
\end{equation}

\subsubsection{
\textbf{Step 5:} We apply the symmetric linear operator $l_{\vizinhancaB_{i_{\scriptscriptstyle{0}}}}$ on Equation \eqref{equation-step4-1-avarage-sketch} 
to simplify the expression, taking into consideration Equation \eqref{equation-step3-igual-volumeVij}:
}
\begin{align*}
\lfunctional_{\scriptscriptstyle \vizinhancaB_{\scriptscriptstyle i_{\scriptscriptstyle 0}}}^{r}\big( \mathcal{V}\mathrm{Av}(f) \big)
& = 
 \sum_{i j k} \lfunctional_{\scriptscriptstyle \vizinhancaB_{i_{0}}}^{r}\big(  f\circ \mathcal{R}_{i j k}^{i_0}\circ\mathcal{P}_{i i_{0}}  \big) 
V_{i j k} +  \lfunctional_{\scriptscriptstyle \vizinhancaB_{i_{\scriptscriptstyle{0}}}}^{r}\big( R_{i_0} \big)\\
 & \stackrel{\eqref{eq-mudanca-variavel-operadorlinear-restrito}}{=}  
\sum_{i j k} \lfunctional_{ \scriptscriptstyle \vizinhancaB_{i}}^{r}\big(  f  \big)V_{i j k}
+  \lfunctional_{\scriptscriptstyle \vizinhancaB_{i_{\scriptscriptstyle{0}}}}^{r}\big( R_{i_0} \big) \\
 & =  \sum_{i j } \lfunctional_{ \scriptscriptstyle \vizinhancaB_{i}}^{r}\big(  f  \big)V_{i j }
+  \lfunctional_{\scriptscriptstyle \vizinhancaB_{i_{\scriptscriptstyle{0}}}}^{r}\big( R_{i_0} \big) \\
 & \stackrel{\eqref{equation-step3-igual-volumeVij}}{=} \sum_{i j } \lfunctional_{ \scriptscriptstyle \vizinhancaB_{i}}^{r}\big(  f  \big)V
+  \lfunctional_{\scriptscriptstyle \vizinhancaB_{i_{\scriptscriptstyle{0}}}}^{r}\big( R_{i_0} \big) \\
& =  \sum_{j=1 }^{m_2} \lfunctional_{ \scriptscriptstyle \vizinhancaB}^{r}\big(  f  \big)V
+  \lfunctional_{\scriptscriptstyle \vizinhancaB_{i_{\scriptscriptstyle{0}}}}^{r}\big( R_{i_0} \big)\\
& =   \lfunctional_{ \scriptscriptstyle \vizinhancaB}^{r}\big(  f  \big) m_{2}V
+  \lfunctional_{\scriptscriptstyle \vizinhancaB_{i_{\scriptscriptstyle{0}}}}^{r}\big( R_{i_0} \big). \\  
\end{align*}

\subsubsection{
\textbf{Step 6:} In  the previous expression, we let the  index $i_0$ vary from $1$ to $m_1$:}

\begin{align*}
 \lfunctional_{\vizinhancaB}^{r}\big( \mathcal{V}\mathrm{Av}(f) \big)
& = \sum_{i=1}^{m_{1}} \lfunctional_{\vizinhancaB_{i}}^{r}\big( \mathcal{V}\mathrm{Av}(f) \big)\\
& = \lfunctional_{ \scriptscriptstyle \vizinhancaB}^{r}\big(  f  \big)m_{1} m_{2}V  
+ \sum_{i=1}^{m_{1}} \lfunctional_{\vizinhancaB_{i}}^{r}\big( R_{i} \big).\\  
\end{align*}
Therefore
\[
\lfunctional_{\vizinhancaB}^{r}\big( \mathrm{Av}(f) \big)=
\big(\frac{ m_{1}m_{2}V }{\mathcal{V}} \big) \lfunctional_{\vizinhancaB}^{r}( f )  +  \sum_{i=1}^{m_{1}} \frac{1}{\mathcal{V}}\lfunctional_{\vizinhancaB_{i}}^{r}\big( R_{i} \big).
\]

Note that Equation  \eqref{equation-step3-controle-resto} 
and the  arbitrariness in the choice of $\epsilon$  
  allow us to conclude that: 
\begin{equation}
\label{equation-step5-avarage-sketch}
\lfunctional_{\vizinhancaB}^{r}\big( \mathrm{Av}(f) \big)=\tilde{\lambda} \lfunctional_{\vizinhancaB}^{r}( f )=\tilde{\lambda} \lfunctional^{r}( f ),
\end{equation}
where 
\[
\tilde{\lambda}=\frac{|\widetilde{\vizinhancaB}|}{\mathcal{V}}=\frac{|G_{x}||\vizinhancaB|}{|G_x||B|}=\frac{|\vizinhancaB|}{|B|}.
\]

\subsubsection{\textbf{Last step:} }  Equation \eqref{equation-1-new-MALEX-lemma-Av-l-orbitlike} now follows from
Equation \eqref{equation-step5-avarage-sketch}, Lemma \ref{lemma-relation-integral-linear-functional-subsets-basic}, and 
the arbitrariness in the choice of $r$.
\end{proof}

\begin{remark}
Consider the vector field $\mathfrak{z}=\widehat{\nabla}b - \nabla b \in\mathfrak{X}(\mathcal{F}_{U^{0}})$
defined in \eqref{equation-diference-between-gradients}.
Note that  \emph{$\mathfrak{z}$ commutes with the AVP system $\{ \vec{X}_{\alpha}\}$.} In fact
since $(\efunction^{-t X_{\alpha}})_{*}\widehat{\nabla b}=\widehat{\nabla b}$ and  $(\efunction^{-t X_{\alpha}})_{*} \nabla b=\nabla b$  we have 
$(\efunction^{-t X_{\alpha}})_{*}\mathfrak{z}=\mathfrak{z}$ and hence, deriving with respect to $t=0$, $[\vec{X}_{\alpha},\mathfrak{z}]=0$. 
\end{remark}


\section{Examples of AVP foliations}
\label{section-exa-AVP}

 In this section, we will provide examples of AVP foliations (recall Definition \ref{definition-a-v-p}).
   In particular, we stress that  \emph{SRF on a fiber bundle with
Sasaki metric are AVP} (see Proposition \ref{MALEX-flows-preserve-volume}) and that
  \emph{isoparametric foliations are AVP} (see Theorem \ref{theorem-isoparemetric-is-AVP}).

\subsection{SRF on a fiber bundle  with Sasaki metric are AVP}
\label{subsection-Ex-SRF-Sasaki-AVP}

\begin{proposition}[Linearized vector fields and volume] \label{MALEX-flows-preserve-volume}
Consider an Euclidean vector bundle $\mathbb{R}^{k} \to E \to B$  with a Sasaki metric $\metric^{\tau}$ induced by an Ehresmann
connection $\mathcal{T}$ tangent to the leaves of a singular foliation $\mathcal{F}$ on $E$. 
Assume that $\mathcal{F}$ is SRF with respect to $\metric^{\tau}$ . 
Then 
there exists a geometric control systems  of linearized vector fields $\mathcal{C}=\{\vec{X}_{\alpha}\}$ generating $\mathcal{F}^{\ell}$, 
whose flows $\efunction^{t X_{\alpha}}$
preserve the Riemannian density $\density_{\metric^{\tau}}$. 
Therefore, the SRF $\mathcal{F}$  on $E$ is an AVP foliation with respect to the Sasaki metric $\metric^{\tau}$.
\end{proposition}

\begin{proof}
For the sake of simplification, we assume that $E$ is orientable. 
It suffices to construct two types of linearized vector fields that preserve the volume $\vol:=\vol_{\metric^{\tau}}$:

\vspace{0.25\baselineskip}

\begin{enumerate}
\item[Type 1:] $\vec{\mathfrak{T}}\in\mathfrak{X}(E)$ whose flows restricted to $B$ 
are non trivial. 
\item[Type 2:] $\vec{\Theta}\in\mathfrak{X}(E)$ whose flows fix the fibers of $E$.  
\end{enumerate}

\vspace{0.25\baselineskip}

\emph{Constructing type 1 vector fields:}

\vspace{0.25\baselineskip}

First, we consider a vector field $\vec{\mathfrak{T}}\in \mathfrak{X}(B)$ that preserves the volume $\vol_{B}$ of $B$. 
If $\dim B=0$, there is nothing to do. If $\dim B=1$, i.e., if $B$ is the circle $S^{1}$, 
it is easy to see the existence of a global vector field that preserves the volume.
 So let us assume that $\dim B\geq 2$ and let us review the construction that given a point $p_{0}\in B$ and 
small neighborhood $W\subset B$ of $p_0$, there exists a vector field $\vec{\mathfrak{T}}\in \mathfrak{X}(B)$ with support on neighborhood $W$ that preserves the volume. Consider a coordinate system $\{p_{i}(\cdot)\}_{i=1}^{m} $so that $p_0$ is identified with $0$ and $\vol_{B}$ 
with $d p_{1}\wedge \cdots \wedge d p_{m}$,  recall \cite[p.6]{Kob95}.
Let $\widetilde{\mathfrak{T}}$ be \textcolor{blue}{a} (Euclidean)  Killing vector field that fixes $0$ and a smooth  non negative function $\rho\colon \mathbb{R}\to\mathbb{R}$, with small  compact support in $(-\epsilon,\epsilon)$, and $\rho(0)=1$. Since the flow of $\widetilde{\mathfrak{T}}$ preserves $\vol_{B}$, 
then the flow of  $\vec{\mathfrak{T}}(p)=\rho(\|p\|) \widetilde{\mathfrak{T}}(p)$ also preserves $\vol_{B}$ and has compact support. By pulling back 
$\vec{\mathfrak{T}}$ via the coordinate system, we identify the vector field $\vec{\mathfrak{T}}$ with a vector field 
on $B$, which we are also denoting as $\vec{\mathfrak{T}}$. 
We observe that even thought the vector field $\vec{\mathfrak{T}}$ preserves the volume $\vol_{B}$, it is not a Killing field over $B$.

Now we can extend the vector field $\vec{\mathfrak{T}}$ on $B$ to a $\pi$-basic vector field on $E$ so that $\vec{\mathfrak{T}}\in\mathfrak{X}(\mathcal{T})$. 
Since $\mathcal{T}\subset T\mathcal{F}^{\ell}$ is homothety-invariant, the vector field $\vec{\mathfrak{T}}$ 
is a homothety invariant $\mathcal{F}^{\ell}$-vector field. Therefore, 
$\vec{\mathfrak{T}}$ is a linearized vector field, and its flows are isometries between the fibers (recall properties of the Sasaki metric $\metric^{\tau}$). 

Let $\vol_{E}$ be the volume form on the fibers $E$, note that $\pi^{*}\vol_{B}\wedge \vol_{E} $ coincides with the volume $\vol_{\metric^{\tau}}$ of the Sasaki metric, because $\pi\colon E \to B$ is a Riemannian submersion.

Since $\pi\circ   \efunction^{t \mathfrak{T}} =  \efunction^{t \mathfrak{T}}\circ\pi$ and $\efunction^{t\mathfrak{T}}$ preserves $\vol_{B}$ we infer that
\[
(\efunction^{t \mathfrak{T}} )^{*}\Big(\pi^{*}\vol_{B}\wedge \vol_{E}\Big)=\pi^{*}\vol_{B}\wedge \vol_{E}.
\]

Based on the construction we recalled above, we can infer that the set of vector fields that preserves the volume $\vol_{B}$ of $B$
 can be chosen as transitive in $B$. Thus, this property remains valid for type 1 vector fields on $\mathfrak{X}(E)$.

\vspace{0.25\baselineskip}

\emph{Constructing type 2 vector fields:}

\vspace{0.25\baselineskip}

Consider the connected isotropy group $G_{p_0}^{0}$ (recall definition in Remark \ref{remark-Connected-component-isotropy}). Using the parallel transport with respect to $\nabla^{\tau}$, 
we can induce action of $\mu\colon  G_{p_0}^{0}\times\pi^{-1}(W)\to \pi^{-1}(W)$ on a neighborhood  $W\subset B$ of $p_0$ that fixes the fibers, and act on each fiber isometrically (with respect to the fiberwise Riemannian metric on the fibers of $\pi$). 
For $\vartheta$ in the Lie algebra of $G_{p_0}^{0}$, set $\vec{\Theta}(x)=d\mu_{x}\big(f(\pi(x)) \vartheta\big)$ for some smooth non negative function $f$ with compact support on $W$ and such that $f(p_0)=1$. Since its flow acts  isometrically on each fiber and its projection on $B$ is the identity, we have that:
\[
( \efunction^{t \Theta} )^{*}\Big(\pi^{*}\vol_{B}\wedge \vol_{E}\Big)=\pi^{*}\vol_{B}\wedge \vol_{E}.
\]
By construction, the restriction of the set type 2 vector fields  is transitive on  the infinitesimal foliations of $\mathcal{F}^{\ell}$. 
\end{proof}

\begin{remark}
\label{remark-control-preserve-volume-adapted-sasaki-metric}
Let $D_{\delta}(0)\to E^{\delta}\to B$ be a disc bundle with a  Riemannian metric $\metric$ and $\mathcal{F}$ be a SRF on 
$(E^{\delta},\metric)$. Consider an adapted Sasaki metric  $\hat{\metric}^{\tau}$ (recall Definition \ref{definition-adapted-sasaki-metric}).
By an argument similar to the proof of Proposition \ref{MALEX-flows-preserve-volume}, we can prove that there is a family of linearized vector fields 
$\{\vec{X}_{\alpha}\}$ whose orbits generate $\mathcal{F}^\ell$ and that preserves the Riemannian density
$\density_{\hat{\metric}^{\tau}}=|\vol_{\hat{\metric}^{\tau}}|$.
\end{remark}

\subsection{ Isoparametric foliations are AVP }
\label{Section-Isoparametric-are-AVP}

In this subsection, we will prove Theorem \ref{theorem-isoparemetric-is-AVP}
i.e., that each isoparametric foliation $\mathcal{F}$ (recall Definition \ref{definition-polar-isoparametric}) on 
a compact manifold $(M,\metric)$ is AVP.

\begin{definition}
\label{definition-polar-isoparametric}
  A SRF $\mathcal{F}=\{ L \}$ on a (complete) Riemannian manifold $(M,\metric)$ is called a \emph{polar foliation}, if for each
	regular point $p\in M$, the set $\polarsection_{p}=\exp_{p}^{\nu}(\nu_{p}(L))$, the so-called \emph{section},  
	is a complete totally geodesic immersed submanifold that intersects all the leaves orthogonally. 
	This is equivalent to saying that the normal bundle of $\mathcal{F}$ restricted to the  regular stratum is integrable,
	see  \cite[Theorem 5.24]{alexandrino-Bettiol-2015}. 
	 A polar foliation with  the property that the restriction of the mean curvature vector field of each leaf to the regular stratum is basic,
	is called an \emph{isoparametric foliation}.
	\end{definition}

Some  references for the theory of isoparametric submanifolds are 
\cite{PalaisTerng88,Berndt-Console-Olmos-2016,Terng-Thorbergsson-95,Thorbergsson-survey-10}. We also suggest
the   survey  \cite{Thor22} for recent contributions in this theory.

On the one hand, this subsection is not a prerequisite for the next  sections, 
and, a reader more interested in analytical discussions can safely skip it. 
On the other hand, this new property of this well-studied geometric object may interest geometers.

\subsubsection{A brief review on polar foliations}

Let us briefly review a few properties of polar foliations (recall Definition \ref{definition-polar-isoparametric}),
which will help us understand the proof of Theorem \ref{theorem-isoparemetric-is-AVP}.
These results are extracted from \cite[Chapter 5]{alexandrino-Bettiol-2015}.

\begin{theorem}[Slice theorem]
\label{malex-slice-theorem}
Let $\mathcal{F}$ be a polar foliation and $S_x$ a slice at $x\in M$. Then 
\begin{enumerate}
\item[(a)] $S_{x}=\cup_{\sigma_x\in\Lambda(x)}$, where $\Lambda(x)$ is the set of \emph{local sections} $\sigma_{x}$, i.e. 
open convex neighborhoods of $\polarsection$ centered at $x$;
\item[(b)] $S_{y}\subset S_{x}$ for each $y\in S_{x};$
\item[(c)] $\mathcal{F}|_{S_x}$ is a polar foliation on $S_x$ with the induced metric;
\item[(d)] the infinitesimal foliation $\mathcal{F}_{x}:= (\exp_{x}^{\nu})^{-1}(\mathcal{F}\cap S_{x})$ is a polar foliation on the vector space, 
$\nu_{x}L_{x}=T_{x}S_{x}$. In addition, the
sections of $\mathcal{F}_{x}$ are sent to the sections of $\mathcal{F}$ via the normal exponential map $\exp_{x}^{\nu}$. 
\end{enumerate}
\end{theorem}
\begin{remark}
It is easy to check that  each polar foliation on Euclidean spaces and spheres is an isoparametric foliation, see e.g., 
\cite[Proposition 5.5]{alexandrino-Bettiol-2015}. 
'More recently, it was proved in \cite[Theorem 1.1]{Liu-Radeschi}
 that every polar foliation  on a simply connected symmetric space with non-negative curvature is isoparametric.
 However, already for hyperbolic spaces, one can easily construct counterexamples of this bijection between polar foliations and isoparametric foliations. 
It is also worth noting here that there are no polar foliations on compact manifolds with negative curvature, see \cite{Toe07}.
\end{remark}

Let us now recall the concept of singular (transverse) holonomy, which extends the classical concept of holonomy of (regular) foliations. 
Let $x_0$ be a regular point of $\mathcal{F}$ and $\beta\colon [0,1]\to L_{x_0}$ be a piecewise smooth curve. We define \emph{a singular holonomy along $\beta$}
as the diffeomorphism $\varphi_{[\beta]}\colon \sigma_{\beta(0)}\subset \polarsection_{\beta(0)}\to \sigma_{\beta(1)}\subset\polarsection_{\beta(1)}$ defined as:
$\varphi_{[\beta]}(x)=\exp_{\beta(1)}^{\nu}(\mathcal{P}_{\beta} Y)$, where 
$Y=\exp_{\beta(0)}^{-1}(x)\in\nu_{\beta(0)}L_{\beta(0)}$ and $\mathcal{P}_{\beta}$ is the parallel transport along $\beta$
 with respect to  the normal connection $\nabla^{\nu}$. 
\begin{proposition} The singular holonomy $\varphi_{[\beta]}\colon \sigma_{\beta(0)}\subset \polarsection_{\beta(0)}\to \sigma_{\beta(1)}\subset\polarsection_{\beta(1)}$
fulfills the following properties:
\begin{enumerate}
\item[(a)]$\varphi_{[\beta]}$ depends only on the homotopy class with fixed endpoints of $\beta;$
\item[(b)] $\varphi_{[\beta]}(x)=L_{x}$, for all $x\in \polarsection_{\beta(0)};$
\item[(c)]$\varphi_{[\beta]}$ is an isometry where $\sigma_{\beta(i)}$ (for $i=0,1$) 
can be chosen to be a convex ball $B_{\delta}(\beta(i))$
including singular points of $\mathcal{F}$, where $\delta$ depends only on the geometry of $\polarsection$ (and not of $\mathcal{F}$).
\end{enumerate}
\end{proposition}
The pseudogroup $\weylgroup(\polarsection)$ 
generated by the singular holonomies maps $\varphi_{[\beta]}$ with $\beta(0)$ and $\beta(1)$ in the same section $\polarsection$
is called \emph{Weyl pseudogroup of $\polarsection$}. It describes how the leaves of $\mathcal{F}$ intersect $\polarsection$. 

\begin{theorem}[Reflection pseudogroup]
The set of singular points of $\mathcal{F}$ on a section $\polarsection$ 
is a locally finite union of totally geodesic hypersurfaces $\{ \mathcal{H}_{i} \}\subset \polarsection$
(the so-called walls of Weyl chambers). The reflections into $\mathcal{H}_{i}$ are also elements of $\weylgroup(\polarsection)$. 
If $\mathcal{F}$ is a polar foliation
on a simply connected space (in particular in Euclidean spaces and round spheres), then 
these reflections generate the Weyl group. 
\end{theorem} 
\begin{remark}
The Weyl pseudogroup  $\weylgroup(\polarsection)$ (and the reflection pseudogroup)
 can be lifted to a group of isometries $\widetilde{\weylgroup}$ of the Riemannian univerval cover $\widetilde{\polarsection}$ of the section $\polarsection$.
All  sections have the same universal cover and the same lifted Weyl group $\widetilde{\weylgroup}$. In particular, given two sections
$\polarsection_0$ and $\polarsection_1$ the pseudogroups 
$(\weylgroup(\polarsection_{0}),\polarsection_0)$ and $(\weylgroup(\polarsection_{1}),\polarsection_1)$ are equivalent (i.e., Haefliger equivalence).  
\end{remark}

\subsubsection{Proof of Theorem \ref{theorem-isoparemetric-is-AVP} } 
For the sake of simplification, we assume that $M$ is orientable. 
Following the notation of this section  $B=L_{p_0}$ is going to be a  fixed  leaf of $\mathcal{F}$,    
$U$ a tubular neighborhood   $\mathrm{Tub}_{\delta}(B)$, and  
$\mathcal{C}(U)=\{\vec{X}_{\alpha}\}\subset\mathfrak{X}(U)$  
is going to be a geometric control system of complete linearized vector fields whose orbits are leaves of the linearized foliation 
$\mathcal{F}^{\ell}\subset\mathcal{F}_{U}$.  We also assume that the flows $\efunction^{t X_{\alpha}}$ preserves a  volume form $\vol^{\tau}$  of a 
Sasaki metric $\metric^{\tau}$ compatible with
$\mathcal{F}_{U}$, that means the associated  Ehresmann connection $\mathcal{T}$ 
is tangent to the leaves of $\mathcal{F}$, i.e.   $\mathcal{T}\subset T\mathcal{F}$,  recall Proposition  \ref{MALEX-flows-preserve-volume}.

The proof of the theorem is divided into two lemmas.

\begin{lemma}
\label{lemma1-isoparemetric-is-AVP}
The flows $\efunction^{t X_{\alpha}}$ send normal spaces to normal spaces of principal leaves on $U$.
\end{lemma}

\begin{proof}
Set  $\varphi$ the restriction of the flow $\efunction^{tX_{\alpha}}$ to the fibers.  
The map $\varphi\colon E_{p}\to E_{\varphi(p)}$  sends
the infinitesimal foliation $\mathcal{F}_{p}$ to the infinitesimal foliation $\mathcal{F}_{\varphi(p)}$
isometrically (with respect to the standard  Sasaki metric $\metric^{\tau}$).
From item (d) Theorem \ref{malex-slice-theorem}, 
we know that infinitesimal foliations of polar foliations are polar foliations on vector spaces. 
Therefore $\varphi$ sends the sections of the polar  foliation $(\mathcal{F}_{p},E_{p},\metric_{p})$ to the sections
of the polar foliation $(\mathcal{F}_{\varphi(p)},E_{\varphi(p)},\metric_{\varphi(p)})$. 
Since the sections of the infinitesimal polar foliations coincide (via identification of normal exponential map)
with sections of the polar foliation $\mathcal{F}$ (see item (d) Theorem\ref{malex-slice-theorem}), 
we conclude that $\varphi$ sends sections of $\mathcal{F}$ to sections
of $\mathcal{F}$.  The result now follows since the normal spaces of regular leaves of $\mathcal{F}$ coincide with the tangent spaces
of sections of $\mathcal{F}$.
\end{proof}
Remark \ref{remark-control-preserve-volume-adapted-sasaki-metric} and the next lemma imply that 
the flows preserve $\vol_{\metric}$. 
\begin{lemma}
\label{lemma2-isoparemetric-is-AVP} 
Let $\hat{\metric}^{\tau}$ be the adapted Sasaki metric (recall Subsection \ref{subsection-Remark-definition-AVP}). Then
\[
\vol_{\metric}|_{U^{0}}=v\vol_{\hat{\metric}^{\tau}},
\]
where  $v\colon U^{0}\to \mathbb{R}$ is the $\mathcal{F}$-basic function on principal stratum $U^{0}$ 
 defined as $v(x)=\frac{\mathcal{V}(x)}{\mathcal{V}^{\tau}(x)}$,
 for $\mathcal{V}(x):=\int_{L_{x}}\density_{\metric}$ and $\mathcal{V}^{\tau}(x):=\int_{L_{x}}\density_{\metric^{\tau}}$, i.e., for 
the volume functions of the  leaves of $\mathcal{F}_{U^{0}}$ associated to the metric $\metric$ and  standard Sasaki metric $\metric^{\tau}$. 
\end{lemma}
\begin{proof}
Let $v\colon U^{0}\to\mathbb{R}$ be the positive function on the principal stratum defined as
\begin{equation}
\density_{\metric}|_{T\mathcal{F}}=v \density_{\metric^{\tau}}|_{T\mathcal{F}}, 
\end{equation} 
where 
$\density_{\metric}|_{T\mathcal{F}}$ and $\density_{\metric^{\tau}}|_{T\mathcal{F}}$ are the Riemannian densities (restricted to the 
distribution of the tangent spaces $T\mathcal{F}$ of regular foliation $\mathcal{F}|_{U^{0}}$)  
associated to the metrics $\metric$ and the
(standard) Sasaki metric $\metric^{\tau}$, respectively.

In order  to prove the lemma, we will need to make local calculations 
(such as in Equation \eqref{Eq1-2-lemma2-isoparemetric-is-AVP}, where we derive in the direction of a normal vector field $\vec{Y}$). 
For this reason, throughout the proof of this lemma, we will be able to replace the concept of density with the concept of volume form.

We start by claiming  that \emph{for  a projectable vector $\vec{Y}$ on  the principal stratum $U^{0}$ of $\mathcal{F}$,  $\vec{Y}$
is orthogonal to the principal leaves (concerning the original metric $\metric$ on $U$) if, and only if,
it is orthogonal to the principal leaves concerning the adapted Sasaki metric $\hat{\metric}^{\tau}$.}
In fact to say that  $\vec{Y}$ is orthogonal with respect to $\metric$ means
that $\vec{Y}(x)$ is tangent to a local section $\sigma_x$, 
and hence (by slice theorem) it is equivalent to saying that it is   
tangent to a section of the infinitesimal foliations. 
By the construction of the adapted Sasaki metric, this section is  contained in the fiber $E_{\pi(p)}$ and that is orthogonal
to the foliation.  \emph{The same holds for  the (standard) Sasaki metric $\metric^{\tau}$. }

From now on, $\vec{Y}$ denotes an
orthogonal vector field to the principal leaves. Also recall that $\mathcal{F}_{U}$  is 
isoparametric with respect to (standard) Sasaki metric $\metric^{\tau}$, see \cite[Lemma 2.1]{ACG20}.

Let us summarize some facts about the mean curvature vector field, which we will use next;
 see \cite[Chapter 4]{Gromoll-Walschap09} and \cite[Lemma 5.2]{Alex-Radeschi-MCF}.
 Let $\formacurvaturamedia(\cdot)=\metric(\vec{H},\cdot)$ and  $\formacurvaturamedia^{\tau}(\cdot)=\metric^{\tau}(\vec{H}^{\tau},\cdot)$
be   mean curvature forms with respect to $\metric$ and $\metric^{\tau}$ 
and $\omega:=\vol_{T\mathcal{F}}$ and $\omega^{\tau}:=\vol_{T\mathcal{F}}^{\tau}$  be  local volume forms 
of the  leaves of $\mathcal{F}_{U^{0}}$   (with respectively with the metrics $\metric$ and $\metric^\tau$). Then we have:
\begin{enumerate}
\item[(a)]   $\omega|_{T\mathcal{F}}=v\omega^{\tau}|_{T\mathcal{F}};$
\item[(b)]  $\mathcal{L}_{\vec{Y}}\omega|_{T\mathcal{F}}=-\formacurvaturamedia(\vec{Y})\omega|_{T \mathcal{F}};$
\item[(c)]   $\mathcal{L}_{\vec{Y}}\omega^{\tau}|_{T\mathcal{F}}=-\formacurvaturamedia^{\tau}(\vec{Y})\omega^{\tau}|_{T \mathcal{F}};$
\item[(d)]  $\vec{H}=-\nabla(\ln(\mathcal{V}));$
\item[(e)]$\vec{H}^{\tau}=-\nabla^{\tau}(\ln(\mathcal{V}^{\tau}))$. 
\end{enumerate}
Items (a), (b), and (c) imply 
$\vec{Y}\cdot v\, \omega^{\tau}|_{T\mathcal{F}}- v\, \formacurvaturamedia^{\tau}(\vec{Y}) \omega^{\tau}|_{T\mathcal{F}}
=-\formacurvaturamedia(\vec{Y}) v \, \omega^{\tau}|_{T\mathcal{F}}$. 
Since $\omega^{\tau}|_{T\mathcal{F}}$ is a volume form, we have 
$\vec{Y}\cdot v=v \big(\formacurvaturamedia^{\tau}(\vec{Y})-\formacurvaturamedia(\vec{Y}) \big)$.  This equation, together with items (d) and (e), implies
\begin{align*}
\vec{Y}\cdot v & = v \Big(\metric^{\tau}(\vec{H}^{\tau},\vec{Y})-\metric(\vec{H},\vec{Y}) \Big)\\
               & = v\Big(  \metric^{\tau}(-\nabla^{\tau}(\ln(\mathcal{V}^{\tau})),\vec{Y})-\metric(-\nabla(\ln(\mathcal{V})),\vec{Y}) \Big)\\
							& = v\Big( \vec{Y}\cdot \ln(\mathcal{V})- \vec{Y}\cdot \ln(\mathcal{V}^{\tau}) \Big)\\
							& = v \Big(\vec{Y}\cdot \ln\big( \frac{\mathcal{V}}{\mathcal{V}^{\tau}}\big) \Big).
\end{align*}

The above equality implies 
$
\vec{Y}\cdot\ln(v)=\frac{\vec{Y}\cdot v}{v}= \vec{Y}\cdot \ln  \big(\frac{\mathcal{V}}{\mathcal{V}^{\tau}}\big) 
$
and hence:
\[
\vec{Y}\cdot\Big(\ln(v)-  \ln\big( \frac{\mathcal{V}}{\mathcal{V}^{\tau}}\big)\Big)=
\vec{Y}\cdot\ln(v)- \vec{Y}\cdot \ln\big( \frac{\mathcal{V}}{\mathcal{V}^{\tau}}\big)  =0.
\]
This implies that the next equation holds on the principal stratum $U^{0}$:
\begin{equation}
\label{Eq1-lemma2-isoparemetric-is-AVP}
\vec{Y}\cdot \ln\Big(\frac{v}{\big( \frac{\mathcal{V}}{\mathcal{V}^{\tau}} \big)} \Big)=0. 
\end{equation}
Now define the basic function on the principal stratum $v_{2}=\frac{\mathcal{V}}{\mathcal{V}^{\tau}}$, 
and the function $v_1=\frac{v}{v_2}$. From \eqref{Eq1-lemma2-isoparemetric-is-AVP} we have 
$0=\vec{Y}\cdot \ln(v_1)=\frac{\vec{Y}\cdot v_1}{v_1}$ and hence:
\begin{equation}
\label{Eq1-2-lemma2-isoparemetric-is-AVP}
\vec{Y}\cdot v_{1}=0,
\end{equation}
which implies that $v_1$ restrict to a section $\polarsection$ is locally constant. 

We claim that: \emph{given a point $p\in B$ and a fixed section $\polarsection$ containing $p$, 
for each  principal point $x_0\in \polarsection$ 
there exists a  minimal segment of unit geodesic 
$\gamma\colon [0,\delta]\to \polarsection$ so that
$\gamma(0)=p\in B$, $\gamma(\delta)=x_0$
and $\gamma|_{(0,\delta]}$ contains only
principal points.} In fact, 
let  $\weylgroup$ be the Weyl group of the infinitesimal foliation $\mathcal{F}_{p}$. We  can find a minimal geodesic $\tilde{\gamma}$ with $\tilde{\gamma}(0)\in \weylgroup(p)=\{p\}$  and  $\tilde{\gamma}(\delta)\in \weylgroup(x_0)$.  
By Klein's argument there are only principal points on $\tilde{\gamma}|_{(0,\delta)}$,
see \cite[Lemma 3.70]{alexandrino-Bettiol-2015}. 
Since the Weyl group $\weylgroup$  is a group of isometries and $p$ is fixed point of the $\weylgroup$ action, 
we can choose   an $\mathrm{w}\in \weylgroup$ (e.g, by composition of reflections) so that $\gamma= \mathrm{w}(\tilde{\gamma})$
 is a minimal unit geodesic joining $x_0$ with $p$, and this concludes the proof of the claim.   

From  Lemma \ref{lemma-volume-sobre-raio} we know that:

\begin{equation}
\label{Eq2-lemma2-isoparemetric-is-AVP}
\lim_{r\to 0} v_{2}(\gamma(r)) = \lim_{r\to 0} \frac{\frac{\mathcal{V}(\gamma(r))}{r^{k}}}{\frac{\mathcal{V}^{\tau}(\gamma(r))}{r^{k}}}= 1.
\end{equation}

 From Equation \eqref{Eq1-2-lemma2-isoparemetric-is-AVP},  we have that $v_1$ is locally constant on the principal part of $\polarsection$, i.e. $v_1=c$.
Hence, for a sequence $t_n\to 0$ we have that $v(\gamma(t_{n}))= c v_{2}(\gamma(t_{n}))$.  
We claim that
\begin{equation}
\label{Eq2-5-lemma2-isoparemetric-is-AVP}
\lim_{r\to 0} v(\gamma(r)) =  1.
\end{equation}
In order to check the claim, set $\vol_{\polarsection}$ and $\vol_{\polarsection}^{0}$ to be the volume forms in a neighborhood of $p=\gamma(0)$ 
of $\polarsection$ with respect to the original metric $\metric$ and the Euclidean metric $\metric^{0}:=\metric_{p}$ respectively. 
Define the function $f\colon \polarsection\to \mathbb{R}$ as $\vol_{\polarsection}=f\vol_{\polarsection}^{0}$.  
Multiplying this expression  on both sides  with $\omega $ and using the fact that  $\omega=v\omega^{\tau}$ we have, on  the principal stratum $U^{0}$ that:
\begin{align*}
\vol_{g}&= \omega\wedge \vol_{\polarsection}\\
&= f \omega\wedge \vol_{\polarsection}^{0}\\\
&= f v \omega^{\tau}\wedge \vol_{\polarsection}^{0}\\
&= f v \vol_{\metric^{\tau}}.
\end{align*}
Since $f(p)=1$ and $(\vol_{g})_{p}= (\vol_{\metric^{\tau}})_{p}$ we infer \eqref{Eq2-5-lemma2-isoparemetric-is-AVP}.

From Equations \eqref{Eq2-5-lemma2-isoparemetric-is-AVP} and  \eqref{Eq2-lemma2-isoparemetric-is-AVP} we conclude that $v_{1}=c=1$.
This argument can be applied for any other principal point $x_0$, section $\polarsection$ and point $p\in B$ and hence we have 
 proved that $v_{1}|_{U^{0}}=1$. This implies that $v=v_2$ and hence a basic function.

Finally recall that    there exists a  metric $\metric_{\mathcal{F}}$
on the submanifold $U^{0}/\mathcal{F}$ 
so that $\pi_{\mathcal{F}}\colon (U_{0},\metric)\to (U_{0}/\mathcal{F},\metric_{\mathcal{F}})$  
is a  Riemannian submersion, where $\pi_{\mathcal{F}}\colon U_{0}\to U_{0}/\mathcal{F}$ is the canonical projection. Therefore  
\begin{align*}
\vol_{\metric}& =\omega\wedge \pi_{\mathcal{F}}^{*}\vol_{g_{\mathcal{F}}}\\
&= v \omega^{\tau}\wedge \pi_{\mathcal{F}}^{*}\vol_{g_{\mathcal{F}}}\\
&= v\vol_{\hat{\metric}^{\tau}},
\end{align*}
which finishes the proof. 
\end{proof}

\begin{lemma}
\label{lemma-volume-sobre-raio}
Let $\gamma\colon [0,\delta]\to \polarsection$ be a minimal horizontal segment of unit geodesic  so that
$\gamma(0)=p_{0}\in B$ and $\gamma|_{(0,\delta]}$ contains only principal points. Then:
\[
 \lim_{r\to 0}\frac{\mathcal{V}(\gamma(r))}{r^{l}}
=\lim_{r\to 0} \frac{\mathcal{V}^{\tau}(\gamma(r))}{r^{l}}=N \big|L_{e_{p_0}}^{0}\big| |B|
\]
where $\gamma'(0)=e_{p_0}$, $\big|L_{e_{p_0}}^{0}\big| $ is the volume (with respect to the flat metric $\metric_{p_{0}}$)
of the (Euclidean) isoparametric leaf $L_{e_{p_0}}^{0}$ of $(\mathcal{F}_{p_{0}}, E_{p_{0}})$,  
$l=\dim(L_{e_{p_0}}^{0})$ and $N$ is the number of intersections of $L_{\gamma(\delta)}$ with $E_{p_{0}}$.
\end{lemma}
\begin{proof}
 The proof follows from straightforward calculations (via coordinate) and Fubini's theorem.
For the sake of completeness, we provide a few additional details here.
  
We are going to check that
$ \lim_{r\to 0}\frac{\mathcal{V}(\gamma(r))}{r^{l}}=N \big|L_{e_{p_0}}^{0}\big| |B|$.  
The proof of the other equality can be demonstrated in the same way.

We start by defining a few objects.
Let $\pi\colon E\to B$ be the metric foot  point projection, $\vol_{\mathcal{T}}$ the volume form (with respect to the
original metric $\metric$)  
on the distribution $\mathcal{T}$  and $\vol_{\pi^{-1}}$ the volume form (with respect  to the original metric) of the fibers $\pi^{-1}$. 
Consider 
$h\colon E\to\mathbb{R}$ be the function defined by  
\begin{equation}
\label{equation-definition-h-lemma-volume-sobre-raio}
\vol_{\mathcal{T}}=h\, \pi^{*}\vol_{B}|_{\mathcal{T}}.
\end{equation}
Note that in particular 
\[
\vol_{\metric}=h \pi^{*}\vol_{B}\wedge \vol_{\pi^{-1}}.
\]
Let $\{\vizinhancaB_{\alpha}\}$ be an (finite) open cover of $B$ and $\{\rho_{\alpha}\}$ a partition of unity  subordinate to $\{ \vizinhancaB_{\alpha}\}$.  
Let $\pi_{r}=\pi|_{L_{\gamma(r)}}\to B$  be restriction of $\pi$ to the leaf $L_{\gamma(r)}$  
and set $\rho_{\alpha}^{r}=\rho_{\alpha}\circ\pi_{r}$ and $\vizinhancaB_{\alpha}^{r}=\pi_{r}^{-1}(\vizinhancaB_{\alpha})$.  
By Fubini's theorem, we have:
 \begin{align*}
\frac{\mathcal{V}(\gamma(r))}{r^{l}} & = 
\sum_{\alpha} \int_{\vizinhancaB_{\alpha}^{r}}\rho_{\alpha}^{r} \, \frac{h}{r^{l}}\,| \pi^{*}\vol_{B}\wedge \vol_{\pi^{-1}_r}|\\
& = \sum_{\alpha} \int_{\vizinhancaB_{\alpha}}\rho_{\alpha}(p) \Big( \int_{\pi^{-1}_{r}(p)} \, h\, \frac{\vol_{\pi^{-1}_r}}{r^{l}}\Big) |\vol_{B}|\\
& = \sum_{\alpha} \int_{\vizinhancaB_{\alpha}}\rho_{\alpha}(p) I(p,r) |\vol_{B}|,
\end{align*}
where 
$ I(p,r)=\Big( \int_{\pi^{-1}_{r}(p)} \, h\, \frac{\vol_{\pi^{-1}_r}}{r^{l}}\Big). $
Our goal is then to check that
\begin{equation}
\label{eq-1-lemma-volume-sobre-raio}
\lim_{r\to 0} I(p_{0},r)=    N \big|L_{e_{p_0}}^{0}\big|.
\end{equation}
In order to integrate on $\pi^{-1}_{r}(p)$ (that is a possible non connected fiber
with $N$ components diffeomorphic to  $L_{e_{p_0}}^{0}$) we need to consider parametrizations. 
Consider a local $\mathcal{F}$-invariant orthonormal  frame $\{\xi_{k}\}$ on $\vizinhancaB_\alpha$ so that
$L_{\xi_{k}(p)}=L_{\xi_{k}(p_0)}$. 
Let $\psi_{\beta}\colon V_{\beta}\to \pi^{-1}_{1}(p_0)\subset E_{p_0}^{1}$ be  parametrizations 
$\psi_{\beta}(\theta)=\sum _{k} \psi_{k,\beta}(\theta)\xi_{k}(p_0)$ 
of the connected components of  $\pi_{1}^{-1}(p_0)$, i.e.,  the isoparametric submanifolds
to $L_{e_{p_0}}^{0}$. 
Define the parametrization 
$F^{r}_{\beta}\colon V_{\beta}\times \vizinhancaB_{\alpha}\to L_{\gamma(r)}$ as
$F^{r}_{\beta}(\theta,p)=  \sum _{k} r\psi_{k,\beta}(\theta)\xi_{k}(p) $
and $F^{r,p}_{\beta}(\theta)=F^{r}_{\beta}(\theta,p)$. 
Note that 
\begin{equation}
\label{eq-2-lemma-volume-sobre-raio}
\frac{\partial}{\partial \theta_i} F_{\beta}^{r,p}(\theta)=
\Big( \sum _{k} r\frac{\partial \psi_{k,\beta}}{\partial \theta_{i}}(\theta)\xi_{k}(p) \Big)_{\sum _{k} r\psi_{k,\beta} \xi_{k}},
\end{equation}
\begin{equation}
\label{eq-3-lemma-volume-sobre-raio} 
(F_{\beta}^{r,p})^{*}\vol_{\pi^{-1}_r}= 
\sqrt{ 
\begin{vmatrix}
\metric( \frac{\partial}{\partial \theta_i} F_{\beta}^{r,p},\frac{\partial}{\partial \theta_j} F_{\beta}^{r,p} )
\end{vmatrix}
} 
d\theta_{1}\wedge \cdots\wedge \mathrm{d}\theta_{l}.
\end{equation}
Therefore:
\begin{equation}
\label{eq-4-lemma-volume-sobre-raio} 
\lim_{r\to 0} \frac{ (F_{\beta}^{r,p})^{*}\vol_{\pi^{-1}_r}}{r^{l}}=
\sqrt{ 
\begin{vmatrix}
\mathrm{g}_{p}( \frac{\partial \psi_{\beta} }{\partial \theta_i} ,\frac{\partial \psi_{\beta}}{\partial \theta_j} )
\end{vmatrix}
} 
d\theta_{1}\wedge \cdots\wedge \mathrm{d}\theta_{l}.
\end{equation}
Note that $\mathcal{T}_{\gamma(r)}$ converges to $ T_{\gamma(0)}B$ and hence from  
Eq. \eqref{equation-definition-h-lemma-volume-sobre-raio} (i.e. definition of $h$) 
we have:
\begin{equation}
\label{eq-5-lemma-volume-sobre-raio} 
\lim_{r\to 0} h\circ F^{r}_{\beta}(\theta,p)=1.
\end{equation}

Now set $\mathrm{P}_{\beta}=F^{1}_{\beta}(V_{\beta},\vizinhancaB_{\alpha})$
and $\mathrm{P}_{\beta,p}=F^{1,p}_{\beta}(V_{\beta})$, and let $\{\rho_{\beta} \}$
be the partition of unity subordinate  to $\{\mathrm{P}_{\beta}\}$ as well as $\mathrm{P}_{\beta, p}$.  
Consider the homothety transformation $\mathbb{P}\colon E-\{0\}\to E^{1}$ defined as 
$\mathbb{P}(\exp(r\xi))= \exp( \xi)$, with $r\in(0,1]$.  Note that
$\rho^{r}_{\beta}=\rho_{\beta}\circ\mathbb{P}$ is partition of unity subordinate
to $\mathrm{P}_{r,\beta}=F^{r}_{\beta}(V_{\beta},\vizinhancaB_{\alpha})$ and $\mathrm{P}_{r,\beta,p}=F^{r,p}_{\beta}(V_{\beta})$. Then we have

\begin{align*}
\lim_{r\to 0} I(p,r) &= \lim_{r\to 0} \sum_{\beta} \int_{W^{r,\beta,p}} \rho_{\beta}^{r} \, h \, \frac{\vol_{\pi^{-1}_{r}(p)}}{r^{l}}\\
& = \lim_{r\to 0} \sum_{\beta} \int_{V_{\beta}} \rho_{\beta}^{r} \circ F_{\beta}^{r,p}(\theta) 
\,  \, h\circ F^{r,p}_{\beta}(\theta) \frac{(F_{\beta}^{r,p})^{*} \vol_{\pi^{-1}_{r}(p)}}{r^{l}}\\
& = \lim_{r\to 0} \sum_{\beta} \int_{V_{\beta}} \rho_{\beta} \circ F_{\beta}^{1,p}(\theta) 
\, \, h\circ F^{r,p}_{\beta}(\theta) \frac{(F_{\beta}^{r,p})^{*} \vol_{\pi^{-1}_{r}(p)}}{r^{l}}\\
& =  \sum_{\beta} \int_{V_{\beta}} \rho_{\beta} \circ F_{\beta}^{1,p}(\theta)  \sqrt{ 
\begin{vmatrix}
\mathrm{g}_{p}( \frac{\partial \psi }{\partial \theta_i} ,\frac{\partial \psi}{\partial \theta_j} )
\end{vmatrix}
} 
d\theta_{1}\wedge \cdots\wedge \mathrm{d}\theta_{l}\\
& =   N \big|L_{e_{p_0}}^{0}\big|.
\end{align*}
\end{proof}


\section{A basic Rellich–-Kondrachov--Hebey--Vaugon theorem}
\label{Section-Rellich-Kondrachov}

Hebey and Vaugon proved in \cite{hebey3} that 
when a compact Riemannian manifold $(M,\metric)$ admits
an isometric action by a compact group  $G$ 
(assuming that $\dim G(x)>0, \forall x\in M$) then  the subset of Sobolev $G$-basic functions $W^{1,p}(M)^{G}$, admits
better Sobolev embedding than the one we have in general for Riemannian manifolds. To prove this they used the existence of slices (which by
hypothesis have dimension smaller than $M$) to  reduce
the local study of $G$-basic functions to functions on the slices.

Using Remark \ref{remark-equivalent-definition-SRF}, which generalizes \cite[Lemma 1, p.863]{hebey3}, 
it can be checked  that the  proof of \cite[Corollary 1, p.867]{hebey3}, also presented at \cite[Theorem 9.1]{Hebey_2000}), 
holds for SRF without any further changes. More precisely,  we can reformulate \cite[Theorem 9.1, p.252]{Hebey_2000}) as follows:

\begin{theorem}[A basic Rellich--Kondrachov--Hebey--Vaugon theorem]\label{thm:reillich}
Suppose that $M$ is a connected compact Riemannian manifold and $\mathcal{F}$ is a SRF on $M$ whose leaves are closed with no trivial leaf (that is, $d^*:= \min_{x\in M}\mathrm{dim}L_x\geq 1$). For any $\infty > p\geq 1$ there exists $p_0 > p^* := np/(n-p)$ such that for any $1\leq q<p_0$, the canonical embedding $W^{1,p}(M)^\mathcal{F} \hookrightarrow L^q(M)$ is completely continuous. That is, compact. Moreover, if $p \geq n-d^*$, the results hold for any $q\geq 1$.
\end{theorem}

\begin{remark}
\label{remark-thm:reillich}
A result analogous to Theorem \ref{thm:reillich}
holds for basic linearizable functions $\closedlinearizedbasic$.

\end{remark}

\begin{proof}
The proof is a straightforward modification of Theorem 9.1 in \cite[p. 252-253]{Hebey_2000}. 
		For the sake of completeness, let us briefly review the main   ideas.

From Remark \ref{remark-equivalent-definition-SRF} we can check that for each point $x\in M$ there exists an open set $\Omega$ satisfying
\begin{enumerate}[(i)]
    \item $\varphi(\Omega) = P\times S\subset \mathbb{R}^\mathrm{d}\times \mathbb{R}^{n-d};$
    \item $\forall y \in \Omega,~P\times \mathrm{pr}_2(\varphi(y)) \subset \varphi(L_y\cap \Omega)$, where $\mathrm{pr}_2 \colon \mathbb{R}^{d}\times
\mathbb{R}^{n-d} \to \mathbb{R}^{n-d}$ is the second-factor projection. \label{item:restritor}
\end{enumerate}
Using the compactness of $M$ and the fact that the leaves of $\cal F$ are closed, we can find finitely many coordinate charts $\{(\Omega_{i},\varphi_i)\}$ with $\Omega_i$ satisfying (i)-(ii).

Item \eqref{item:restritor} implies that any smooth basic function $u\colon M\rightarrow \mathbb{R}$ is such that
    $u\circ \varphi_i^{-1} \colon P_i\times S_i \rightarrow \mathbb{R}$ is constant along the first factor, so we identify $u\circ \varphi_i^{-1}(x,y) \equiv \widetilde{u}_i(y) \colon S_i \rightarrow \mathbb{R}$. For integration theory purposes, it is convenient, and we can assume without loss of generality that $\varphi_i$ is defined on some open set $\widetilde\Omega_i$ containing the closure $\overline{\Omega}_i$ such that $\varphi_i(\widetilde\Omega_i) = P_i\times S_i$ and $\overline S_i\subset\widetilde S_i$. In this manner, we obtain that $\widetilde u_i\in C^{\infty}(\bb R^{n-d_i})$.

    Proceeding \emph{mutatis mutandis} as in \cite[p.253]{hebey1}, for fixed $i$ and $1\leq p< \infty$ we get the existence of positive numbers $A_i,~\widetilde A_i$ (which do not depend on $u$) such that
    \begin{align}
        \int_{\Omega_i}|u|^p\mathrm{d}\mu_{\metric} &= \int_{P_i\times S_i}\left(|u|^p\sqrt{{\mathrm{det}~\metric}^i}\circ \varphi_i^{-1}\right)\mathrm{d}x\mathrm{d}y\\
        &\leq A_i\int_{P_i\times S_i}|u\circ \varphi_i^{-1}|^p\mathrm{d}x\mathrm{dy}\\
        &= \widetilde{A}_i\int_{S_i}|\widetilde u_i|^p\mathrm{d}y \label{eq:primeiraqueajuda}
    \end{align}
    Similarly we find $B_i$ (not depending on $u$) and such that
    \begin{align}
        \int_{\Omega_i}|u|^p\mathrm{d}\mu_{\metric} \geq B_i\int_{S_i}|\widetilde u_i|^p\mathrm{d}y.
    \end{align}
    and $\widetilde B_i > 0$ (also not depending on $u$)
    \begin{align}
        \int_{\Omega_i}|\nabla u|^p\mathrm{d}\mu_{\metric} \geq \widetilde B_i\int_{S_i}|\nabla{\widetilde u}_i(y)|^p\mathrm{d}y. \label{eq:ultimaqueajuda}
    \end{align}

   Set $d^* := \min_id_i > 0$. For any $i$, $n-d_i \leq n-d^*$. Using that each $\Omega_i$ has a smooth boundary, one can combine inequalities \eqref{eq:primeiraqueajuda}-\eqref{eq:ultimaqueajuda} with the Sobolev embedding theorem for bounded open sets on $\mathbb{R}^{n-d_i}$ (cf. \cite[Theorem 6, p.270]{evans2010partial}) to get the following:
   
 \textbf{Assume first that $1\leq p < n-d_i$}.
      For any $1\leq q <\frac{(n-d_i)p}{n - d_i -p}$ there exists $C_i >0$ such that
      \begin{align*}
      \left(\int_{\Omega_i}|u|^q\mathrm{d}\mu_{\metric}\right)^{1/q} &\leq C_i\left(\left(\int_{\Omega_i}|\nabla u|^p\mathrm{d}\mu_{\metric}\right)^{1/p}{+}\left(\int_{\Omega_i}|u|^p\mathrm{d}\mu_{\metric}\right)^{1/p}\right).
    \end{align*}

    Moreover, denoting by $|i|$ the cardinality of the collection $\{(\Omega_i,\varphi_i)\}$ it is straightforward obtaining
    \begin{align*}
        \left(\int_M|u|^q\mathrm{d}\mu_{\metric}\right)^{1/q}&\leq \sum_{i=1}^{|i|}\left(\int_{\Omega_i}|u|^q\mathrm{d}\mu_{\metric}\right)^{1/q}\\
        \sum_{i=1}^{|i|}\left(\left(\int_{\Omega_i}|\nabla u|^p\mathrm{d}\mu_{\metric}\right)^{1/p}{+}\left(\int_{\Omega_i}|u|^p\mathrm{d}\mu_{\metric}\right)^{1/p}\right)&\leq |i|\left(\left(\int_{M}|\nabla u|^p\mathrm{d}\mu_{\metric}\right)^{1/p}{+}\left(\int_{M}|u|^p\mathrm{d}\mu_{\metric}\right)^{1/p}\right)
    \end{align*}

      Therefore, we have that for $p_0 := \min_i\frac{(n-d_i^*)p}{n-d_i^* - p}$, for any $1\leq q < p_0$ it holds we can find a universal constant $N = N(A_i,B_i,\widetilde A_i,\widetilde B_i,C_i)$ (depending on the covering) such that
    \begin{equation}\label{eq:bounding}
        \left(\int_M|u|^q\mathrm{d}\mu_{\metric}\right)^{1/q}\leq N\left(\left(\int_M|\nabla u|^p\mathrm{d}\mu_{\metric}\right)^{1/p}{+}\left(\int_M|u|^p\mathrm{d}\mu_{\metric}\right)^{1/p}\right).
    \end{equation}

    \textbf{Assume now that $p \geq n-d^*$}. In this case, following the same argumentation as before but also using that $p\geq n-d^*\geq n-d_i$ for each $i$, we have that for any fixed $q\geq 1$, we can bound
       \begin{align*}
      \left(\int_{\Omega_i}|u|^q\mathrm{d}\mu_{\metric}\right)^{1/q} &\leq C_i\left(\left(\int_{\Omega_i}|\nabla u|^p\mathrm{d}\mu_{\metric}\right)^{1/p}{+}\left(\int_{\Omega_i}|u|^p\mathrm{d}\mu_{\metric}\right)^{1/p}\right).
    \end{align*}
The same computation as before ensures an analogous inequality to that of Equation \eqref{eq:bounding}.

Lastly, one observes that the former shows the validity and the continuity of the embeddings in question in the statement. Standard arguments, 
such as in \cite[Chapter 5,Theorem 1]{evans2010partial}, 
show these embeddings are compact for any $q \geq 1$ in the former case, and any $q < \frac{(n - d^*) p}{n - d^* - p}$ in the latter case.
\end{proof}

%


\section{An abstract setup for variational problems with (and via) symmetries}
\label{section-An abstract setup for variational problems with (and via) symmetries}

As we pointed out earlier,  this paper does not aim to delve deeply into analytical problems. 
However, we can describe one of the common methods employed in the Calculus of Variations in combination with our symmetric criticality principle to illustrate the strength of the theory.

In Subsection \ref{subsectio-motivate-example}, we present a specific family of partial differential equations that mathematicians typically study, known in the literature as ``$p$-Kirchhoff bi-non-local problems.'' 
Then, in Subsection \ref{Section-Applications-PP}, we will provide a proof of Proposition \ref{corollary:kirshofinho} 
by applying Theorem \ref{theorem-simple-version-Palais-principle-variational-formulation} 
(or its generalization, Theorem \ref{theorem-Palais-principle-variational-formulation}), 
alongside classical arguments from the Calculus of Variations. Finally, in Section \ref{Subsection-Proof of Theorem-thm:prettygeneralzinho}
we generalize the discussion presented in the Subsection \ref{Section-Applications-PP} 
 to the  functional energy presented in Theorem \ref{proposition-operator-J-kirschoff-generalized}.

For purely didactic reasons, we present here the proofs for the case of AVP orbit-like foliations.
But we point out that by Definition \ref{definition-linearized-basic}, Remark \ref{remark-thm:reillich} and 
Theorem \ref{theorem-Palais-principle-variational-formulation}
we do not need the orbit-like hypothesis.
Without the orbit-like hypothesis, the reader should replace $W^{1,p}(M)^{\mathcal{F}}$ with $\closedlinearizedbasic$, 
and Theorem \ref{theorem-simple-version-Palais-principle-variational-formulation} with Theorem~\ref{theorem-Palais-principle-variational-formulation}  throughout the proofs in this section.

\subsection{A motivating example} 
\label{subsectio-motivate-example}
$p$-Kirchhoff bi-non-local problems are related to a stationary model for the
vibration of an elastic string using a variation of the classical wave equation. 
Following \cite{martinez} or \cite{ALVES200585}, there has been recent interest in proving the existence of critical (minimal) points for functionals defined in $W^{1,p}(M)$ of the form
\begin{equation}\label{eq:energy0}
J_{\lambda}(u) := \weight(\|u\|^p) - \frac{\lambda}{r+1}\left(\int_MF(u,x) \density_{\metric}\right)^{r+1},
\end{equation}
where we have $\weight (t):= \int_0^t \mathrm{m}(s)\mathrm{d}s$ for a continuous non-negative function $\mathrm{m}\colon \bb R_+ \rightarrow \bb R_+$  with additional hypotheses, depending on the PDE being considered. A typical critical point $u\in W^{1,p}(M)$ of $J_\lambda$ is characterized by satisfying 
the following for all $v \in W^{1,p}(M)$:
\begin{align}\label{eq:criticalweak}
0 =\mathrm{d}J_{\lambda}(u)[v] := \mathrm{m}(\|u\|^p)\int_M|\nabla u|^{p-2}\metric(\nabla u,\nabla v)\density_{\metric}
-\lambda\left(\int_MF(u,x)\density_{\metric}\right)^r\int_M f(u,x)v \, \density_{\metric}.
\end{align}

We say that  $u\in W^{1,p}(M)$ satisfying \eqref{eq:criticalweak} is a \emph{weak solution} to some PDE. 
For example, Equation \eqref{eq:criticalweak} represents in the literature the existence of a weak solution for the PDE 
\[
-\mathrm{m}(\|u\|^p)\mathrm{div}_{\metric}(|\nabla u|^{p-2}\nabla u) = -\mathrm{m}(\|u\|^p)\Delta_p u 
= \lambda f(u,x)\left[\int_MF(u,x)\density_{\metric}\right]^r,~p\geq 2,~\lambda \in \bb R,
\]
where $f\colon \bb R\times M \rightarrow \bb R$ is a function satisfying a minimum regularity requirement, which makes the problem well-defined, $\int_0^tf(s,x)\mathrm{d}s:= F(t,x)$. The term $\mathrm{m}$ is referred to in the literature as a \emph{weight}, and usually it is a non-negative continuous function. 
Particular examples of what we have  discussed can be found at
\cite{correa} and \cite{MENDEZ2021124671}.

Next, some common techniques in the field of Calculus Variations shall be employed to the proof of  Proposition \ref{corollary:kirshofinho}.


\subsection{Proof  of Proposition \ref{corollary:kirshofinho}}
\label{Section-Applications-PP}

\begin{proof}

		We first observe that  $J_{\lambda}$ is $C^1$. In fact,
		$u\mapsto \int_M|\nabla u|^p\density_{\metric}$ is $C^1$ since 		$p\geq 2$. 		Moreover, assuming that the map $t\mapsto \weight (t)$ is $C^1$, the chain rule implies that 
		$u\mapsto \weight\big(\frac{1}{p}\int_M|\nabla u|^p\density_{\metric}\big)$ is $C^1$.
	Using that $M$ is $n$-dimensional, with $n\geq 3$ and $n>p\geq 2$, we get:
    \begin{align*}
        p^* &= \frac{np}{n-p}\\
        &= \frac{p}{1-\tfrac{p}{n}}\\
        &> p \geq 2.
    \end{align*}
    Thus $u\mapsto \int_M|u|^{p^*}\density_{\metric}$ is $C^1$. 
		Since $r > 1$ we get that $u\mapsto \left(\int_M|u|^{p^*}\density_{\metric}\right)^{r+1}$ is $C^1$.

Due to Theorem  \ref{theorem-simple-version-Palais-principle-variational-formulation},  to prove the existence of a weak basic solution, i.e.  
$b_0\in W^{1,p}(M)^{\mathcal{F}}$
so that  $\mathrm{d} J(b_0)=0$ it suffices  to prove the existence
of $b_0\in W^{1,p}(M)^{\mathcal{F}}$ such that $\mathrm{d} J(b_0)b=0$ for all $b\in W^{1,p}(M)^{\mathcal{F}}$.
In other words, we can reduce our search for weak solutions to the closed vector space $W^{1,p}(M)^{\mathcal{F}}$.

We fix $\epsilon > 0$ and consider the subset 
\[
\mathbf{M}_{\epsilon}^{\cal F} := \left\{u \in W^{1,p}(M)^{\cal F} : \int_M|u|^{p^*}\density_{\metric} = 
\epsilon^{\tfrac{1}{r+1}}(r+1)^{\tfrac{1}{r+1}}p^* \right\}.
\]
It is straightforward to check that $\mathbf{M}_{\epsilon}^{\cal F}$ is non-empty.

\begin{claim}
\emph{$\mathbf{M}_{\epsilon}^{\cal F}$ is a closed co-dimension one submanifold of $W^{1,p}(M)^{\cal F}$.} 
\end{claim}

\begin{proof}
We first check that for each  $u\in \mathbf{M}_{\epsilon}$, the $C^1$-map 
$u\mapsto G(u) := \int_M|u|^{p^*}\density_{\metric} -\epsilon^{\tfrac{1}{r+1}}(r+1)^{\tfrac{1}{r+1}}p^*$ has surjective derivative.
Indeed since $\mathrm{d}G(u)[v] = p^{*}\int_M|u|^{p^*-2}uv\, \density_{\metric}$ for every $v\in W^{1,p}(M)^{\cal F}$, we have that
$\mathrm{d}G(u)[u] = p^*\int_M|u|^{p^*}\density_{\metric} = \epsilon^{\tfrac{1}{r+1}}(r+1)^{\tfrac{1}{r+1}}{p^*}^2 > 0$. We fix $u\in \mathbf{M}_{\epsilon}^\cal F$ and consider $\cal V_u := \{v \in W^{1,2}(M)^{\cal F} : \mathrm{d}G(u)[v] = p^{*}\int_M|u|^{p^*-2}uv\, \density_{\metric} = 0\}$. Since $u\mapsto G(u)$ is a $C^1$-map it follows that $\cal V_u$ is closed in $W^{1,p}(M)^{\cal F}$. Moreover, $W^{1,p}(M)^{\cal F}/\cal V_u$ is closed and 
\[
W^{1,p}(M)^{\cal F} = \left(W^{1,p}(M)^{\cal F}/\cal V_u\right) \oplus \cal V_u.
\]
We thus have verified the needed hypotheses for the implicit function theorem for Banach spaces 
-- also known as the Sard--Smale Theorem (Corollary 1.5 in \cite{sardsmale}). Therefore,
 $\mathbf{M}_{\epsilon}^{\cal F}$ is a closed co-dimension one submanifold of $W^{1,p}(M)^{\cal F}$. 
\end{proof}

\begin{claim}
\emph{$\mathbf{M}_{\epsilon}^{\cal F}$ is  weakly closed.}
\end{claim}

\begin{proof}
Take $\{u_m\} \subset \mathbf{M}_{\epsilon}^\cal F$ a weakly-convergent sequence with limit $u$, i.e. $u_m \rightharpoonup u$ in $W^{1,p}(M)^{\cal F}$. 
We claim that $u \in \mathbf{M}_{\epsilon}^\cal F$. The basic version of Rellich--Kondrachov Theorem \ref{thm:reillich} implies that $\{u_m\}$ strongly converges to $u$ with respect to the $L^{p^*}(M)$-norm. Hence,
\[
\epsilon^{\tfrac{1}{r+1}}(r+1)^{\tfrac{1}{r+1}}p^* = \lim_{m\to\infty}\int_M|u_m|^{p^*}\density_{\metric} = \int_M|u|^{p^*}\density_{\metric}.
\]
\end{proof}

\begin{claim}
\emph{$J_{\lambda}\Big|_{\mathbf{M}_{\epsilon}^{\cal F}}$ is weakly lower-semicontinuous and coercive.}
\end{claim}

\begin{proof}
To check $J_{\lambda}\Big|_{\mathbf{M}_{\epsilon}^\cal F}$ is coercive, we note that since we are considering elements of $W^{1,p}(M)^{\cal F}$ 
with the fixed constraint $\int_M|u|^{p^*}\density_{\metric} = \epsilon^{\tfrac{1}{r+1}}(r+1)^{\tfrac{1}{r+1}}p^*$, 
verifying the claim for $J_{\lambda}\Big|_{\mathbf M_{\epsilon}^{\cal F}}$ is reduced to understanding the behavior 
of $J_{\lambda}\Big|_{\mathbf{M}_{\epsilon}^\cal F}$ when $\int_M|\nabla u|^p\density_{\metric} \rightarrow +\infty$. 
Moreover, using that $\weight \colon \bb R_+\rightarrow \bb R_+$ is continuous and 
$\lim_{t\to\infty}\weight(t) =\infty$, we have that $\int_M|\nabla u|^p\density_{\metric}\rightarrow +\infty$ implies that
\[
\weight \left(\tfrac{1}{p}\int_M|\nabla u|^p\density_{\metric}\right) \rightarrow +\infty.
\]
Since
$J_{\lambda}(u) = \weight \left(\frac{1}{p}\int_M|\nabla u|^p\density_{\metric}\right)-\lambda\epsilon$, we conclude that 
$J_{\lambda}\Big|_{\mathbf{M}^{\cal F}_{\epsilon}}$ is coercive. It is only left to check that $J_{\lambda}\Big|_{\mathbf{M}_{\epsilon}^\cal F}$ 
is weakly lower-semicontinuous.

Using that $\mathbf{M}_{\epsilon}^{\cal F}$ is weakly closed, we consider a weakly-convergent sequence
$\mathbf{M}_{\epsilon}^{\cal F}\supset \{u_m\} \rightharpoonup u \in \mathbf{M}_{\epsilon}^\cal F$. 
On the one hand since $\|\cdot\|_{1,p}$ is weakly lower-semicontinuous, and by the Rellich--Kondrachov Theorem we have
$\int_M|u_m|^p\density_{\metric}\rightarrow \int_M|u|^p\density_{\metric}$, it follows that 
\[
\liminf_{m\to\infty}\int_M|\nabla u_m|^p\density_{\metric} \geq \int_M|\nabla u|^p \density_{\metric}.
\]  
On the other hand, since $\weight \colon \bb R_+ \rightarrow \bb R_+$ is continuous, it is lower-semicontinuous. Being $\weight$ convex, the function $W^{1,p}(M)^{\cal F}\ni u\mapsto  \weight \left(\frac{1}{p}\int_M|\nabla u|^p\density_{\metric}\right)$ is continuous (in the norm topology) and convex, thus weakly lower-semicontinuous (see \cite[Corollary 3.9+Remark 6, p.61]{brezis2010functional}). 
Therefore, $\liminf_{m\to\infty}\weight\left(\frac{1}{p}\int_M|\nabla u_m|^p\density_{\metric}\right)
 \geq \weight\left(\frac{1}{p}\int_M|\nabla u|^p\density_{\metric}\right)$  and thus 
\[
\liminf_{m\to\infty}J_{\lambda}(u_m) \geq J_{\lambda}(u). \]
\end{proof}

Since $J_\lambda|_{\mathbf{M}^{\cal F}_{\epsilon}}$ is coercive and weakly lower-semicontinuous, there exists $u_0\in \mathbf{M}_{\epsilon}^\cal F$ which is an extremal  point (a minimum) of $J_\lambda|_{\mathbf{M}^{\cal F}_{\epsilon}}$. The Lagrange Multiplier Theorem (see \cite{Botelho2013}) ensures the existence of some $\theta \in \bb R$ such that
\[
\theta\mathrm{d}G(u_0)[v]= \mathrm{d}J_{\lambda}(u_0)[v], \ \forall v\in W^{1,p}(M)^{\cal F}.
\]
Writing explicitly, we have $\forall v \in W^{1,p}(M)^{\cal F}$ 
\begin{align*}
\theta p^* \int_M|u_{0}|^{p^*-2}u_{0}v\density_{\metric} &=
\mathrm{m}\left(\frac{1}{p}\int_M |\nabla u_{0} | ^p \density_{\metric}\right)\int_M|\nabla u_{0}|^{p-2}\metric(\nabla u_{0},\nabla v)\density_{\metric}\\
&-\lambda\left(\int_M\frac{1}{p^*}|u_{0}|^{p^*}\density_{\metric}\right)^r\int_M |u_0|^{p^*-2}u_{0}v \, \density_{\metric}\\
\end{align*}
Taking $\lambda^*=\lambda^*(\epsilon) := \lambda\epsilon^{\tfrac{r}{r+1}}(r+1)^{\tfrac{r}{r+1}} + \theta p^*$, we can rewrite the former identity as
\begin{multline*}
0=  \mathrm{m}\left(\frac{1}{p}\int_M|\nabla u_0|^p\density_{\metric}\right)\int_M|\nabla u_0|^{p-2}\metric(\nabla u_0,\nabla v)\density_{\metric}
-\lambda^* \int_M|u_{0}|^{p^*-2}u_{0}v\density_{\metric}
=\mathrm{d}J_{\lambda^*}(u_0)[v],\\~\forall v \in W^{1,p}(M)^{\cal F}.
\end{multline*}

Since the extremal point  $u_0\in \mathbf{M}_{\epsilon}^\cal F$  depends on $\epsilon$, by varying $\epsilon$ we  obtain the claimed sequence.
\end{proof}



\subsection{A generalized formulation of $p$-Kirchhoff problem}
\label{Subsection-Proof of Theorem-thm:prettygeneralzinho}

In what follows, we  generalize  the ideas of the  proof of  Proposition \ref{corollary:kirshofinho}, which we can consider a 
general formulation of $p$-Kirchhoff's problem.

\begin{theorem}\label{thm:prettygeneralzinho}
Let $(M^n,\metric)$,  be a  $n$-compact  Riemannian manifold 
with an  AVP  foliation $\mathcal{F}$, where $n\geq 3$. We assume that the leaves of  $\mathcal{F}$ are closed, and each leaf has positive dimension. 
Set 
\[
J_{\lambda}(u):= \weight\Big(\int_{M} | \nabla u |^{p}\density_{\metric}\Big)\int_M\cal L(|\nabla u|^2,u,x)\density_{\metric} 
- \frac{\lambda}{c}\left(\int_MF(u,x)\density_{\metric}\right)^{r+1},
\] 
where $r, c, \lambda >0$ and $p\in [2,n[$. Assume that $J_{\lambda}\colon W^{1,p}(M)\to\mathbb{R}$ is a 
 $C^1$-energy functional  for a continuous non-negative map $\weight\geq 0$, and that the Lagrangian $\mathcal{L}$ and $F$ areof class $C^{1}$ and $\mathcal{F}$-basic.
Moreover, suppose that:
\begin{enumerate}[$(A)$]
    \item The map $u\mapsto  \weight \big(\int_{M} | \nabla u |^{p}\density_{\metric}\big) \int_M\cal L(|\nabla u|^2,u,x)\density_{\metric}$ 
		is a $C^1$ weakly-lower semi-continuous map and there exists $C_1$ such that 
    \begin{equation*}
        C_1\|u\|_{1,p}^p \leq \weight \Big(\int_{M} | \nabla u |^{p}\density_{\metric}\Big) \int_M\cal L(|\nabla u|^2,u,x)\density_{\metric}.
    \end{equation*} \label{item:essecoerciva}
    \item  The map $u\mapsto \int_MF(u,x)\density_{\metric}$ is a $C^1$-map 
		such that $x\mapsto F(0,x)$ is constant and \label{item:essecompacta} \begin{enumerate}[$(a)$]
        \item \label{eq:weaklyclosedvaivir} $\left|f(t,x) := \frac{\partial}{\partial s}\Big|_{s=t}F(s,x)\right| \leq a(x) + k|t|^{p^*-1}$, 
				where $k\in (0,\infty)$ and $a\in L^{p^*}(M)$;
        \item \label{eq:sardsmale}for each $u\in W^{1,p}(M)$, the functional 
				$W^{1,p}(M) \ni v\mapsto \left(\int_M F(u,x)\density_{\metric}\right)'(u)[v]$ is not identically zero.
    \end{enumerate}
\end{enumerate}
Then, there exist infinitely many weak  solutions to our problem. More precisely  
there exist a sequence  $\{\lambda^{*}_{i}\}_{i=1}^{\infty} \subset \bb R$  and 
a sequence of distinct non-zero functions
 $\{ u_{i} \}_{i=1}^{\infty}\subset W^{1,p}(M)$ such that
$\mathrm{d}J_{\lambda^{*}_{i}}(u_i) = 0$
\end{theorem} 
\begin{remark}\label{rem:remakrs-regarding-analytic-hypotheses}
\begin{enumerate}
\item Hypothesis given by Item \eqref{eq:sardsmale} hold, for instance, if condition \eqref{eq:weaklyclosedvaivir} holds 
and  $\int_Mf(u,x)^2\density_{\metric} \neq 0$ for every $u\in W^{1,p}(M)^{\cal F}$  holds. 
Indeed, $\mathrm{d}\left(u\mapsto \int_M F(u,x)\density_{\metric}\right)(u)[v] = \int_Mf(u,x)v\density_{\metric}$. 
The hypothesis given by item \eqref{eq:weaklyclosedvaivir} ensures that $f(u,x) \in W^{1,p}(M)^{\cal F}$, 
so we can choose $v = f(u,x)$, thus concluding the remark.
\item As stated above, the orbit-like hypothesis is not necessary for the proofs of Theorems \ref{corollary:kirshofinho} and \ref{thm:prettygeneralzinho}. Nonetheless, in the particular case where the SRF in  Theorem \ref{thm:prettygeneralzinho} is an orbit-like AVP foliation, 
the infinitely many weak solutions may be chosen to be basic, i.e. $\{u_i\}_{i=1}^{\infty}\subset W^{1,p}(M)^{\mathcal{F}}$.
\end{enumerate}
\end{remark}

\begin{proof}[Proof of Theorem \ref{thm:prettygeneralzinho}]
    
		As the first step, we observe that we can find $\pm \epsilon \in \bb R$ such that
    \[
    \mathbf{M}_{\epsilon}^{\cal F} := \{u\in W^{1,p}(M)^{\cal F} : \int_MF(u,x)\density_{\metric} = (\pm c\epsilon)^{1/r+1}\}
    \]
		is non-empty, otherwise the term $\int_MF(u,x) \density_{\metric}$ would be absent from $J_{\lambda}$. 
		From now on, we assume without loss of generality that we are in the case of $+\epsilon$.
		We define $G(u):= \int_MF(u,x)\density_{\metric} - (c\epsilon)^{1/r+1}$ and note that 
		$\mathbf{M}_{\epsilon}^{\cal F}$ is also closed, since $u\mapsto \int_MF(u,x)\density_{\metric}$ is $C^1$.  
    \begin{claim}
       $\mathbf{M}_{\epsilon}^\cal F$ is a closed co-dimension one submanifold of $W^{1,p}(M)^{\cal F}$. 
    \end{claim}
    \begin{proof}
     For any $u\in \mathbf{M}_{\epsilon}^\cal F$ since $u\mapsto \int_MF(u,x)\density_{\metric}$ is $C^1$ and that $f(u,x) = \partial_{s=u}F(s,x)$, we have
    \[
    \mathrm{d}G(u)[v] := \int_Mf(u,x)v \density_{\metric}.
    \]
    For $u$ fixed, from the hypothesis \eqref{eq:sardsmale} we get the existence of some $v\in W^{1,p}(M)^{\cal F}$ such that 
    $\mathrm{d}G(u)[v] \neq 0$. Thus, the derivative of $G$ at any $u\in \mathbf{M}_{\epsilon}^\cal F$ is surjective. 
		In addition to that, the linear subspace 
		$\cal V_u := \{v\in W^{1,p}(M)^{\cal F} : \int_M f(u,x)v \density_{\metric} = 0\}$ is closed and 
		$W^{1,p}(M)^{\cal F} = \left(W^{1,p}(M)^{\cal F}/\cal V_u\right)\oplus \cal V_u$. 
		The hypotheses of the Banach-space version of the regular value theorem, also known as the Sard--Smale Theorem (Corollary 1.5 in \cite{sardsmale}), 
		are satisfied and  we conclude that
		$\mathbf{M}_{\epsilon}^\cal F$ is a co-dimension one submanifold of $W^{1,p}(M)^{\cal F}$.
    \end{proof}

\begin{claim}
    $\mathbf{M}^{\cal F}_{\epsilon}$ is weakly closed.
\end{claim}
\begin{proof}
Take $\mathbf{M}_{\epsilon}^{\cal F}\supset \{u_m\} \rightharpoonup u\in W^{1,p}(M)^{\cal F}$. 
Due to the basic Rellich--Kondrachov Theorem \ref{thm:reillich}, we get that 
$\int_M|u_m|^{p^*} \density_{\metric} \rightarrow \int_M|u|^{p^*}\density_{\metric}$.  
Since $p^*\geq 2$, up to passing to a subsequence, we can suppose that $\{u_m\}$ converges pointwise almost everywhere to $u$. 
Since for a fixed $x\in M$ the map $u\mapsto F(u(x),x)$ continuous, then $F(u_m,\cdot)\colon x\mapsto F(u_m(x),x)$ 
converges pointwise a.e. to $F(u(x),x)$. We now prove that the following holds:
    \begin{enumerate}[(i')]
    \item \[
    \left|F(u_{m} (x),x)-F(0,x)\right|\leq a(x)|u_m(x)| + \frac{k}{p^*}|u_m(x)|^{p^*}~\text{a.e. }x\in M 
    \]
        \item $\int_{M} a |u_m|\density_{\metric} \rightarrow \int_{M} a |u|\density_{\metric}$.  
    \end{enumerate}
    With these points, the Generalized Dominated Convergence Theorem implies
		\[
		\int_MF(u_m,x) \density_{\metric}\rightarrow \int_M F(u,x)\density_{\metric}.
		\]
    This verifies that $\mathbf{M}_{\epsilon}^{\cal F}$ is weakly closed, 
		since if $\{u_m\}\subset\mathbf{M}_{\epsilon}^{\cal F}$ weakly converges to $u$, 
		then $(c\epsilon)^{\tfrac{1}{r+1}}=\int_MF(u_m,x)\density_{\metric}$ for each $m$. 
		Under the assumption $\lim_{m\to\infty}\int_{M} F(u_{m},x)\density_{\metric}=\int_MF(u,x)\density_{\metric}$, 
		we get that $\int_MF(u,x)\density_{\metric}=(c\epsilon)^{\tfrac{1}{r+1}}$, i.e. $u\in \mathbf{M}^{\cal F}_{\epsilon}$.

     Fix $m\geq 1$ arbitrarily. Observe that both sets
     \[
     [u\geq 0]:=\{x\in M:u(x)\geq 0\},~[u<0]:=\{x\in M:u(x)<0\}
     \]
     are measurable, since $u$ is measurable. Using Item \eqref{eq:weaklyclosedvaivir} 
		and the Fundamental Theorem of Calculus, we get that for $x\in [u\geq 0]$ it holds
    \begin{align*}
        \left|F(u_m(x),x) - F(0,x)\right| &= \left|\int_0^{u_m(x)}f(t,x)\mathrm{d}t\right|\\
        &\leq\int_0^{u_m(x)}|f(t,x)|\mathrm{d}t\\
        &\leq\int_0^{u_m(x)}a(x) + k|t|^{p^*-1}\mathrm{d}t\\
        &= a(x)u_m(x) + k\left(\int_0^{u_m(x)}|t|^{p^*-1}\mathrm{d}t\right)\\
        &= a(x)u_m(x) + \frac{k}{p^*}u_m(x)^{p^*}.
    \end{align*}
    On the other hand, for $x\in [u<0]$, we have
        \begin{align*}
        \left|F(u_m(x),x) - F(0,x)\right| &= \left|\int_{u_m(x)}^0f(t,x)\mathrm{d}t\right|\\
        &=\left|\int_0^{-u_m(x)}f(-t,x)\mathrm{d}t\right|\\
        &\leq\int_0^{|u_m(x)|}|f(-t,x)|\mathrm{d}t\\
        &\leq \int_0^{|u_m(x)|}a(x)+k|t|^{p^*-1}\mathrm{d}t\\
        &= a(x)|u_m(x)|+k|u_m(x)|^{p^*-1}.
    \end{align*}
    This implies (i').

     Let us verify (ii').  
  H\"older's inequality ensures that
    \begin{equation*}
        \left|\int_{M} a\, u_{m}\density_{\metric}\right|\leq 
				\left(\int_{M} |a|^{p^*}\density_{\metric}\right)^{\tfrac{1}{p^*}}\left(\int_{M} |u_{m} |^{\tfrac{np}{n(p-1)+p}}
				\density_{\metric}\right)^{\tfrac{n(p-1)+p}{np}}.
    \end{equation*}
  Since $p^*/(np/n(p-1)+p) \geq 1$, another H\"older's inequality argument implies that there exists $C>0$ such that 
    \[
    \left(\int_{M} |u_{m}|^{\tfrac{np}{n(p-1)+p}} \density_{\metric}\right)^{\tfrac{n(p-1)+p}{np}}\leq C\left(\int_M |u_{m} |^{p^*}  
		\density_{\metric} \right)^{\frac{1}{p^*}}.
		\]
    Thus, 
    \[
    \int_M\left|a\, u_{m}\right|\density_{\metric}\leq  C\left(\int_M|a|^{p^*}\density_{\metric}\right)^{\tfrac{1}{p^*}}\left(\int_M|u_{m}|^{p^*}  
		\density_{\metric} \right)^{\frac{1}{p^*}}.
		\]
    
    Therefore
    \[
    \lim_{m\to\infty}\int_{M} \left|a\, u_{m}\right|\density_{\metric}\leq C\left(\int_{M} |a|^{p^*}\density_{\metric}\right)^{\tfrac{1}{p^*}}\left(\int_M|u|^{p^*}  \density_{\metric} \right)^{\frac{1}{p^*}}< \infty.
    \]      
So there exists $m^*\in \mathbb{N}$ such that the sequence $\left\{\int_M|a\, u_m|\density_{\metric}\right\}_{m\geq m^*}$ is contained in a compact subset of $\mathbb{R}$. Hence, it is possible to extract a converging subsequence $\left\{\int_M|a\, u_{m_k}|\density_{\metric}\right\}_{k\in \bb N}$ to some $L\in \bb R$, and such that the same holds for $U\subset M$ measurable. Since $au_{m_k}$ converges almost everywhere to $au$ 
they converge in measure. Moreover, $au\in L^1(M)$ and thus by \cite[4.5.6. Theorem]{Bogachev} we have that  $L=\int_Ma(x)u(x)\density_{\metric}$. 
Consequently, $\int_{M} a\, u_{m_k}\density_{\metric}\rightarrow \int_{M} a\, u\density_{\metric}$, verifying item (ii').

    Now the Generalized Dominated Convergence Theorem allows us to conclude that 
		$\lim_{m\rightarrow \infty}\int_MF(u_{m},x)\density_{\metric}=\int_MF(u,x)\density_{\metric}$.\end{proof}
 
    Finally, since $J_{\lambda}\Big|_{\mathbf{M}_{\epsilon}^\cal F}$ has the form
    \[
    \weight\Big(\int_M|\nabla u|^p\density_{\metric}\Big)\int_M\cal L(|\nabla u|^2,u,x)\density_{\metric} - \lambda\epsilon,
    \]
    Item \eqref{item:essecoerciva} implies that $J_{\lambda}\Big|_{\mathbf{M}_{\epsilon}^\cal F}$ is coercive and 
		weakly lower-semicontinuous. Therefore, there exists $u\in \mathbf{M}_{\epsilon}^\cal F$ such that $\mathrm{d}J_{\lambda}(u)[v] = 0$ for every $v\in \cal V_u$. The Lagrange Multiplier Theorem \cite{Botelho2013} ensures the existence of some $\theta \in \bb R$ such that
\[
\mathrm{d}J_{\lambda}(u)[v] = \theta\mathrm{d}G(u)[v]~\forall v\in W^{1,p}(M)^{\cal F}.
\]
Similarly to the proof of Proposition \ref{corollary:kirshofinho} we can check that by taking $\lambda^* = \lambda^*(\epsilon,\theta) := 
\frac{\lambda}{c}(r+1)(c\epsilon)^{\tfrac{r}{r+1}} + \theta$, we can rewrite the former equation as
\[
\mathrm{d}J_{\lambda^*}(u)[v] = 0~\forall v \in W^{1,p}(M)^{\cal F}.
\]
Theorem  \ref{theorem-simple-version-Palais-principle-variational-formulation}   
finishes the result. 
\end{proof}

\subsection{Pointed Kirchhoff problems and an open problem}
\label{sec:pointedkirsh}

 \emph{Electrorheological fluids} (see \cite{ruzicka2007electrorheological}) are characterized by their ability to change their mechanical properties drastically when influenced by an external electromagnetic field (see \cite{KANU1999775}): These fluids are such that their viscosity depends on the electric field in the fluid. The field induces (turbulent) string-like formations in the fluid, which are parallel to the field and can raise the viscosity proportionally to the orders of magnitude. Such a phenomenon is known as the \emph{Winslow effect} (\cite{winslow}). This motivates the theory of  \emph{Sobolev Spaces with Variable Exponents}, 
see \cite{diening2011lebesgue} and also \cite{FAN2001424,diening2011lebesgue}.
In this context, one could investigate weak solutions to the following problem:
\begin{equation}\label{eq:pointedkirch}
 -\mathrm{m}\left(\int_{M}\frac{1}{p(x)}|\nabla u|^{p(x)}\mathrm{d}\mu_{\metric}\right)\Delta_{p(x)}u=\lambda |u|^{q(x)-2}
u\left[\int_M\frac{1}{q(x)}|u|^{q(x)}\right]^r 
\end{equation}
where $p, q \in C(M)$,
 $N>q(x) \geq p(x)>1$, $r>0$ and $\lambda$ are real parameters, and $\mathrm{m}\colon\mathbb{R}_+\to \mathbb{R}_+$ is a Kirchhoff function. The $p(x)$-Laplacian operator is defined as 
 \[
 \Delta_{p(x)}u=\mathrm{div}_{\metric}(|\nabla u| ^{p(x)-2}\nabla u).
\]
It is non-homogeneous and depends on the exponent Lebesgue space $L^{p(x)}(M)$ and the variable exponent Sobolev space $W^{1,p(x)}(M)$ 

\begin{question}
\label{question2}
Can  our work  be partially generalized to this context?
\end{question}

\section{Metric partition on Hilbert manifolds}
\label{section-Metric-foliation-Hilbert-manifolds}

In this section, inspired by the proof of Proposition \ref{lemma-new-theorem-Diego-general-version-Palais-Hilbert},
 we will proof Theorem \ref{theorem-CorrosTheorem-A}  
which guarantees a symmetric criticality principle for basic functionals with respect to a metric partition on a Hilbert manifold.

\subsection{Hilbert manifolds}\label{S: Hilbert manifolds}
We begin by presenting the notion of differentiation for Hilbert (or Banach) vector spaces following \cite{Lang} and then present the notion of a Hilbert smooth manifold.

Consider two Hilbert spaces $E$, $F$, an open subset $U\subset E$, and $f\colon U\to F$ a continuous map. We say that \emph{$F$ is Fréchet  
differentiable at $x_0\in U$} if there exists a bounded linear map $A\colon E\to F$ such that
\[
\lim_{\Vert h\Vert_E\to 0 } \frac{\Vert f(x_0+h)-f(x_0)-A(h)\Vert_f}{\Vert h\Vert_E}=0.
\]
The linear map $A$ is uniquely determined, and thus we define the \emph{derivative of $f$ at $x_0$} as $\mathrm{d}f_{x_0} = A$. 
If $f$ is differentiable at each point of $U$, we say that \emph{$f$ is differentiable}, with derivative $\mathrm{d} f\colon U\to \mathrm{L}(E,F)$. Here, $\mathrm{L}(E, F)$ denotes the linear space of all bounded linear operators, and we consider the operator norm on it. In this fashion, $L(E, F)$ is again a Hilbert space to consider higher derivatives. In the case when $\mathrm{d} f$ is a continuous map, we say that $f$ is \emph{of class $C^1$}, and we may consider the derivative of 
$\mathrm{d}f$ at $x_0$, which we denote by $D_{x_0}^2f$. In this way, we may consider higher derivatives of $f$. 
If derivatives of $f$ of any order exist on $U$, then we say that $f$ is a \emph{$C^\infty$-map}, or a \emph{smooth map}.  

Let $\M$ be a second-countable Hausdorff topological space. 
A \emph{Hilbert atlas on $\M$} is a collection of pairs $(U_\alpha,\varphi_\alpha)$, where $\alpha\in \Lambda$ an index set, satisfying:
\begin{enumerate}[(i)]
\item Each $U_\alpha\subset \M$ is an open subset, and $\{U_\alpha\}_{\alpha\in \Lambda}$ is an open cover of $\M$.
\item Each $\varphi_\alpha\colon U_\alpha\to \varphi_\alpha(U_\alpha)\subset E_{\alpha}$ is an homeomorphism  onto an open subset $\varphi_\alpha(U_\alpha)$ of a Hilbert space $E_\alpha$.
\item For any two indices $\alpha,\beta\in \Lambda$, the set $\varphi_\alpha (U_\alpha\cap U_\beta)\subset E_\alpha$ is open.
\item For any two indices $\alpha,\beta\in \Lambda$, the map 
\[
\varphi_\beta\circ\varphi_\alpha^{-1}\colon \varphi_\alpha(U_\alpha\cap U_\beta)\to \varphi_\beta(U_\alpha\cap U_\beta)
\] 
is smooth. 
\end{enumerate}
The space $\M$ equipped with a Hilbert atlas is called a \emph{smooth Hilbert manifold}.

In the case when $E_\alpha=E$ for all indices, we can define a Riemannian metric $g$ on $\M$ (see \cite[Chapter VII]{Lang}), and $(\M,g)$ is called a \emph{smooth Riemannian Hilbert manifold}.


\subsection{Metric partition on Hilbert manifolds}\label{S: Metric foliations on Hilbert spaces} 
In this section, we introduce the notion of metric partition on general smooth Riemannian Hilbert manifolds. The notion of a metric foliation has been defined on finite-dimensional manifolds and finite-dimensional Alexandrov spaces in  \cite{LytchakPre, LytchakPhD}, and also discussed in \cite{KapovitchLytchak20}.

Let $(\M,\metric)$ be a smooth Riemannian Hilbert manifold. For a $C^1$ curve $\gamma\colon [a,b]\to \M$ we define the \emph{length $l_{\metric}(\gamma)$} 
of $\gamma$ as
\[
l_{\metric}(\gamma) = l(\gamma) = \int_a^b \metric(\gamma'(t),\gamma'(t))^{\frac{1}{2}}\, dt = \int_a^b\|\gamma'(t)\|_{\metric}\, dt.
\]
We define the length of a piecewise $C^1$ path as the sum of the lengths of the constituting $C^1$  curves. 
For a connected smooth Riemannian Hilbert manifold $(\M,\metric)$ 
we defined the \emph{distance induced by $\metric$} between two points $p,q\in \M$, denoted by $d_\metric$, as
\[
d_{\metric}(p,q) = \inf\{l(\gamma)\mid \gamma \mbox{ is a piecewise $C^1$ path from $p$ to $q$}\}.
\]

With the definition of distance, we define metric partition on Riemannian Hilbert manifolds.

\begin{definition}\label{D: Metric Partition}
Let $\M$ be a smooth Riemannian Hilbert manifold. A \emph{metric partition } $\FH$ of $\M$ is a partition 
$\FH = \{\LH_u\mid u\in \M\}$ into closed
subsets $\FH_u$, called \emph{leaves}, such that the leaves are locally equidistant. 
That is, for each $u\in \M$, there exists $\varepsilon>0$ such that if $d(u,\LH)<\varepsilon$ for some $\LH\in \FH$, 
then there exists an open neighborhood  $\vizinhancaB\subset \LH_u$ of $u$ in $\LH_u$, such that
for any $v\in \vizinhancaB$ we have $d(v,\LH) = d(u,\LH)$.
\end{definition}

We recall that for a fixed point $u\in \M$ in a Hilbert manifold $\M$, by  \cite[Corollary 5.2 (2)]{Lang} there exists $\varepsilon>0$ 
such the exponential map $\exp_u$ is a diffeomorphism over the closed ball $B_\varepsilon (u)$ of radius $\varepsilon$ centered at $u$.

Given a metric partition $\FH$ on a Hilbert manifold $\M$, we define for a leaf point 
$\LH_u = \{u\}\in \FH$ and $0\leq \lambda \leq 1$ the \emph{homothetic transformation} $h_\lambda\colon B_\varepsilon (u)\to B_\varepsilon (u)$ 
given by $h_\lambda(\exp_u(\zeta)) = \exp_\varepsilon(\lambda \zeta)$. 
Then we have the following result, which is analogous to \cite[Section 6.2, Lemma 6.2]{Molino} for point leaves.

\begin{lemma}[Homothetic transformation lemma]\label{L: Homothetic transformation}
Given $(\M,\FH)$ a metric partition on a Hilbert manifold, we fix $u\in M$ such that $\LH_u = \{u\}$. We fix sufficiently small $\varepsilon>0$, in particular such that the exponential map is a diffeomorphism over the tubular neighborhood 
$B_\varepsilon (u)\subset \M$. Then $B_\varepsilon(u)$ is saturated by $\FH$, i.e. if $\LH\cap B_\varepsilon (u)\neq \varnothing$,
 then $\LH\subset B_\varepsilon(u)$.  Furthermore,  for $0\leq \lambda\leq 1$ the map $h_\lambda$ over $B_\varepsilon(u)$ sends leaves of $\FH$ to leaves of $\FH$.
\end{lemma}

\begin{proof}
Fix $v=\expo_u(\zeta)\in B_\varepsilon(u)\cap \LH$. Then since $\LH_u = \{u\}$, 
we have $d(v,\LH_u)=d(v,u)\leq \varepsilon$. 
We may assume that $\varepsilon$ is small enough so that there exists an open neighborhood 
$\vizinhancaB\subset \LH_v$ of $v$ such that for any $v'\in \vizinhancaB$ we have $d(v',u) = d(v',\LH_u) = d(v,u)\leq \varepsilon$.

Now, for $h_\lambda(v)$ and $v$ there exists open subsets $\Omega\subset 
\LH_v$ and $\Omega_\lambda\subset \LH_{h_\lambda(v)}$ such that for any point $v'\in \Omega_\lambda$ 
the distance $d(v',\Omega)$ is constant. Denoting by $d(v,u) = \rho$, we have $d(v,h_\lambda(v))= \rho(1-\lambda)$. 
From this it follows that $d(v',\Omega) = \rho(1-\lambda)$ for  arbitrary $v'\in \Omega_\lambda$.

Now fix $v'\in\Omega_\lambda$, and denote by $w\in \Omega$ the closest point to $v'$ in $\Omega$. Now observe that  $v'\in \partial B_{\rho\lambda}(u)$, and $w\in  \partial B_{\rho}(u)$. Writing $v'= \expo_u(\zeta')$, the distance $\rho(1-\lambda)$ between $v'$ and $\partial B_{\rho\lambda}(u)$ is only realized by a geodesic of the form $\expo_u(t\zeta')$. Thus, we conclude that $v'= h_\lambda(w)$. This implies that $\Omega_\lambda = h_{\lambda}(\Omega)$. 
\end{proof}

In the following result 
 we stress that the set $\basicspace$ admits a  ``cone structure'' which will   suffice for our purposes in this paper.

\begin{lemma}
\label{lemma-condition-basicspace-geodesic}
Let $(\M,\FH)$ be a metric partition on a Riemannian Hilbert manifold. 
Then the set $\basicspace = \{u\in \M\mid \LH_u = \{u\}\}$  satisfies the following property: 
for $u\in \basicspace$ there exists a $B_{\epsilon}(u)$ such that
if $v\in \basicspace\cap B_{\epsilon}(u)$ and 
$v=\exp_{u}(\zeta)$ then the geodesic $\exp_{u}(t\zeta)$ is contained
in $\basicspace$ for $|t|<1$.  
\end{lemma}

\begin{proof}
Fix $u\in \basicspace$. Observe that by definition, for a sufficiently small $\varepsilon>0$, for every 
$0<r\leqslant\varepsilon$ we have that if the distance from $v$ to $u$ is equal to $r$, 
then the distance from $w\in \LH_v$ to $u$ is $r$. That is if $v\in \partial B_r(u)$, then $\LH_v\subset\partial B_r(u)$.  

For $\lambda\in(0,1]$ we consider the homothety-transformation $h_\lambda\colon B_\varepsilon(u)\to B_\varepsilon(u)$ 
defined as $h_\lambda(\expo_u(\zeta)) = \expo_u(\lambda\zeta)$. 
By Lemma \ref{L: Homothetic transformation}, the map $h_\lambda$ preserves the leaves of the foliation $\FH|_{B_\varepsilon(u)}$.

Let us now see that  if $v\in B_\varepsilon(u)\cap\basicspace$, 
then $h_{\lambda}(v)$ is in a $0$ dimensional leaf. Suppose this is not the case. 
Then there exists a non-constant curve $\gamma\colon I\to \LH_{h_\lambda(v)}$. 
However, since $h_{\lambda}^{-1}$ also maps leaves to leaves, then we get a nonconstant curve $\bar{\gamma}\colon I \to \LH_v$, which is a contradiction. This implies that if $v=\expo_u(\zeta)\in \basicspace$ then the geodesic $\expo_u(t\zeta)$ is contained in $\basicspace$. 
 \end{proof}

\subsection{A version of a principle of symmetric criticality on Hilbert manifolds } 
\label{S: Proof of Theorem A}

We present the proof of the main result of this section, following the proof of \cite[Theorem in Section 2]{palais1979}

First note that 
Lemma \ref{lemma-condition-basicspace-geodesic} allow us to  differentiate
 the operator $\JH$ along geodesics contained in the set $\basicspace$ 
and hence we can define the concept of critical points of $\JH|_{\basicspace}$.

\begin{definition}
 $u\in \basicspace$ is a critical point of $\JH|_{\basicspace}\to\mathbb{R}$  if 
\[
\frac{d}{d t} \JH\circ\gamma(t)|_{t=0}= \mathrm{d} \JH_{u}(\gamma'(0))=0
\]
for all geodesic 
$t\to \gamma(t)=\exp_{u}(t\zeta)\subset\basicspace$, where $|t|<1$. 
\end{definition}

\begin{theorem}
\label{theorem-CorrosTheorem-A}
Let $\M$  be a Riemannian Hilbert manifold, with
a  metric partition (see Definition \ref{D: Metric Partition}) $\FH$. Let
$\JH \colon \M\to\mathbb{R} $ be a differential operator constant along
the leaves of $\FH$ and denote by $\basicspace$
the set of $0$ dimensional leaves of $\FH$.
If $b\in \basicspace$ is a critical point of $\JH|_{\basicspace}\colon \basicspace\to\mathbb{R}$
then $b$ is a critical point of $\JH$.
\end{theorem}

\begin{proof}[Proof of  Theorem \ref{theorem-CorrosTheorem-A}]
We start by recalling some observations from  \cite[p.~23]{palais1979}. For  the given $C^1$ 
function  $\JH \colon \M\to \R$ there is a unique associated gradient vector field $\nabla \JH\colon \M\to T\M$ on 
$\M$ (see \cite[Proposition 6.1]{Lang}) which is related to the differential 
$\mathrm{d}\JH$ of $\JH$ by the following identity for any $u\in \M$:
\[
\mathrm{d} \JH_{u}(V) = \langle V,\nabla \JH(u)\rangle_u,\quad \mbox{for all } V\in T_u \M.
\] 
From this, it follows that $\nabla \JH(u) = 0$ if and only if $u$ is a critical point of $\JH$.

Fix now  a critical point $u\in \basicspace$ of $\JH|_{\basicspace}$  and assume that $\nabla \JH(u)\neq 0$. 

Set $\zeta=\nabla \JH(u)\in T_u\M$, and consider $t_0>0$ small enough so that for $t_0\geq t \geq 0$ 
the curve $\gamma_0(t)=\expo_u(t\zeta)$ is well defined. Moreover we take $t_0$ small enough \textcolor{blue}{so that on} $B_{t_0}(\bar{0})\subset T_u\M$ 
the exponential map $\expo_u$ is a diffeomorphism onto $B_{t_0}(u)\subset \M$. 
We now assume that for $v=\gamma_0(t_0)$ the leaf $\LH_v$ is at least $1$-dimensional. Then, there exists in 
$B_{t_0}(u)\cap \LH_v$ a  non-constant  curve $\alpha\colon (-\varepsilon,\varepsilon)\to \LH_v\cap B_{t_0}(u)$ such that $\alpha(0) = v$. 
We observe that by construction there exists a curve $\zeta\colon (-\varepsilon,\varepsilon)\to T_u\M$  such that $\zeta(0) = \zeta$ 
and the curve $\alpha$ is of the form $\alpha(s) = \expo_u(t_0\zeta(s))$. Recall that for $0<\lambda\leq 1$ 
the map $h_\lambda\colon B_{t_0}(u)\to B_{t_0}(u)$ given by $h_\lambda(\expo_u(\zeta))=\expo_u(\lambda\zeta)$ preserves the foliation $\FH$. 
 We also stress that the leaf $\LH_{v}$ stays equidistant from the leaf $\LH_{u}=\{u\}$.  
Thus we can choose $\zeta(s)$ so that $\|t_0\zeta(s)\|=\|t_0\zeta\|$.
Then for each $s\in (-\varepsilon,\varepsilon)$ we get a geodesic $\gamma_s\colon [0,t_0]\to \M$ 
of $\M$ given by $\gamma_s(t)= \expo_u(t\zeta(s))$. Observe that $\gamma_s(0) = u$, $\gamma_{s}'(0) = \zeta(s)$, 
$\|\zeta(s)\|=\|\zeta\|$,
and $\gamma_s(t)\in \LH_{\gamma_0(t)}$. 
In particular we have $\JH\circ \gamma_s(t) = \JH\circ\gamma_0(t)$ for each $s$ and $t$. Therefore we get
\begin{linenomath}
\begin{align*}
    \langle \nabla \JH(u), \gamma_s'(0) \rangle &= \frac{d}{dt}(\JH\circ\gamma_s(t))|_{t=0}\\
    &=\frac{d}{dt}(\JH\circ\gamma_0(t))|_{t=0}\\
    &=\langle \nabla \JH(u),\zeta\rangle\\
    &=\|\nabla \JH(u)\|^2.
\end{align*}
\end{linenomath}
 Thus we get that $\|\nabla \JH(u)\|\|\zeta(s)\|\cos(\theta_s)=\|\nabla \JH(u)\|^2$, 
where $\theta_s$ is the angle from $\zeta(s)$ to $\zeta=\nabla \JH (u)$. 
Given that $\|\zeta(s)\|=\|\nabla \JH(u)\|\neq 0$ we conclude that $\theta_s=0$.
That is we have 
\[
\gamma_s'(0) = \nabla \JH (u).
\]
This implies that the curve $\alpha(s)= \expo_u(t_0\gamma_s'(0))$ is constant, which is a contradiction. Thus we conclude that $\widetilde{L}_v$ is $0$-dimensional, 
 and hence, by applying Lemma \ref{lemma-condition-basicspace-geodesic},    we have that $\exp_u(t\nabla \JH(u))\in \basicspace$. 
Now by definition we have $\|\nabla \JH(u)\|^2=\mathrm{d} \JH_{u}(\nabla \JH(u))=\frac{d}{d t} (\JH(\exp_u(t\nabla \JH(u))))|_{t=0}=0$, 
which is a contradiction. 
Therefore $\nabla \JH(u)=0$, i.e., $u$ is a critical point of $\JH$.
\end{proof}

\begin{remark}
Consider the inner product $\langle \cdot,\cdot\rangle_{\perp}$ defined in Equation \eqref{eq1-lemma-new-theorem-Diego-general-version-Palais-Hilbert}
 and let  $W^{1,2}_{0,\perp}(U_r)$ be the closure of $C^{\infty}_{c}(U_r)$ (i.e, \emph{the completion})  
with respect to  the associated norm $\|\cdot\|_{\perp}$. 
As observed in Lemma \ref{lemma1-new-theorem-Diego-general-version-Palais-Hilbert}, there exists a linear isometric action
$W^{1,2}_{0,\perp}(U_r)\times  \mathcal{G}^{\ell}\to W^{1,2}_{0,\perp}(U_r)$. 
If this action is at least continuous, we would have that the partition by orbits would be a metric partition, because $\mathcal{G}$ is compact (the fiber of the source map are compact and $B$ is compact).
Thus, it would be possible to use the above theorem to give an alternative proof of Theorem \ref{theorem-Diego-general-version-Palais-Hilbert}.
To avoid additional technicalities, we do not pursue this direction here and leave the verification of continuity as an open question.
We believe this additional line of argument can be addressed in future investigations, 
particularly in the context of groupoid actions on Banach and Hilbert spaces.
\end{remark}

\section{Appendix: Linear Holonomy groupoid and average map}
\label{section-appendix}

In this section, we address some technical lemmas used in Subsection \ref{subsection-avarage-first-intuition} 
and in the proof of Proposition \ref{proposition-new-MALEX--Av-l-orbitlike}.

\subsection{Properties of the lifted foliation $\widetilde{\mathcal{F}}$}
 \label{MALEXsec-Sasaki metrics-SRF-holonomy-groupoid}

Let us start by stressing a few properties of the foliation constructed in Example  \ref{example-holonomy-groupoid-lifted-foliation}

\begin{lemma}%
\label{Properties-Lifted-Foliation}
\quad
\begin{enumerate}
\item[(a)] $\widetilde{\mathcal{F}}$ is a regular foliation. In addition for each $\xi \in \mathbb{O}(E)$
\begin{equation*}
T_\xi \widetilde{L}_{\xi} = \widetilde{\mathcal{T}}_{\xi} \oplus T_\xi G_{p}^{0} (\xi), \hspace{1cm} T_\xi \mathbb{O}(E)_p 
= T_\xi \widetilde{S}_\xi \oplus T_\xi G_{p}^{0} (\xi),
\end{equation*}
where $\widetilde{\pi} (\xi) = p$ and $\widetilde{S}_\xi$ is a slice at $\xi$.

\vspace{0.25\baselineskip}

\item[(b)] Given any two leaves of $\widetilde{\mathcal{F}}$, there is a $T \in \mathbb{O}(k)$ 
such that $\mathrm{m}_T$ (the right action of $T$) is a diffeomorphism between them.

\vspace{0.25\baselineskip}

\item[(c)] $\widetilde{\mathcal{F}}$ has trivial holonomy.

\vspace{0.25\baselineskip}

\item[(d)] There exists a fiberwise metric $\widetilde{\mathrm{g}}$ on $T \widetilde{\mathcal{F}}$ such that:
\begin{enumerate}
\item[(d.1)]  $\tilde{\pi}_{\widetilde{L}}\colon \widetilde{L}\to B$ 
a Riemannian submersion, so that the fibers $\{G_{p}=\tilde{\pi}_{\widetilde{L}}^{-1}(p)\}_{p\in B}$ are isometric.
\item[(d.2)] $\mathrm{m}_T$ restricted to any leaf is an isometry, for all $T \in \mathbb{O}(k)$.
\end{enumerate}
\end{enumerate}
\end{lemma}

\begin{proof} \quad
\begin{enumerate}
\item[(a)] Since the horizontal distribution $\mathcal{T}$ is tangent to the foliation and the foliation is transverse to the fibers of $\mathbb{O}(E)$, we have $\widetilde{\mathcal{T}}_{\xi} \subset T_\xi \widetilde{L}_\xi \not\subset T_\xi \mathbb{O}(E)_p$. When we intersect $T_\xi \mathbb{O}(E) = \widetilde{\mathcal{T}}_{\xi} \oplus T_\xi \mathbb{O}(E)_p$ with $T_\xi \widetilde{L}_\xi$, we obtain
\begin{equation*}
T_\xi \widetilde{L}_\xi = \widetilde{\mathcal{T}}_{\xi} \oplus T_\xi (\mathbb{O}(E)_p \cap \widetilde{L}_\xi) 
= \widetilde{\mathcal{T}}_{\xi} \oplus T_\xi G_{p}^{0} (\xi).
\end{equation*}

\noindent Then the regularity of the foliation $\widetilde{\mathcal{F}}$ follows from the fact that the action of $G_{p}^{0}$ on $\mathbb{O}(E)_p$ is free.

\noindent Moreover, the transversality of the foliation to the fibers of $\mathbb{O}(E)$ implies that $T_\xi \widetilde{S}_\xi \subset T_\xi \mathbb{O}(E)_p \not\subset T_\xi \widetilde{L}_\xi$. Intersecting  $T_\xi \mathbb{O}(E) = T_\xi \widetilde{S}_\xi \oplus T_\xi \widetilde{L}_\xi$ 
with $T_\xi \mathbb{O}(E)_p$ imply
\begin{equation*}
T_\xi \mathbb{O}(E)_p = T_\xi \widetilde{S}_\xi \oplus T_\xi (\widetilde{L}_\xi \cap \mathbb{O}(E)_p) 
= T_\xi \widetilde{S}_\xi \oplus T_\xi G_{p}^{0} (\xi).
\end{equation*}

\vspace{0.25\baselineskip}

\item[(b)] First we note that $\mathrm{m}_T (\widetilde{\varphi}^\ell \circ \xi) = \widetilde{\varphi}^\ell \circ \mathrm{m}_T (\xi)$ 
for all linearized flow $\widetilde{\varphi}^\ell$ and $T \in \mathbb{O}(k)$, which implies that $\mathrm{m}_T$ is a foliated diffeomorphism.

\noindent Finally, each leaf of $\mathcal{F}^\ell$ intersects all the fibers of $E$ (by construction), and $\widetilde{\mathcal{F}}$ 
inherit this property. Then given $\xi, \zeta \in \mathbb{O}(E)$ there is $\hat{\xi} \in \mathbb{O}(E)_{\pi(\xi)} \cap L_\zeta$ 
and $T \in \mathbb{O}(k)$ such that $\hat{\xi} = \mathrm{m}_T (\xi)$. 
Since $\mathrm{m}_T$ is a foliated diffeomorphism, $\mathrm{m}_T$ restricts to a diffeomorphism between $L_\xi$ and $L_{\hat{\xi}} = L_\zeta$.

\vspace{0.25\baselineskip}

\item[(c)] We consider $\varphi_{\tilde{\alpha}}\colon \widetilde{S}_{\xi_0} \subset \mathbb{O} (E)_{\pi(\xi_0)} \to \widetilde{S}_{\xi_0}$ 
to be a holonomy map along a loop $\alpha$. Since the right action of $\mathbb{O}(k)$ in $\mathbb{O}(E)$ 
is transitive and foliated, for a given $\xi \in \widetilde{S}_{\xi_0}$ there is $T_\xi \in \mathbb{O}(k)$ such that $\xi = \xi_0 \cdot T_\xi$ and
\begin{equation*}
\varphi_{\tilde{\alpha}} (\xi_0 \cdot T_\xi) = \varphi_{\tilde{\alpha}} (\xi_0) \cdot T_\xi = \xi_0 \cdot T_\xi
\end{equation*}
Hence $\varphi_{\tilde{\alpha}}$ must be the identity map.

\vspace{0.25\baselineskip}

\item[(d)] 
We start by defining the metric $\tilde{\metric}|_{\tilde{\pi}_{\widetilde{L}}^{-1}}=\langle \cdot, \cdot\rangle$ on the fibers. 
We consider a cover  $\{ \vizinhancaB_{\alpha}\}$  of $B$ and 
$\{\widetilde{\vizinhancaB}_{\alpha}=\tilde{\pi}_{\widetilde{L}}^{-1}(\vizinhancaB_{\alpha})\}$  a cover of $\widetilde{L}$.  
Set $K=G_{p_{0}}$ and consider a bi-invariant metric $Q$ on $K$.   We recall that $\mathrm{Hol}_{p_0}\subset G_{p_0}$. 
The idea is to define   parametrizations 
$\varphi_{\alpha}\colon K\times \vizinhancaB_{\alpha} \to \widetilde{\vizinhancaB}_{\alpha}$ so that
$d(\varphi_{\beta}^{-1}\circ\varphi_{\alpha})\vec{\Theta}= \vec{\Theta}\cdot h_{\alpha, \beta}$ for $h_{\alpha,\beta}\in \mathrm{Hol}_{p_0}\subset G_{p_0}$ 
and for each $\vec{\Theta}$  right invariant vector field on $K$.  This implies that 
\[
Q\big(d(\varphi_{\beta}^{-1}\circ\varphi_{\alpha})\vec{\Theta}_{i}, d(\varphi_{\beta}^{-1}\circ\varphi_{\alpha})\vec{\Theta}_{j}\big)
= Q\big(\vec{\Theta}_{i}\cdot h_{\alpha,\beta} ,\vec{\Theta}_{j}\cdot h_{\alpha,\beta} \big)
= Q\big(\vec{\Theta}_{i},\vec{\Theta}_{j}\big),
\]
what guarantee  that the metric $\langle \cdot,\cdot \rangle=\varphi_{\alpha}^{*}Q\times p$ does not dependent on $\varphi_{\alpha}$, i.e.  
 it is well defined. Parametrizations $\varphi_{\alpha}$ can be constructed using a $\mathcal{F}$-parallel local frame $\xi_{\vizinhancaB_{\alpha}}$. More precisely
$\varphi_{\alpha}(k,p)=\rho_{\alpha}(k)\cdot \xi_{\vizinhancaB_{\alpha}}$ for  an isomorphism $K\to G_{p_{\alpha}}$ constructed also using a 
$\mathcal{F}$-parallel frame, along a geodesic joining $p_0$ to a chosen point $p_{\alpha}\in \vizinhancaB$. 
Now define the metric on $\widetilde{\mathcal{T}}$ so that 
$\mathrm{d}_\xi \pi\colon  (\widetilde{\mathcal{T}}_{\xi}, \pi^\ast g_p) \to (T_p B, g_p)$ is an isometry and define 
that the distribution $\widetilde{\mathcal{T}}$ to be  orthogonal to the fibers $\{\tilde{\pi}_{\widetilde{L}}^{-1}(p)\}_{p\in B}$. 
The rest of the item follows from the two facts:
\begin{itemize}
\item $\mathrm{m}_T$ is a isometry between $G_{p}^{0}(\xi)$ and $G_{p}^{0} (\xi \circ T)$, 
since $\widetilde{\mu}_{\xi \circ T} = \mathrm{m}_T \circ \widetilde{\mu}_\xi$;
\item $\mathrm{m}_T$ is an isometry between $\widetilde{\mathcal{T}}_\xi$ and this $\widetilde{\mathcal{T}}_{\xi \circ T}$, 
since $\widetilde{\mathcal{T}}$ is a horizontal distribution on $\mathbb{O}(k)$-principal bundle 
(i.e. $\mathrm{d\hspace{0.1ex}} \mathrm{m}_T (\widetilde{\mathcal{T}}) = \widetilde{\mathcal{T}}$) and $\pi \circ \mathrm{m}_T = \pi$.
\end{itemize}
\end{enumerate}
\end{proof}

\begin{remark} \label{action-leaves}
Item (b) of the previous lemma implies that the induced
proper $\mathbb{O}(k)$-action on $T \mathbb{O}(E)$ restricts to $T \widetilde{\mathcal{F}}$ and the fiberwise metric $\widetilde{\mathrm{g}}$ in item (d) assures this action is isometric. 
\end{remark}

\vspace{0.5\baselineskip}

As discussed in Example \ref{example-principal-groupoid}  the metric on $\tilde{s}^{-1}(\xi)$ descends to a metric on $s^{-1}(p)$ via pushforward by $\rho |_{\tilde{s}^{-1}(\xi)}$, which becomes an isometry. Thus we obtain an isometry $\psi_\xi\colon s^{-1} (p)\to \tilde{L}_\xi$, which makes the following diagram commute: 

\begin{center}
\vspace{\baselineskip}
\hspace{0.5cm}
\begin{tikzcd}[column sep={1.5cm}, row sep={1.5cm}]
\tilde{s}^{-1}(\xi) \arrow{r}{\tilde{t}} \arrow[swap]{d}{\rho\hspace{0.25ex}} & \tilde{L}_\xi \\
s^{-1} (p) \arrow[swap]{ur}{\psi_\xi} &
\end{tikzcd}.
\vspace{\baselineskip}
\end{center}

In Lemma \ref{s-fiber-isometric-lifted-leaf} we prove that it is possible to locally collect all the isometries $\psi_\xi:= \tilde{t} |_{\tilde{s}^{-1}(\xi)} \circ (\rho |_{\tilde{s}^{-1}(\xi)})^{-1}$ in a unique submersion $\Psi$.

\begin{lemma}\label{s-fiber-isometric-lifted-leaf}
Given a leaf $\widetilde{L} \in \widetilde{\mathcal{F}}$ and $\xi \in \Gamma(\mathbb{O}(E)_U)$ a local orthonormal frame into $\widetilde{L}$, there exists  $\Psi\colon s^{-1}(U) \to \widetilde{L}$ such that:
\begin{enumerate}
\item[(a)] $\tilde{\pi} \circ \Psi = t$ and $\Psi \circ \textbf{1} = \xi$;

\vspace{0.25\baselineskip}

\item[(b)] $\Psi$ is a submersion and $\Psi|_{s^{-1}(p)} = \psi_{\xi(p)}$ is an isometry, for all $p \in U$;

\vspace{0.25\baselineskip}

\item[(c)] (invariance by linearized flows) $\mu (\psi_{\xi(p)}^{-1} (\widetilde{\varphi}^{\ell}(\xi(p))), u_p) = \varphi^{\ell}(u_p)$, for all linearized flow $\varphi^{\ell}\colon E_U \to E$ and $u_p \in E_U$.
\end{enumerate}
\end{lemma}

\begin{proof}
We define the map $\Psi\colon s^{-1}(U) \to \widetilde{L}$ by setting
\begin{equation*}
\Psi (\varphi_{\tilde{\alpha}} \cdot \mathbb{O}(k)) = \tilde{t} (\varphi_{\tilde{\alpha}_\xi}) = \tilde{\alpha}_\xi (1),
\end{equation*}
where $\varphi_{\tilde{\alpha}} \cdot \mathbb{O}(k) \in s^{-1}(U) \subset \mathcal{G}_1$ denotes the $\mathbb{O}(k)$-class of the holonomy map $\varphi_{\tilde{\alpha}}$, and $\varphi_{\tilde{\alpha}_\xi}$ is the unique representative of $\varphi_{\tilde{\alpha}} \cdot \mathbb{O}(k)$ such that $\tilde{\alpha}_\xi (0) = \xi(\alpha(0))$ with $\alpha =\pi \circ \tilde{\alpha}$.

\begin{enumerate}
\item[(a)] Follows directly from this definition and from the fact that $\tilde{\pi} \circ \tilde{t} = t \circ \rho$.

\vspace{0.25\baselineskip}

\item[(b)] We observe that $\tilde{s}^{-1}(\xi(U)) \subset \widetilde{\mathcal{G}}_1$ is a submanifold since $\tilde{s}$ is a submserion and $\xi(U) \subset \widetilde{L}$ is a submanifold. Also $\rho (\tilde{s}^{-1}(\xi(U))) = s^{-1}(U)$ and therefore $\rho_{\xi}:= \rho |_{\tilde{s}^{-1}(\xi(U))}$ is a submersion onto $s^{-1}(U)$.

\noindent Furthermore $\rho_{\xi}$ is injective and therefore a diffeomorphism: Indeed if $\varphi_{\tilde{\alpha}} \cdot \mathbb{O}(k) = \varphi_{\tilde{\beta}} \cdot \mathbb{O}(k) \in \tilde{s}^{-1}(\xi(U))$, then there exists a $T \in \mathbb{O}(k)$ such that $\tilde{\alpha} \cdot T = \tilde{\beta}$ and consequently $\tilde{\alpha}(0) = \xi_{\alpha(0)} = \xi_{\beta(0)} = \tilde{\beta}(0)$. Thus $T = \mathrm{Id}$.

\noindent Then we conclude that $\Psi$ is a submersion since $\Psi = \tilde{t} \circ \rho_{\xi}^{-1}$. This identity also implies that $\Psi|_{s^{-1}(p)} = \psi_{\xi(p)}$ for all $p \in U$, which is an isometry.

\vspace{0.25\baselineskip} 

\item[(c)] Consider $\varphi_{\tilde{\alpha}_\xi}$ the unique representative of $\psi_{x}^{-1} (\widetilde{\varphi}^{\ell}(\xi_p))$ such that $\tilde{\alpha}_\xi (0) = \xi(\alpha(0))$. Then \begin{equation*}
\tilde{\alpha}_\xi (1) = \psi_{\xi(p)} (\varphi_{\tilde{\alpha}_\xi} \cdot \mathbb{O}(k)) = \widetilde{\varphi}^{\ell}(\xi_p) = \varphi^{\ell} \circ \xi_p
\end{equation*}
which implies that
\begin{equation*}
\mu(\psi_{x}^{-1}(\widetilde{\varphi}^{\ell}(\xi_p)), u_p) = (\tilde{\alpha}_\xi (1) \circ \tilde{\alpha}_\xi (0)^{-1}) (u_p) = \varphi^{\ell} (u_p).
\end{equation*}
\end{enumerate}
\end{proof}

\begin{remark} 
We observed in Remark \ref{action-leaves} that the $\mathbb{O}(k)$-action in $T \widetilde{\mathcal{F}}$ is by isometries. Since the target $\tilde{t}$ map is 
$\mathbb{O}(k)$-invariant, the induced fiberwise metric on the fibers of the source $\tilde{s}$ is invariant with respect to the $\mathbb{O}(k)$-action.
\end{remark}

\begin{remark}
The construction of the metric on $\tilde{s}$-fibers can be done in the case of a generic groupoid $\widetilde{\mathcal{G}}$, for which the corresponding Lie algebroid $A = \tilde{\mathbf{1}}^{\ast} \mathrm{Ker\hspace{0.1ex}} (\mathrm{d} \tilde{s}) \to \widetilde{\mathcal{G}}_{0}$ is Riemannian (i.e. provided with a fiberwise metric). According to Boucetta in \cite{Bou11}, the target $\tilde{t}\colon  s^{-1}(p) \to \widetilde{\mathcal{G}}(p)$ restricted to the orthogonal complement of the Lie algebra $\tilde{\mathfrak{g}}_p$ of the isotropy group $\widetilde{\mathcal{G}}_p$ is an isomorphism. It is possible to push forward the metric of $\tilde{\mathfrak{g}}_{p}^{\perp}$ to the tangent space of the orbits $T_p \widetilde{\mathcal{G}}(p)$, which turns out to  be $\tilde{t} |_{s^{-1}(p)}$ into a Riemanian submersion onto the orbit $\widetilde{\mathcal{G}}(p)$. Also, we observe that the fiberwise metric of the fibers of the source $\tilde{s}$ map   
turns out  to be the right multiplication into an isometry.
\end{remark}

\vspace{0.5\baselineskip}

\subsection{Average operator on $\mathcal{G}^{\ell}$} \label{sec:average}

Here we consider the average operator presented in Definition \ref{definition-AV-general}.

\begin{lemma} [Smoothness of the Average]  \label{Av-bem-definido}
The function $\mathcal{V}$ is constant and equal to the volume of any leaf of $\widetilde{\mathcal{F}}$ and $\mathrm{Av}(f)$ is a $\mathcal{G}^\ell$-basic smooth function, for every $f \in C^{\infty}_{c} (E^\delta)$.
\end{lemma}

\begin{proof}
Since the $s$-fibers are isometric to any leaf of $\widetilde{\mathcal{F}}$ (see Lemma \ref{s-fiber-isometric-lifted-leaf}), they are compact and hence $\mathcal{V}(p)$ is well defined (i.e., finite). Furthermore, for each $g \in \mathcal{G}_1$, the right multiplication $\mathrm{R}_g\colon s^{-1}(t(g)) \to s^{-1} (s(g))$ is an isometry, which implies that $\mathcal{V}(p)$ is constant.

Furthermore, the average is basic. Indeed, given $(g, u_p) \in \mathcal{G}_1 \times_B E^\delta$, set $\mu(g, u_p) = v_q$ (which implies $t(g) = \pi (v_q) = q$). Since the right multiplication by $g$, $\mathrm{R}_g$, is an isometry and $\mu_{v_q} = \mu_{u_p} \circ \mathrm{R}_g$, it follows that
\begin{equation*}
\mathrm{Av}(f) \left( \mu(g, u_p) \right) = \frac{1}{\mathcal{V}} \int_{\mathrm{s}^{-1}(q)} (f \circ \mu_{v_q}) \, \nu_q = \frac{1}{\mathcal{V}} \int_{\mathrm{s}^{-1}(p)} (f \circ \mu_{u_p}) \, \nu_p = \mathrm{Av}(f) (u_p).
\end{equation*}

\vspace{0.25\baselineskip}

It remains to prove that the average is smooth. We prove this in an open neighborhood of each $u_p \in E^{\delta}$. Consider $\Psi\colon s^{-1}(U) \to \widetilde{L}$ the submersion described in Lemma \ref{s-fiber-isometric-lifted-leaf} with $p \in U$, and denote $\psi_{q} = \Psi |_{s^{-1}(q)}$ for each $q \in U$. The function $\tilde{f}\colon \widetilde{L} \times E_U \to \mathbb{R}$ defined by
\begin{equation*}
\tilde{f} (\xi, v_q) = \left( f \circ \mu_{v_q} \circ \psi_{q}^{-1} \right) (\xi)
\end{equation*}
is smooth by construction. Since $\tilde{f} \circ (\Psi \times_{U} \mathrm{Id}) = f \circ \mu$ for $\Psi \times_U \mathrm{Id}\colon s^{-1}(U) \times_U E_U \to E_U$ defined as $(\Psi \times_{U} \mathrm{Id}) (g, v_q) = (\Psi(g), v_q)$, which is a submersion, then we have
\begin{equation*}
\mathrm{Av}(f)(v_q) = \frac{1}{\mathcal{V}} \, \int_{s^{-1}(q)} (f \circ \mu_{v_q}) \, \nu_{q} = \frac{1}{\mathcal{V}} \int_{\widetilde{L}} \tilde{f} (\cdot, v_q) \, \tilde{\nu},
\end{equation*}
for each $v_q \in E^{\delta}_{U}$ (given that $\psi_q$ is an isometry). This proves the average is smooth in $E^{\delta}_{U}$.
\end{proof}



\end{document}